\documentclass[12pt, a4paper,reqno]{amsart}

\usepackage{amssymb}
\usepackage{amsthm}
\usepackage{amsmath} 
\usepackage{graphicx}
\usepackage{latexsym,amsfonts}
\usepackage{xcolor}
\usepackage{enumitem}
\usepackage{float}
\usepackage{caption}
\usepackage{subcaption}
\usepackage[backend=biber,
style=alphabetic,
sorting=nty]{biblatex}
\renewbibmacro{in:}{}
\DeclareFieldFormat[article, inbook]{title}{#1}
\usepackage{anyfontsize}
\usepackage{t1enc}

\newtheorem{thm}{Theorem}[section]
\newtheorem{cor}[thm]{Corollary}
\newtheorem{lem}[thm]{Lemma}
\newtheorem{prop}[thm]{Proposition}
\theoremstyle{remark}
\newtheorem{rem}[thm]{Remark}
\theoremstyle{definition}
\newtheorem{defs}[thm]{Definition}
\newtheorem{conj}[thm]{Conjecture}

\theoremstyle{plain} 
\newcommand{\thistheoremname}{}
\newtheorem{genericthm}[thm]{\thistheoremname}

  \newtheorem*{genericthm*}{\thistheoremname}
\newenvironment{namedthm*}[1]
  {\renewcommand{\thistheoremname}{#1}%
   \begin{genericthm*}}
  {\end{genericthm*}}

\begin{document}

\begin{titlepage} 
	\newcommand{\HRule}{\rule{\linewidth}{0.5mm}} 
	
	\center 
	
	
	\textsc{\LARGE University of Zurich}\\[1.5cm] 
	
	\textsc{\Large Institute of Mathematics}\\[0.5cm] 
	
	\textsc{\large Master's Thesis}\\[0.5cm] 

	
	\HRule\\[0.4cm]
	
	{\huge\bfseries On the Structure of Foliations on \\ \vspace{0.8cm} Dilation Surfaces}\\[0.6cm] 

	\HRule\\[1.5cm]
	

		\begin{center}
			\large
			\textit{Author}\\
			Anna Sophie \textsc{Schmidhuber} 
		\end{center}

		\begin{center}
			\large
			\textit{Supervisor}\\
			Prof. Dr. Corinna \textsc{Ulcigrai} 
		\end{center}
	
	
	
	\vfill\vfill\vfill 
	
	{\large\today} 
	
	
	\vfill\vfill
	\includegraphics[width=0.3\textwidth]{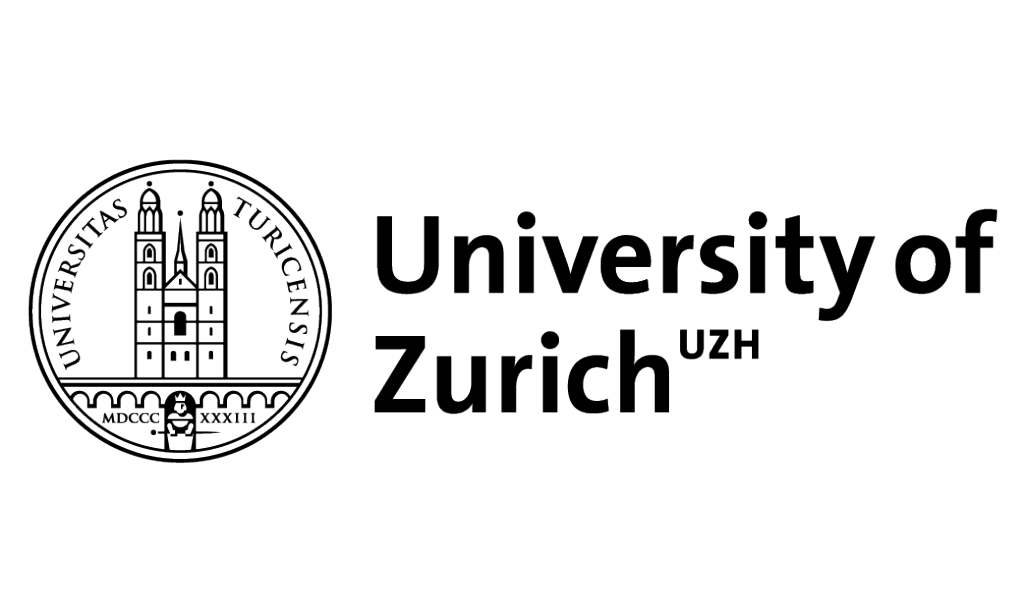}\\[1cm] 
	 
	
	\vfill 
\end{titlepage}
 \thispagestyle{empty}
\null\newpage

\begin{center}

\end{center}
\begin{center}

{\textsc{Abstract}}\\\vspace{0.5cm}
\end{center}

\thispagestyle{empty}
Dilation surfaces are geometric surfaces modelled after the complex plane whose structure group is generated by the groups of translations and dilations. For any dilation surface, for any direction $\theta$ in $S^1$, there exists a foliation on the surface called the directional foliation in direction $\theta$. In this thesis, we prove a structure theorem for the directional foliations on dilation surfaces using a decomposition theorem established by C.J. Gardiner in the 1980s in \cite{gardiner}. We show that given a directional foliation on any dilation surface, there exists a decomposition of the surface into finitely many subsurfaces on which the foliation structure is in one of four possible cases: completely periodic, Morse-Smale, minimal or Cantor-like. We further prove that in the last two cases, the first return map on a segment transversal to the foliation is semi-conjugated to a minimal interval exchange transformation. As a corollary, we obtain an analogous result for affine interval exchange transformations. Throughout the thesis, we accompany our results with an explicit example of a dilation surface called the Disco surface, building on the extensive study of this surface presented in \cite{cascades}. We analyze the directional foliations on the Disco surface that exhibit non-trivially recurrent behaviour and explain geometrically why these foliations accumulate to a Cantor set.

 \thispagestyle{empty}
 \null\newpage


\begin{center}

\end{center}
\begin{center}

{\textsc{Acknowledgements}}\\\vspace{0.5cm}
\end{center}
\thispagestyle{empty}
I would like to express my gratitude to Prof. Dr. Corinna Ulcigrai for introducing me to dilation surfaces, a fascinating topic which proved to be truly well-chosen for this thesis. I would further like to sincerely thank her as well as Guido Marinoni for all the long and helpful discussions throughout the past year. Their support and deep insights have been essential to my progress. Special thanks goes to Selim Ghazouani, whose publications on dilation surfaces have been my main reference during preparation. He further created the first version of the statement of the main theorem of this thesis and helpfully replied to all of my questions. Moreover, I would like to thank Charles Fougeron who kindly took the time to discuss his work on the Disco surface with me. 

Lastly, I wish to thank my family, my friends, and especially my partner. Without their unconditional and unwavering support this thesis would not have been possible. 
\thispagestyle{empty}
\null\newpage

{\fontsize{11.5}{12.5}\selectfont\thispagestyle{empty} \tableofcontents  \null\newpage}
\null\newpage


\section{Introduction}
\subsection{Translation surfaces and dilation surfaces} In recent years, \textit{dilation surfaces} have become a highly active topic in mathematical research. They are a natural generalization of much better known objects called \textit{translation surfaces}. A translation surface is a geometric surface that is locally homeomorphic to the complex plane and whose structure group is the group of translations. An equivalent, more hands-on definition is given using the \textit{polygonal representation} of a translation surface. Take a union of polygons in the complex plane with oriented edges that come in pairs: any two edges of a pair are required to have opposite orientation, to be parallel to each other and to have the same length. Edges of the same pair can then be "glued" together using a translation $z \rightarrow z + v$ for $v \in \mathbb{C}$. The surface we obtain is a translation surface, and any translation surface has such a polygonal representation. An example is given in Figure 1, edges that share the same color are identified with each other.

\begin{figure}[h]
\centering
\includegraphics[width=0.6\textwidth]{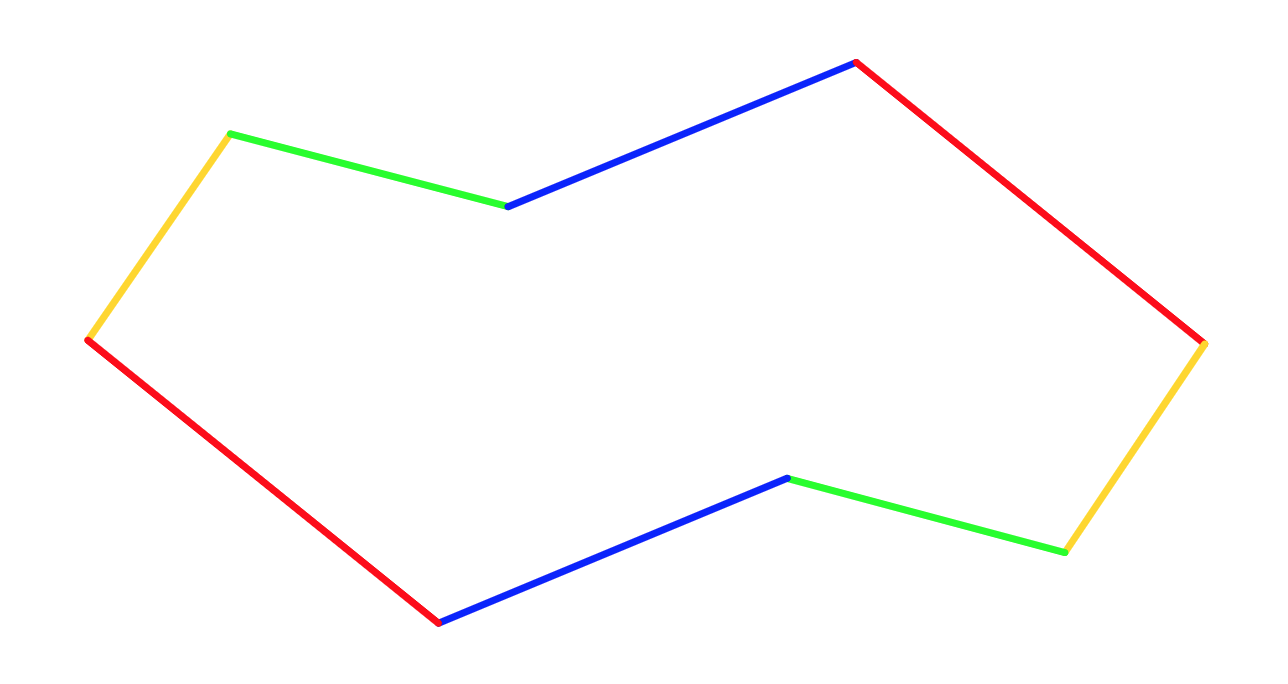}
\caption{A translation surface of genus two.}
\end{figure}

Translation surfaces have been widely studied over the past four decades. One of their key features is the fact that it is possible to define a \textit{directional foliation} on them. By a foliation on the surface or plane we mean, loosely speaking, a partition of the space into disjoint curves called \textit{leaves} that are locally parallel to each other. Pick $\theta \in S^1$ and consider the union of all straight lines in the complex plane in direction $\theta$. This defines a foliation on the complex plane that is invariant by translations and hence extends to a foliation on any translation surface via its polygonal representation. We can easily visualize the leaves of this foliation: take a point $p$ in the polygonal representation and start drawing a line in direction $\theta$. Whenever this line hits a point of an edge, continue drawing the line at the point on the other edge to which the first point is glued to, following the same direction. If we continue doing this both for direction $\theta$ and $-\theta$, we obtain the leaf of the foliation in direction $\theta$ going through $p$.

\begin{figure}[h]
\centering
\includegraphics[width=0.6\textwidth]{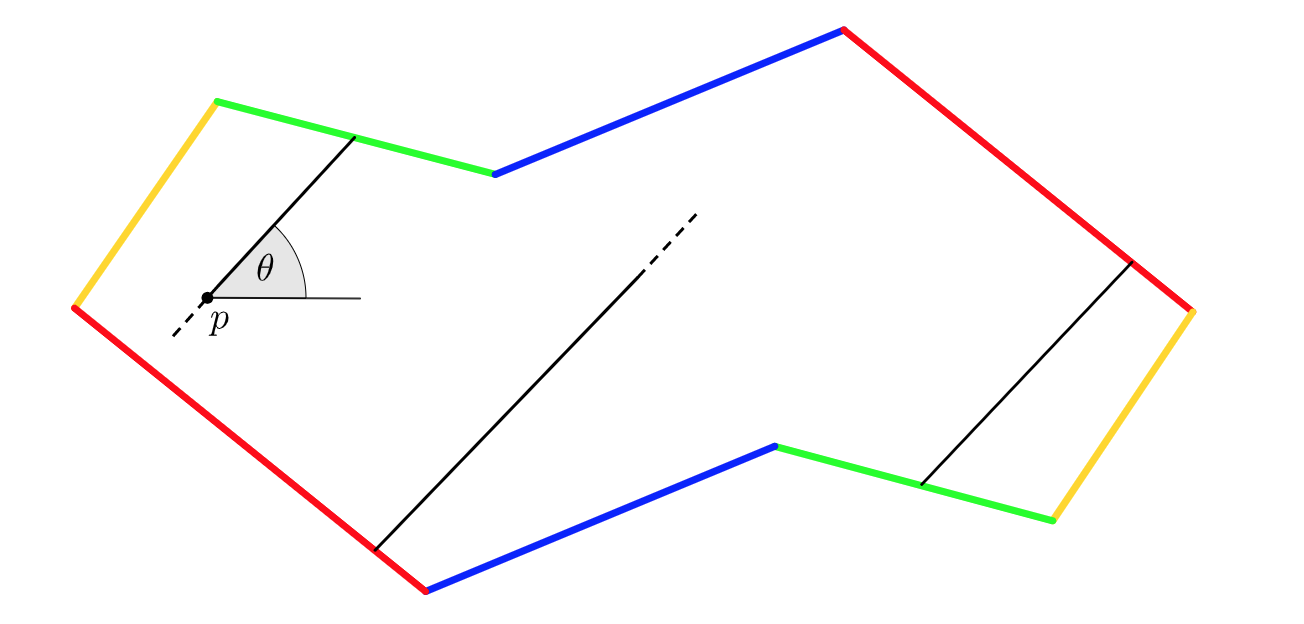}
\caption{A leaf through $p$ of the foliation in direction $\theta$.}
\end{figure}

This leaf can behave in different ways. For example, it could close up and repeat itself again and again to form a \textit{closed leaf}; we call a foliation for which every infinite leaf is closed \textit{completely periodic}. A leaf can also never close up but instead be dense on some subsurface. We call a foliation for which every infinite leaf is dense \textit{minimal}. These are the only two types of leaves we can have on translation surfaces. A natural question to ask is what the structure of the directional foliation on a translation surface looks like for a "typical direction". In fact, this question has already been fully answered: any directional foliation on a translation surface is also a directional flow, using the parametrization given by the Lebesgue distance in the complex plane. A famous theorem of Kerchoff, Masur and Smillie then asserts that for a full measure set of directions $\theta \in S^1$, the directional flow on any (connected) translation surface is uniquely ergodic.  

An important fact to note is that for the directional foliation on translation surfaces to be well-defined, we only require the edges of the polygon that represents the surface to be parallel, but not necessarily to be of the same length. Imagine extending our "glueing maps" to maps of the form $z \rightarrow \lambda z + v$, where $\lambda \in \mathbb{R}_+^*$ and $v \in \mathbb{C}$, meaning that we use a dilation as well as a translation. In this way we can also glue polygons in the plane whose edges come in pairs that are parallel to each other but might have different lengths. A surface obtained in this way is a dilation surface, an example is given in Figure 3.

\begin{figure}[h]
\centering
\includegraphics[width=0.43\textwidth]{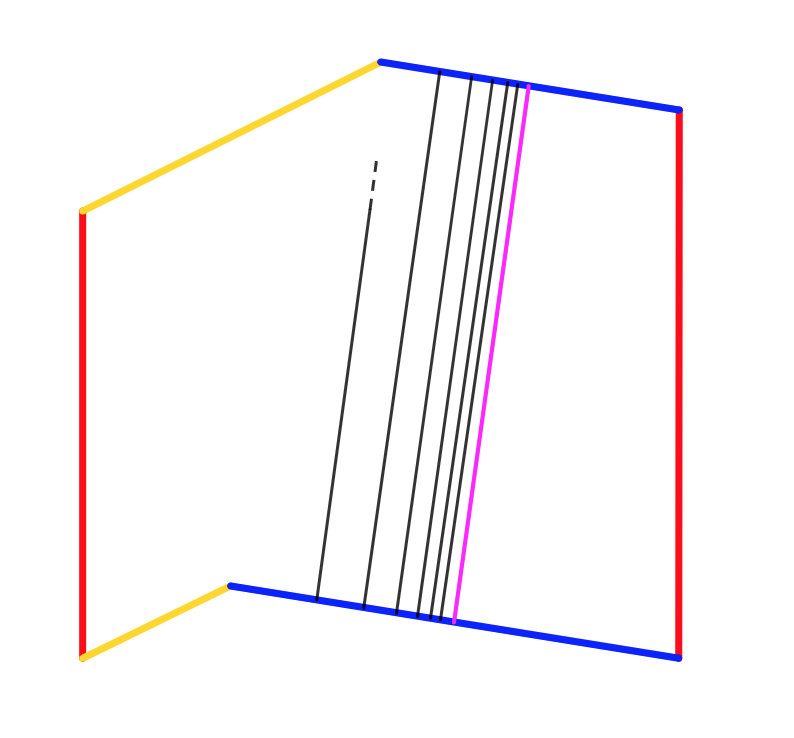}
\caption{A dilation torus and two leaves of a directional foliation.}
\end{figure}

More generally, a dilation surface is geometric surface that is locally homeomorphic to the complex plane and whose structure group is generated by the group of translations and dilations. We will see that unlike in the case of translation surfaces, a dilation surface does not always have a polygonal representation (such that the vertices of the polygon project to actual singularities on the surface). 

The spectrum of possible behaviours of the directional foliation on dilation surfaces is much broader than the one of translation surfaces. Of course, like in the case of translation surfaces, the directional foliation on a subsurface can be completely periodic (1) or minimal (3). But it can also exhibit so-called \textit{hyperbolic behaviour}: In Figure 3, two leaves of a directional foliation are drawn such that there exists a closed leaf (there in pink) that connects the midpoints of the two blue edges. As the two blue edges are glued with a dilation, the black leaf "spirals" towards the pink closed leaf and eventually accumulates on it. In fact, this is true for any leaf entering the region between the blue edges, bounded to the right by the red edge and to the left by the line that connects the left endpoints of the blue edges. Such a region is called an \textit{affine cylinder}, if the two blue edges had the same length it would be called a \textit{flat cylinder}. A flat cylinder, and hence the directional foliation on a translation surface, never exhibits this type of hyperbolic behaviour.

A foliation on a dilation surface which can be decomposed entirely into affine cylinders, meaning that any leaf is attracted or repelled by a finite number of closed leaves, is called $\textit{Morse-Smale}$ (2). Unlike in the case of translation surfaces, the question about the "typical" behaviour of the directional foliation on a dilation surface has not yet been answered. However, Selim Ghazouani conjectures in \cite{ghazouani} that for any dilation surface $S$ which is not a translation surface, for a full measure set of directions in $S^1$, the directional foliation on $S$ is Morse-Smale (see Conjecture \ref{selimconjecture} in this thesis). This conjecture is supported by all the examples that we will see. Nevertheless, for some examples of dilation surfaces, on a measure zero set of directions, we can prove that there exist a wider range of possible dynamical behaviours.
This is the case for the Disco surface, a dilation surface that is introduced in Figure 5 in the next section and will serve as our main example throughout the thesis. Here, we will see that there exist "special" directions that are \textit{Cantor-like} (4), i.e for which the corresponding foliation accumulates on a set whose cross-section is a Cantor-set. Why this behaviour arises on the Disco surface will be explained in a geometric way in Chapter 4. One of the goals of this thesis is then to show that together with the Cantor-like behaviour, we have already completed the list of possible behaviours for directional foliations on dilation surfaces. Our main theorem, stated at the end of the introduction, shows that any dilation surface can be split into a finite number of subsurfaces on which the directional foliation is either in case 1,2,3 or 4. 

\subsection{Interval exchange transformations} Our main theorem also includes a more precise description of case 3 and 4 using the notion of a first return map on a suitably chosen transversal segment. More precisely, consider a foliation in direction $\theta$ on a dilation surface. A \textit{transversal segment} $\Sigma$ is a curve on the surface that never travels along the foliation, meaning it has no open subintervals contained in a leaf of the foliation. On $\Sigma$, we can define a first return map as follows: For any point $p$ in $\Sigma$, travel along the leaf through $p$ in direction $\theta$. If you hit again a point in $\Sigma$, define this point to be $f(p)$. This extends to a map $f$ on $\Sigma$ that is not necessarily defined everywhere. It turns out that for dilation surfaces, $f$ is an \textit{affine interval exchange transformation}. 

In the same way that dilation surfaces are a natural generalization of translation surfaces, affine interval exchange transformations are a natural generalization of \textit{standard interval exchange transformations}: Consider a finite partition of the interval $X=[0,1)$ into connected subintervals. A standard interval exchange transformation, or IET in short, is a bijective map $T:X \rightarrow X$ that rearranges each interval of the partition using translations only. An affine interval exchange map, or AIET for short, is defined in the same way but with the additional freedom of using translations as well as a dilations to rearrange the intervals, meaning that the length of the intervals is no longer preserved. An example is given in Figure 4.

\begin{figure}[h]
\centering
\includegraphics[width=0.8 \textwidth]{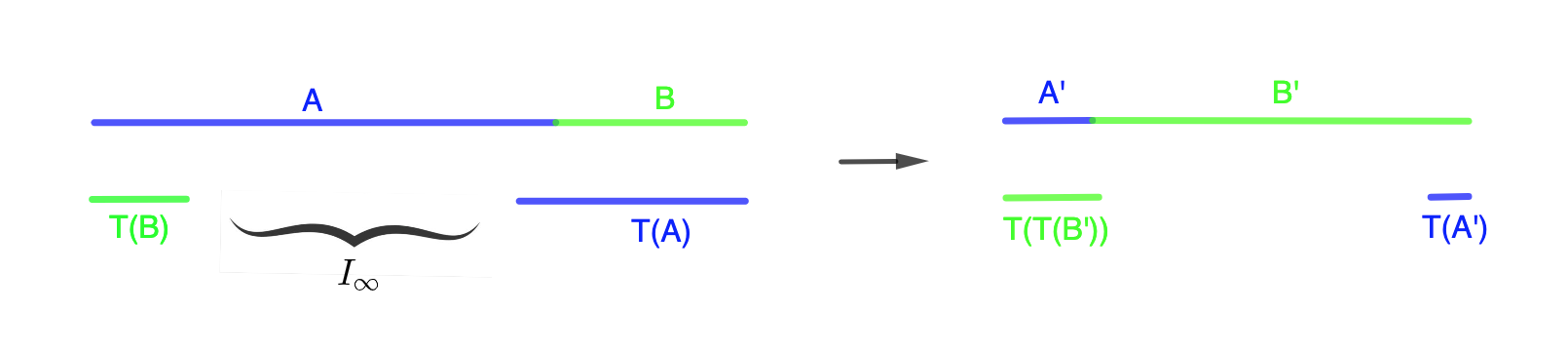}
\caption{An AIET on four intervals.}
\end{figure}
\vspace{3mm}

Both maps are special cases of \textit{generalized interval exchange transformations}, or GIETs for short, maps that are defined as above but now the intervals can be rearranged using orientation preserving diffeomorphisms. These maps in turn are generalizations of orientation-preserving diffeomorphism of the circle and have been extensively studied over the past decades. While IETs have been well-understood, many questions are still open for GIETs, however there is evidence that much of their dynamics can be reduced to the case of AIETs (see \cite{aprioribounds} for further information). There has been a long history in understanding the relationship between affine and standard interval exchange transformations using so called \textit{semi-conjugacies}, where a semi-conjugacy between a map $f$ defined on $X$ and a map $g$ defined on $Y$ is a surjection $h: X \rightarrow Y$ such that $f \circ h = h \circ g$. A famous theorem of Marmi, Moussa and Yoccoz (see \cite{MarmiMoussaYoccoz}) states that "most" IETs admits a semi-conjugated AIET with a \textit{wandering interval}, i.e an interval whose forward images never intersect itself. The proof uses a procedure called \textit{blowing-up}, where a suitable point on the domain of definition of the IET is replaced with a whole interval. If this point has been well-chosen, the new map obtained is in "most" cases an AIET and the orbit of the blown-up point then becomes a wandering interval. Using the opposite procedure, i.e by "blowing down" whole intervals to one point, our main theorem will show that in case 3 and 4, on a suitably chosen transversal segment of the dilation surface, the first return map with respect to the foliation is semi-conjugated to a minimal IET.

We would like to present another link between dilation surfaces and interval exchange transformations: we claim that to any standard or affine interval exchange transformation $T:X \rightarrow X$ we can associate a dilation surface. Take two copies of $X$ and arrange one on the top, one on the bottom. The point $x$ on the copy above will be identified with the point $T(x)$ on the copy below. We join the right and left endpoints with two vertical lines that we identify with each other. The surface we obtain is a dilation surface if $T$ is an AIET and a translation surface if $T$ is an IET. An example of this construction for the map in Figure 4 is given in Figure 5, the dilation surface we obtain in this case is the so-called Disco surface.

\begin{figure}[h]
\centering
\includegraphics[width=0.78\textwidth]{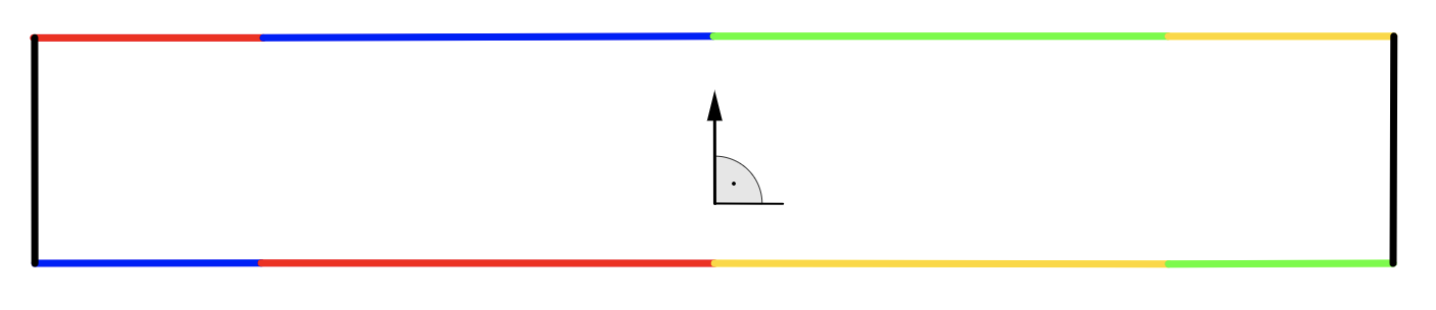}
\caption{The Disco surface.}
\end{figure}

Denote by $\mathcal{F}_{\pi /2}$ the vertical foliation on this surface. Then the first return map on any horizontal cross section $\Sigma$ of the surface is exactly $T$. Thus, understanding the dynamics of standard and affine interval exchange transformations amounts to understanding the directional foliations of dilation surfaces and translation surfaces. 

\subsection{Main results} In this section, we state the main results of this thesis. We have already mentioned some of the results throughout the introduction, however, we still lack the notion of a \textit{non-trivially recurrent leaf closure} to state the theorem in full. 
A \textit{recurrent leaf} is a leaf that keeps coming back to any neighbourhood of any point that lies on it. Leaves that close up again are recurrent, but those we call \textit{trivially recurrent}. Examples of leaves that are non-trivially recurrent are leaves that are dense in a subsurface or accumulate on a Cantor-like set as in the case of the Disco surface. The closure of a non-trivially recurrent leaf is what we call a \textit{non-trivially recurrent leaf closure}. Our theorem then asserts that

\begin{thm}\label{maintheorem} Given a directional foliation $\mathcal{F}_{\theta}$ on any dilation surface $S$, there exists a decomposition of $S$ into subsurfaces that either have no recurrent leaf or are in one of the following cases:
\begin{enumerate}
\item Flat cylinders where the foliation is completely periodic,
\item Affine cylinders where the foliation is Morse-Smale,
\item Minimal subsurfaces where the foliation is minimal,
\item Subsurfaces where the foliation is Cantor-like.
\end{enumerate}
In case (3) and (4), the first return map on any finite union of segments transversal to $\mathcal{F}_{\theta}$ that intersects a non-trivially recurrent leaf is semi-conjugated to a minimal IET.  
\end{thm}

The proof of the first part of the theorem heavily relies on a decomposition theorem of C. J. Gardiner that was published around 1983 for flows with finitely many singularities on surfaces in \cite{gardiner}. It states that given a surface and a flow (or foliation) with finitely many singularities, there exists a finite set of curves that separate the surface into subsurfaces with at most one non-trivially recurrent leaf closure. The proof of the second part of the theorem, which involves an explicit construction of the semi-conjugacy, uses some of the ideas contained in \cite{gutierrez}, a paper from Carlos Gutierrez published in 1983. 

Applying our main theorem to the suspension of any AIET, we deduce the following corollary:

\begin{cor}\label{maincorollary} Given an affine interval exchange transformation $T:X \rightarrow X$, there exists a decomposition of $X$ into finitely many subsets $L_1, \dots L_n$ such that $L_i$ is a finite union of intervals for $i \in \{1,\dots,n \}$ that either does not intersect a recurrent orbit of $T$ or the first return map $f: L_i \rightarrow L_i$ is in one of the following cases: 

\begin{enumerate} 
\item completely periodic,
\item Morse-Smale,
\item minimal,
\item Cantor like.
\end{enumerate}
In case (3) and (4), $f$ is semi-conjugated to a minimal IET.
\end{cor}

Both our main theorem as well as its corollary will be visualized throughout the text using the Disco surface as an example. We hope that this will provide the reader not only with the theoretical background on the structure of foliations on dilation surfaces, but also with a geometric picture to keep in mind. 

\subsection{Short roadmap} We conclude the introduction by providing a short roadmap for the coming chapters. In the second chapter, we explain the prerequisite material for this thesis, including the formal definition of foliations on surfaces. Chapter 3 is an introduction to dilation surfaces and should be read by anyone not familiar with the subject. In particular, we define the notion of \textit{linear holonomy} that is used in many of our key arguments in the proof of our main theorem. In Chapter 4, we give an overview of the four different types of behaviours of directional foliations on dilation surfaces. Here we also introduce the \textit{Disco surface}, our most important example throughout the thesis. Chapter 5 is dedicated to the statement and explanation of Gardiner's decomposition theorem. We then prove our main theorem as well as its corollary for affine interval exchange transformations in Chapter 6.



\section{Preliminaries}

In this chapter, we recall all of the prerequisite material for the following chapters. In particular, we introduce surfaces with singularities and foliations on surfaces. We further recall the definition of a Cantor set and give the formal definition of affine interval exchange transformations. 

\subsection{Surfaces with singularities} In order to introduce dilation surfaces in Chapter 3, we need the definition a \textit{surface} using an \textit{atlas}. An atlas consists of a collection of \textit{charts} that cover a topological space and describe its local structure. More precisely,

\begin{defs} A \textit{chart} for a topological space $X$ is a homeomorphism $\phi$ from an open subset $U$ of $X$ to an open subset of $\mathbb{C}$. The chart is traditionally recorded as the ordered pair $(U,\phi)$.
\end{defs}
\begin{defs} An \textit{atlas} for a topological space $X$ is an indexed family $\mathcal{A}=(U_i,\phi_i)_{i\in I}$ of charts on $X$ which \textit{covers} $X$, that is $\bigcup_{i \in I} = X$. 
\end{defs}

\begin{figure}[h]
\centering
\includegraphics[width=0.53\textwidth]{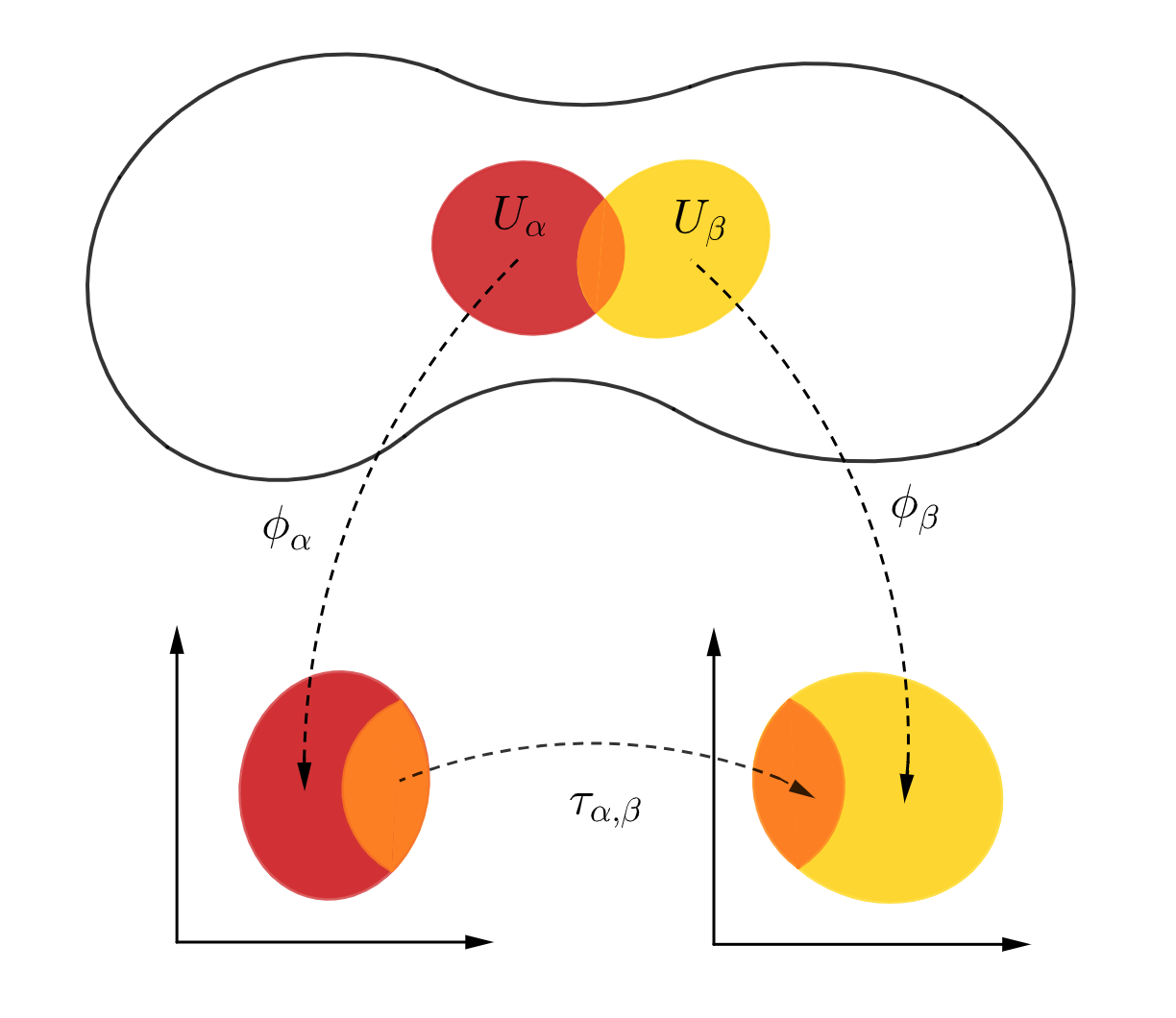}
\caption{The transition map between two charts of an atlas.}
\end{figure}

We call a topological space $X$ that has such an atlas a \textit{surface} or a \textit{two-dimensional manifold}. It is possible to ask for additional structure on a surface using so-called \textit{transition maps}. These maps describe how two charts relate to each other, an illustration is given in Figure 6.  

\begin{defs} Suppose $(U_{\alpha},\phi_{\alpha})$ and $(U_{\beta},\phi_{\beta})$ are two charts such that $U_{\alpha} \cap U_{\beta}$ is non-empty. The \textit{transition map} $\tau_{\alpha,\beta}: \phi_{\alpha}(U_{\alpha} \cap U_{\beta}) \mapsto \phi_{\beta}(U_{\alpha} \cap U_{\beta})$ is the map defined by $$\tau_{\alpha,\beta}=\phi_{\beta} \circ \phi_{\alpha}^{-1}$$ 
\end{defs}

For example, if we require the transition maps to be differentiable, we obtain what we call a \textit{differentiable surface}. For the definition of dilation surfaces, we will require these maps to be of the form $z \rightarrow \lambda z + v$, where $\lambda \in \mathbb{R}_{>0}$ and $v \in \mathbb{C}$. 

Some surfaces might have points where the atlas structure on them is not well-defined, these points are called \textit{singularities}. As we will see, dilation surfaces are surfaces with finitely many singularities. We will further see that the neighbourhood of a singular point on a dilation surface is defined to be homeomorphic to a \textit{k-sheet covering} of the punctured complex plane. Hence, we give the definition of a k-sheet covering below. 

\begin{defs}
    Let $X$ be a topological space. A \textit{k-sheet covering} of $X$ is a continuous map $\pi: E \mapsto X$ such that there exists a discrete space $D$ of cardinality $k$ and for every $x \in X$ there is an open neighbourhood $U \subset X$ such that $\pi^{-1}(U) = \amalg_{d\in D} V_d$ and $\pi|_{V_d}: V_d \mapsto U$ is a homeomorphism for every $d \in D$.
\end{defs}

Another topological notion that will be important for us to define the \textit{linear holonomy} of a curve on a dilation surface in Chapter 3 is a \textit{free} or \textit{based homotopy} of \textit{loops} on a surface.    

\begin{defs} Let $X$ be a topological space, then a path $\sigma:[0,1] \rightarrow X$ on $X$ is called a loop on $X$ if $\sigma(0) = \sigma(1)$. 
\end{defs}

\begin{defs} Let $X$ be a topological space, let $\sigma$ be a loop on $X$. A \textit{free homotopy} of loops is a continuous map $H:[0,1] \times [0,1] \rightarrow X$ such that $\gamma_s(t):= H(s,t)$ is a loop for each fixed $s \in [0,1]$, that is $H(s,0) = H(s,1)$ for all $s \in [0,1]$. \end{defs}

\begin{defs} For a point $x \in X$, a \textit{homotopy of loops based at $x$} is a homotopy $H:[0,1] \times [0,1] \rightarrow X$ where $H(s,0)=H(s,1)=x$ for all $s \in [0,1]$. \end{defs}

These definitions give rise to two equivalence classes of loops on $X$, called the \textit{free homotopy classes} and the \textit{based homotopy classes} of $X$. Two loops are in the same free homotopy class if they are obtained from each other through a free homotopy of loops and they are in the same based homotopy class if they are obtained from each other through a based homotopy of loops.

By now, we have introduced surfaces with singularities and some topological notions that that will help us to define key objects on dilation surfaces. In a next step, we want to equip surfaces with a dynamical structure, in particular we want to define \textit{flows} and \textit{foliations} on surfaces. 

\subsection{Flows and foliations} A flow on a surface is a map that induces a movement of all points on the surface over time. More formally,

\begin{defs}Let $S$ be a surface. A \textit{flow} on $S$ is a continuous map $\phi: S \times \mathbb{R} \rightarrow S$ with the properties
\begin{align*}
    \phi(s,0) = s,\hspace{0.5cm} \phi(s,t_1+t_2) = \phi(\phi(s,t_1),t_2), \hspace{0.4cm} \forall s \in S, t_i \in \mathbb{R}. 
\end{align*}
We further call a point $p \in S$ for which $\phi^t(p)$ is not well-defined for any $t \in \mathbb{R}$ a \textit{singular point} of $\phi$. 
\end{defs}

\begin{defs} Let $\phi$ be a flow on a surface $S$. Fix $t\in \mathbb{R}$ and let $\phi^t$ be the map $\phi(\cdot,t): S \rightarrow S$. The \textit{forward (backward)} \textit{orbit} of a point $s \in S$ is the set 
\begin{align*}
    O^+(s)=\{ \phi^t(s):t \geq 0 \}, \hspace{1cm} (O^-(s)=\{ \phi^t(s):t \leq 0 \} )
\end{align*}
The union  $O(s)=O^+(s) \cup  O^-(s)$ is called the \textit{orbit} through $s$.
\end{defs}

We will see in the next chapter that dilation surfaces are not equipped with a well-defined notion of distance, hence also not with a well-defined notion of unit speed along which we can "travel" along an orbit. Hence we have to change to the machinery of foliations. A foliation on a surface is essentially a partition of the surface into curves called \textit{leaves} that are locally parallel to each other. Foliations are a generalization of flows, as the orbits of a flow determine the \textit{leaves} of a foliation. In line with \cite{Camacho} we define foliations on surfaces in the following way:
\begin{defs}\label{deffoliation}
    Let $S$ be a surface. A \textit{foliation} of $S$ is an atlas $\mathcal{F}$ on $S$ such that any transition map $\tau_{\alpha,\beta}$ is of the form $$\tau_{\alpha,\beta}(x,y) = (h_1(x,y),h_2(y))$$ where $h_1,h_2$ are diffeomorphisms.
    \end{defs}
    
\begin{figure}[h]
\centering
\includegraphics[width=0.5\textwidth]{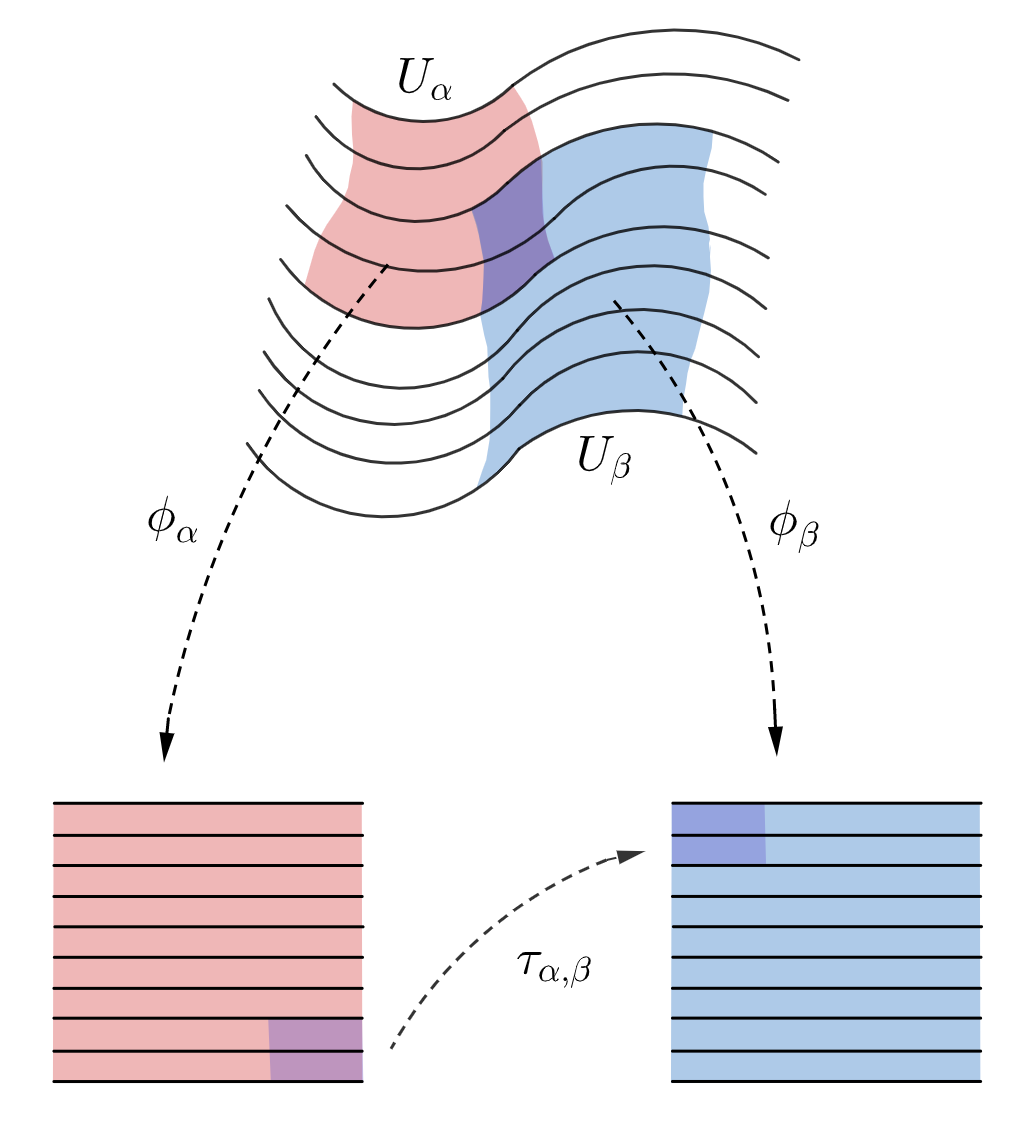}
\caption{The transition maps between two charts of a foliation.}
\end{figure}

A foliation on a surface might have points on the surface where the atlas $\mathcal{F}$ is not well-defined. We call these points \textit{singular points of the foliation}. All non-singular points can be partitioned into \textit{leaves} of the foliation. 
 
\begin{defs} For a chart $(U,\phi) \in \mathcal{F}$, let $h_c$ be the intersection between the horizontal line $y=c$ and $\phi(U)$. If $c$ is chosen such that this intersection is nonempty, then $\phi^{-1}(h_c)$ is called a \textit{plaque} of $\mathcal{F}$. A \textit{path of plaques} of $\mathcal{F}$ is a sequence $\alpha_1, \dots, \alpha_k$ of plaques such that $\alpha_j \cap \alpha_{j+1} \neq \emptyset$ for all $j \in \{1,\dots,k-1\}$.
\end{defs}

\begin{defs} Define the equivalence relation $p \sim q$ if there exists a path of plaques $\alpha_1, \dots, \alpha_k$ with $p \in \alpha_1, q \in \alpha_k$. Then the equivalence classes of $\sim$ are called \textit{leaves} of $\mathcal{F}$. 
\end{defs}

The study of the leaves of foliations is central to this thesis. One way to describe a leaf is through its set of accumulation points. For this, we introduce the notion of the \textit{limit set} of a leaf.

\begin{defs}\label{deflimitset} Let $l$ be a leaf of a foliation on the surface $S$. The \textit{limit set of $l$} is the intersection of the closures of the sets $l-K$ where $K$ is any compact subset of $l$. 
\end{defs}

We call a leaf $l$ \textit{regular} or \textit{finite} if its limit set does not consist of one or two singularities, if it consists of two singularities we call it a \textit{saddle connection}. A regular leaf that is homeomorphic to the circle $S^1$ is called a \textit{closed leaf} or a \textit{periodic leaf}. 

Another way to describe leaves of foliations on dilation surfaces requires so-called \textit{transversal segments}. A transversal segment on a foliated surface is a segment that never travels "along" a leaf. More formally,

\begin{defs} A segment $\Sigma$ on a surface $S$ is called a \textit{transversal segment} if for every point $p \in \Sigma$, there exists a neighbourhood $U_p$ of $p$ and a diffeomorphism $\psi: U_p \rightarrow \mathbb{R}^2$ such that $\psi(p) = (0,0)$ and
\begin{itemize}
\item$\psi(\Sigma \cap U_p)$ is mapped to the x-axis given by $x=0$. 
\item If $l$ is a leaf that intersects $U_p$, then $\psi(l \cap U_p)$ is mapped to the line $y=c$ for some constant $c \in \mathbb{R}$.
\end{itemize}
\end{defs}

In Chapter 3, we will define the \textit{first return map} on a transversal segment for dilation surfaces which yields a way of describing the leaves of a foliation using a one-dimensional map. We will further see in Chapter 4 that in some cases, foliations on dilation surfaces can accumulate on a Cantor set which is why we recall the definition of a Cantor set in the next section.

\subsection{Cantor sets} The reader might be familiar with the ternary Cantor set, obtained by removing the open middle third of an interval and repeating this process over and over for the intervals that remain. A general Cantor set is homeomorphic to the ternary Cantor set. Let $S$ be a subset of a topological space $X$. 
\begin{defs} Let $x \in S$, then $x$ is called an \textit{isolated point} of $S$ if there exists a neighbourhood of $x$ that does not contain any other points of $S$.
\end{defs}
\begin{defs} $S$ is \textit{connected} if we cannot write $S$ as the union of two or more disjoint non-empty open subsets. $S$ is \textit{totally disconnected} if the only connected subsets of $S$ are one-point sets. 
\end{defs}

\begin{defs} $S$ is called a \textit{Cantor set} if it is closed, totally disconnected and has no isolated points. 
\end{defs}

\subsection{IETs and AIETs} To conclude this chapter, we provide the formal definition of \textit{standard} and \textit{affine interval exchange transformations} and \textit{semi-conjugacies}. During the course of this thesis, we will see how these objects are related to dilation surfaces and why they are so important to us. 

\begin{defs} Let $a=a_0 < a_1 \dots a_{m-1} < a_m = b$ be a partition of the interval $X=[a,b)$. A map $T:X \rightarrow X$ is called \textit{ affine interval exchange transformation} or AIET in short if it is a bijective map of the form 
$$
T(x) = \lambda_i x + c_i, \hspace{0.5cm} x \in [a_{i-1}, a_i), \hspace{0.5cm} i = 1,2,...,m
$$
for some vectors $c = (c_1,...,c_m), \lambda=(\lambda_1,..., \lambda_m) \in \mathbb{R}^m$. If $\lambda = (1,1,...,1)$, we call $T$ a  \textit{(standard) interval exchange transformation} or IET in short.
\end{defs}
 
\begin{figure}[h]
\centering
\includegraphics[width=0.65\textwidth]{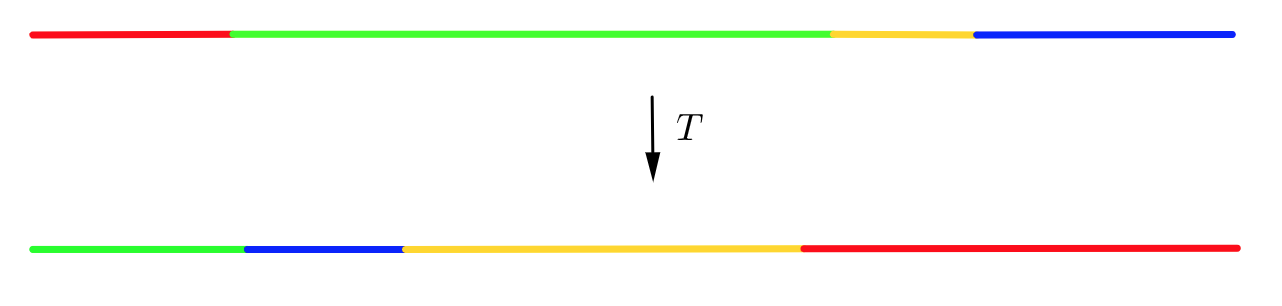}
\caption{Illustration of an AIET on four intervals.}
\end{figure}
\vspace{3mm}


We state the next few definitions for general maps from a topological space to itself, of which interval exchange transformations are a particular case. Let $X$ be a topological space and $f:X \rightarrow X$ be a map.

\begin{defs} The \textit{forward (backward) orbit} of a point $p \in X$ is defined as the set
\begin{align*}
    O^+(p)=\{f^n(p):n \in \mathbb{N}, n \geq 0 \}, \hspace{1cm} (O^-(p)=\{ f^n(p):n \in \mathbb{Z}, n \leq 0 \} )
\end{align*}
An \textit{infinite orbit} is simply an orbit whose set of points is infinite, a \textit{finite orbit} is an orbit whose set of points is finite. 
\end{defs}

\begin{defs} $p \in X$ is called a \textit{periodic point} if there exists $n \in \mathbb{N}$ such that $f^n(p) = p$. \end{defs}

\begin{defs} $f$ is called \textit{minimal} if there are only finitely many finite orbits of points $p \in X$ and if every infinite orbit is dense in $X$.
\end{defs}
Later on in this thesis, we will establish a relationship between AIETs that arise from dilation surfaces and minimal IETs using so-called \textit{semi-conjugacies}. Hence we state the definition of a semi-conjugacy below. 

\begin{defs} Let $X$ and $Y$ be topological spaces, let $f:X \rightarrow X$ and $g:Y \rightarrow Y$ be piecewise continuous functions. We say that $f$ is \textit{semi-conjugated} to $g$ if there exists a map $h : Y \rightarrow X$ that is continuous, monotonic and surjective such that $$ f \circ h = h \circ g.$$ If in addition $h$ is injective and its inverse is continuous, meaning that $h$ is a homeomorphism, then we say that $h$ and $g$ are \textit{conjugated}. 
\end{defs}

We have now defined all of the prerequisite material and we now move on to the main part of this thesis. We proceed with an introduction to dilation surfaces.

\label{chapter2}
\section{Introduction to Dilation Surfaces}  
In this chapter, we formally introduce dilation surfaces and explain some of their key properties. We further define \textit{directional foliations} on dilation surfaces as well as the key tools used in the study of these foliations.

\subsection{Dilation surfaces} In line with \cite{ghazouani} and \cite{cascades}, we define dilation surfaces as follows: 
\begin{defs} A \textit{dilation surface} is a surface $S$ together with a finite set of points $\Sigma \subset S$ and an atlas $\mathcal{A}=(U_i,\phi_i)_{i\in I}$ on $S \backslash \Sigma$ whose charts $\phi_i$ take values in $\mathbb{C}$ such that:
\begin{enumerate}
\item the transition maps are locally restrictions to elements of Aff$_{\mathbb{R}_{+}^*}(\mathbb{C}) = \{z \mapsto \lambda z+v \hspace{0.1cm} |\hspace{0.1cm} \lambda \in \mathbb{R}, \lambda>0, v \in \mathbb{C}\};$  
\item each point of $\Sigma$ has a punctured neighborhood which is homeomorphic to some k-sheet covering of $\mathbb{C}^{*}$.
\end{enumerate}
We call an element of the set $\Sigma$ a singularity of the dilation surface $S$. \end{defs}

While this definition is rather formal, the picture of a dilation surface to have in mind is a collection of euclidean polygons glued together along pairs of parallel sides using dilations and translations only. This is illustrated in Figure 9, where two euclidean polygons are glued together to obtain a genus two dilation surface of angle $6 \pi$ called the two-chamber surface. Note that the vertices of the two polygons all project to the single singularity on the surface. 

\begin{figure}[h]
\centering
\includegraphics[width=1.0\textwidth]{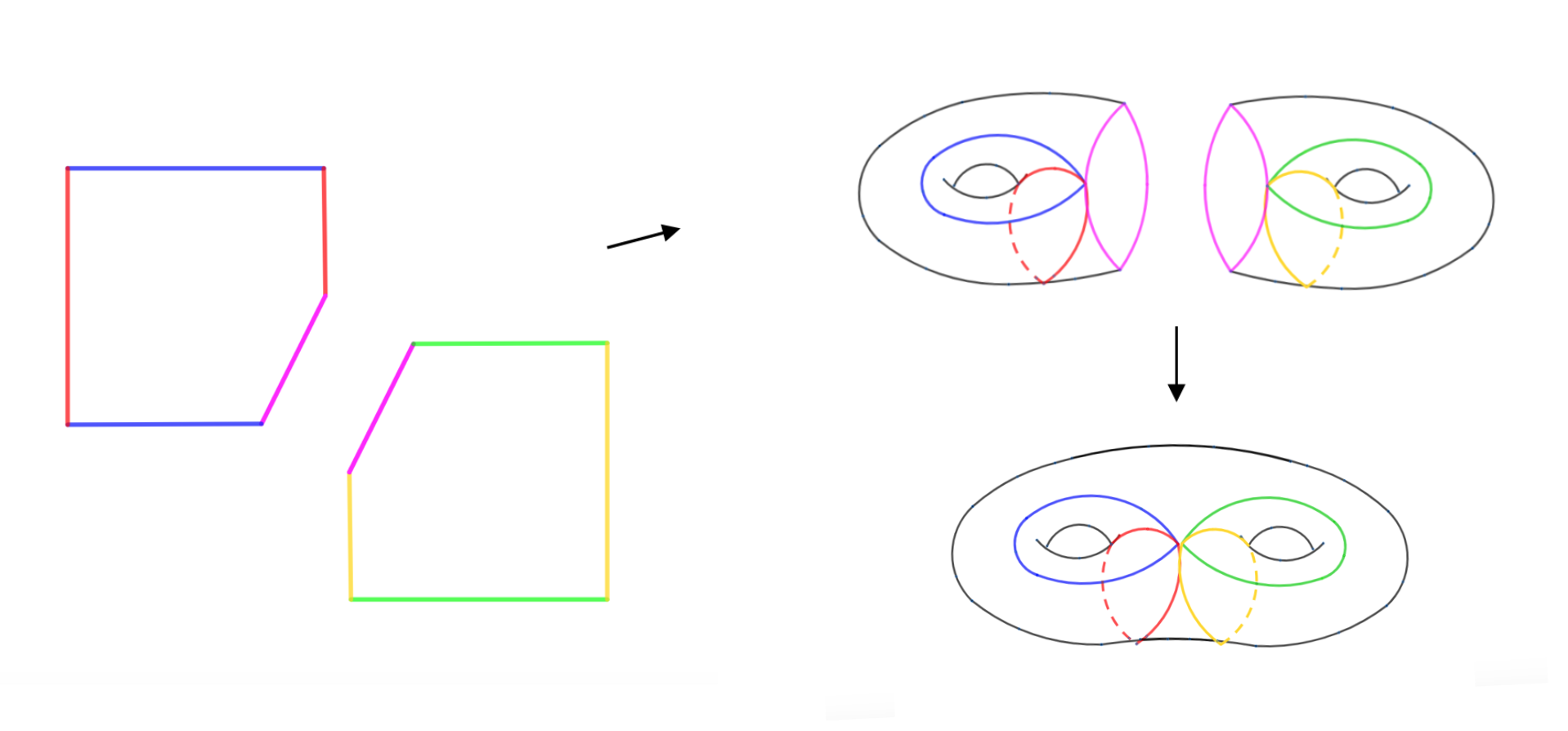}
\caption{The two-chamber surface.}
\end{figure}

To make this construction more formal, let $P_1, \dots P_m$ be a collection of polygons in the complex plane with oriented edges such that the interior of the polygon is always to the left of the edges. Suppose that for any edge $s_i$ of $P_k$ there is an edge $s_j$ of $P_l$, parallel so $s_i$ with opposite orientation, where $j \neq i$. Suppose further that there is a map $f_i \in $Aff$_{\mathbb{R}_{+}^*}(\mathbb{C})$ such that $f_i(s_i)=s_j$ and $f_j = f_i^{-1}$. Consider the quotient space $S$ obtained by identifying all $s_i$ with the corresponding $s_j$ through the map $f_i$, denote by $h$ the quotient map.

\begin{prop} The surface $S$ obtained from the polygons $P_1, \dots, P_m$ by identifying the edge $s_i$ with the edge $s_j$ through the map $f_i \in $Aff$_{\mathbb{R}_{+}^*}(\mathbb{C})$ as described above is a dilation surface. 
\end{prop}
\begin{proof} We first show that for each $p \in S$ that is not the image of a vertex under $h$, we can find a neighbourhood $U_p$ of $p \in S$ and a chart $\phi_p: U_p \mapsto \mathbb{C}$ such that for any $p,q \in S$, whenever $U_p \cap U_q \neq \emptyset$, $\phi_p \circ \phi_q^{-1}$ is of the form $z \mapsto \lambda z+v \in$ Aff$_{\mathbb{R}_{+}^*}(\mathbb{C})$. If $p \in S$ is not the image of a vertex, there are two cases:

\begin{enumerate}

\item $p \in S$ is the image under $h$ of an interior point $\overline{p}$ of a polygon $P$. Then there exists a neighbourhood $U_{\overline{p}}$ in $p$ that is homeomorphic to a neighbourhood $U_p$ in $S$ via the quotient map. We define the inverse quotient map $h^{-1}: U_p \mapsto U_{\overline{p}}$ to be the chart $\phi_p$, cleary this map is a homeomorphism. 

\item $p \in S$ is the image of an interior point of an edge $s_j = f_i(s_i)$ of a polygon $P$. Then $p$ has two pre-images $\overline{p}_i,\overline{p}_j $ and $p$ has a neighbourhood $U_p$ in $S$ whose pre-image is the union of two half-disks $D_i, D_j$ of parallel diameters and radii $r_i,r_j$, one of them bordering on $s_i$ and the other one on $s_j$. We have that $U_{\overline{p}} := f_i(D_i) \cup D_j$ is a full disk of radius $r_j$. Take the inverse quotient map $h^{-1}$ and modify it to be injective by sending any point in $U_p$ that has two pre-images to its pre-image in $s_i$ only. Then compose the modified inverse quotient map with $f_i$ to obtain a homeomorphism $\phi_p: U_p \mapsto U_{\overline{p}}$.
\end{enumerate}
\begin{figure}[h]
\centering
\includegraphics[width=0.6\textwidth]{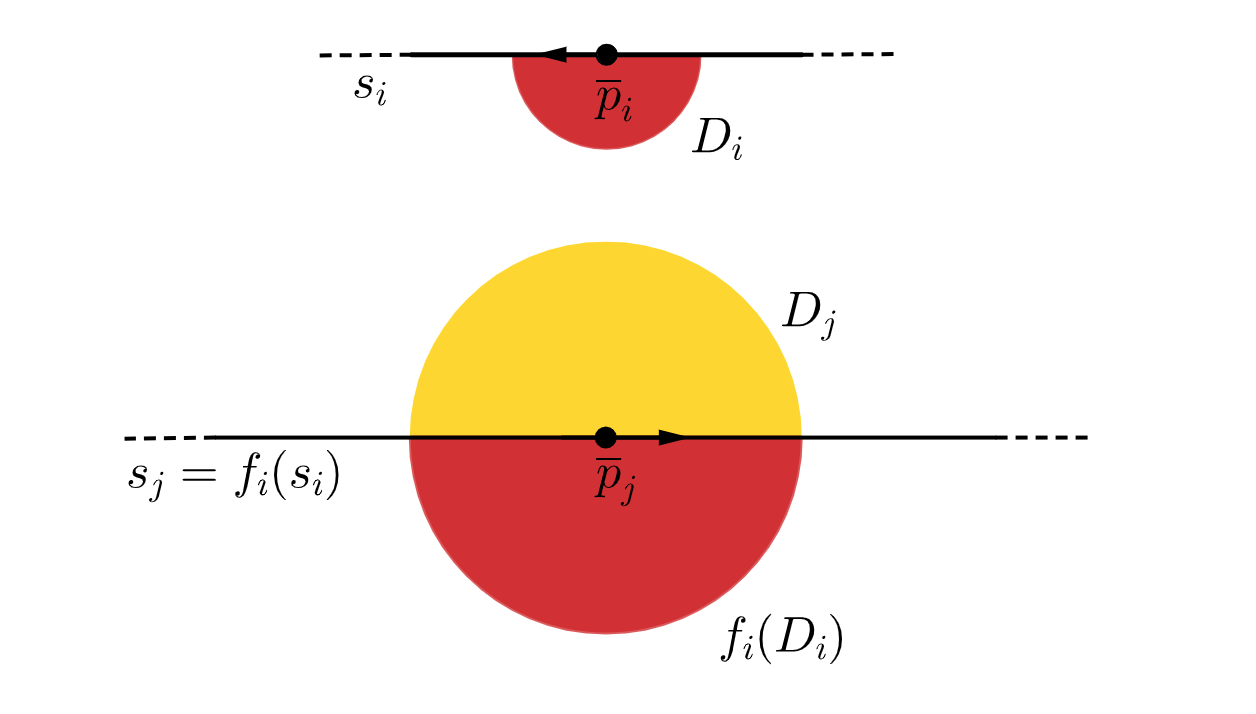}
\caption{The construction of $U_{\overline{p}} \subset \mathbb{C}$.}
\end{figure}

If $p,q$ are both in the image of an interior point of some $P_i$, then $\phi_p \circ \phi_q^{-1}$ is the identity, if one of them is the image of an interior point of an edge of some $P_i$, then $\phi_p \circ \phi_q^{-1}$ is either the identity or a map in Aff$_{\mathbb{R}_{+}^*}(\mathbb{C})$. Hence, the first part of the definition is satisfied.

Let now $p \in S$ be the image of a vertex of a polygon $P$. We want to show that there exists a neighbourhood $U_p$ of $p$ such that $U_p \backslash \{p \}$ is homeomorphic to a k-sheet covering of $\mathbb{C}^{*}$. The pre-image of $p$ is a set of vertices $p_1, \dots p_k$ of $P$ since endpoints of edges are identified  only with endpoints of edges. There is a neighbourhood $U_p$ of $p$ such that the pre-image of $U(p) \backslash \{p \}$ in $P$ is a union of angular sectors $S_i$ of angle $\theta_i$ and radius $r_i$, where $i=1 \leq i \leq k$ . Any of these sectors is bounded by edges $m_i$ and $n_i$, so that $m_{i+1}$ is parallel to $n_i$ for $i=1 \leq i \leq k-1$. Because $n_k$ is parallel to $e_1$, there is an integer $l \in \mathbb{N}$ so that
$$
\sum_{i=1}^k \theta_i = 2\pi l
$$
We construct a homeomorphism $t$ from the pre-image of $U_p \backslash \{p\}$ to a k-sheet covering of $\mathbb{C}^*$ as follows: Take the angular sector $S_1$ of radius $r_1$ and map it (using a map in Aff$_{\mathbb{R}_{+}^*}(\mathbb{C})$) to the angular sector $S_1'$ of radius 1 centered at $0$, so that $m_1$ is mapped to the horizontal axis. Then map the angular sector $S_2$ of radius $r_2$ to the angular sector $S_2'$ of radius 1 centered at $0$, so that $m_2$ is mapped to $e_1$. As the angles add up to a multiple of $\pi$, we will obtain a k-sheet covering of $\mathbb{C}^*$. Then, modify the inverse quotient map to be injective so that whenever a point has two pre-images, choose only the one that lies in $\bigcup_{i=1}^{k} e_i$. The composition of the inverse quotient map with $t$ will give the desired homeomorphism between $U_p \backslash \{p\}$ and a k-sheet covering of $\mathbb{C}^*$. 
\end{proof}
In Figure 11, we illustrate a neighbourhood of the singular point on the two chamber surface that is homeomorphic to a  3-sheet covering of $\mathbb{C}^*$.  

\begin{figure}[h]
\centering
\includegraphics[width=0.75\textwidth]{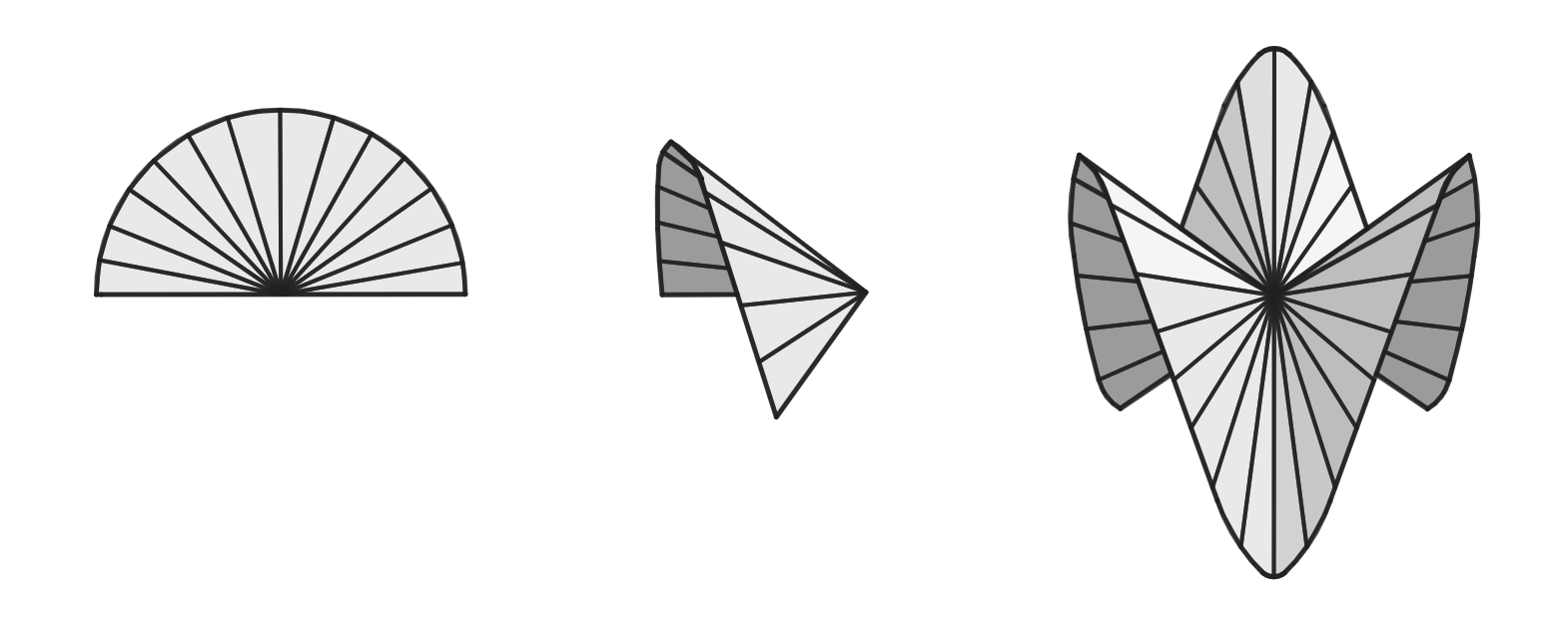}
\caption{A 3-sheet covering of $\mathbb{C}^*$.}
\end{figure}

\begin{rem} We have shown that the surface we obtain by glueing the edges of a polygon using dilations is indeed a dilation surface. However, not every dilation surface has a polygonal representation such that the vertices of the polygons project to actual singularities of the surface. We will see a result in Chapter 4 that gives us a necessary and sufficient condition for such a polygonal representation to exist.
\end{rem}

\subsection{Directional foliations and the first return map} Our main objects of interest, \textit{directional foliations}, come naturally with every dilation surface. To see this, we first remark that in the complex plane, a straight line in direction $\theta \in S^1$ is simply a straight line whose angle with the horizontal axis is equal to $\theta$. Because dilation surfaces are modelled after the complex plane, we can define the same notion for dilation surfaces. 

\begin{defs}  A curve on a dilation surface is called a \textit{straight line in direction $\theta$} if its image by a chart in the complex plane is always a straight line in direction $\theta$. \end{defs}

The fact that these objects are well-defined follows from the fact that straight lines in direction $\theta$ in the complex plane remain straight lines in direction $\theta$ under the action of Aff$_{\mathbb{R}_{+}^*}(\mathbb{C})$, the structure group of dilation surfaces. It is important to note however that the notion of distance is not preserved with respect to the action of Aff$_{\mathbb{R}_{+}^*}(\mathbb{C})$. This is why in the following we define the directional foliation on a dilation surface and not the directional flow. Let $S$ be a dilation surface with atlas $\mathcal{A}$. Then the transition maps $\tau_{\alpha,\beta}$ are in Aff$_{\mathbb{R}_{+}^*}$. As these maps preserve horizontal lines, they are of the form 
 $$\tau_{\alpha,\beta}(x,y) = (h_1(x,y),h_2(y))$$
Hence $\mathcal{A}$ satisfies the definition of an atlas for a foliation (see Definition \ref{deffoliation}). The leaves of this atlas are straight lines in direction $\theta = 0$. 
\begin{defs} Let $S$ be a dilation surface with atlas $\mathcal{A}$. For any $\theta \in S^1$, for any chart $(U_i,\phi_i) \in \mathcal{A}$, we can define a new chart $(U_i,\tilde{\phi_i})$ via $\tilde{\phi_i} := r_{-\theta} \circ \phi_i$ where $r_{-\theta}$  is a rotation in direction $-\theta$. Then the atlas $\mathcal{\tilde{A}}$ given by those new charts defines a foliation on $S$ whose leaves are such that their image by a chart in $\mathcal{A}$ is always a straight line in direction $\theta$. We call the foliation given by $\mathcal{\tilde{A}}$ the \textit{directional foliation in direction $\theta$} on $S$ and denote it by $\mathcal{F}_{\theta}$. 
\end{defs}

For a given directional foliation, for any point $p \in S$, we denote by $l_p$ the \textit{leaf through $p$}. We have seen in the introduction how to draw the leaves of the directional foliation on a dilation surface that has a polygonal representation. 

\begin{figure}[h]
\centering
\includegraphics[width=0.55\textwidth]{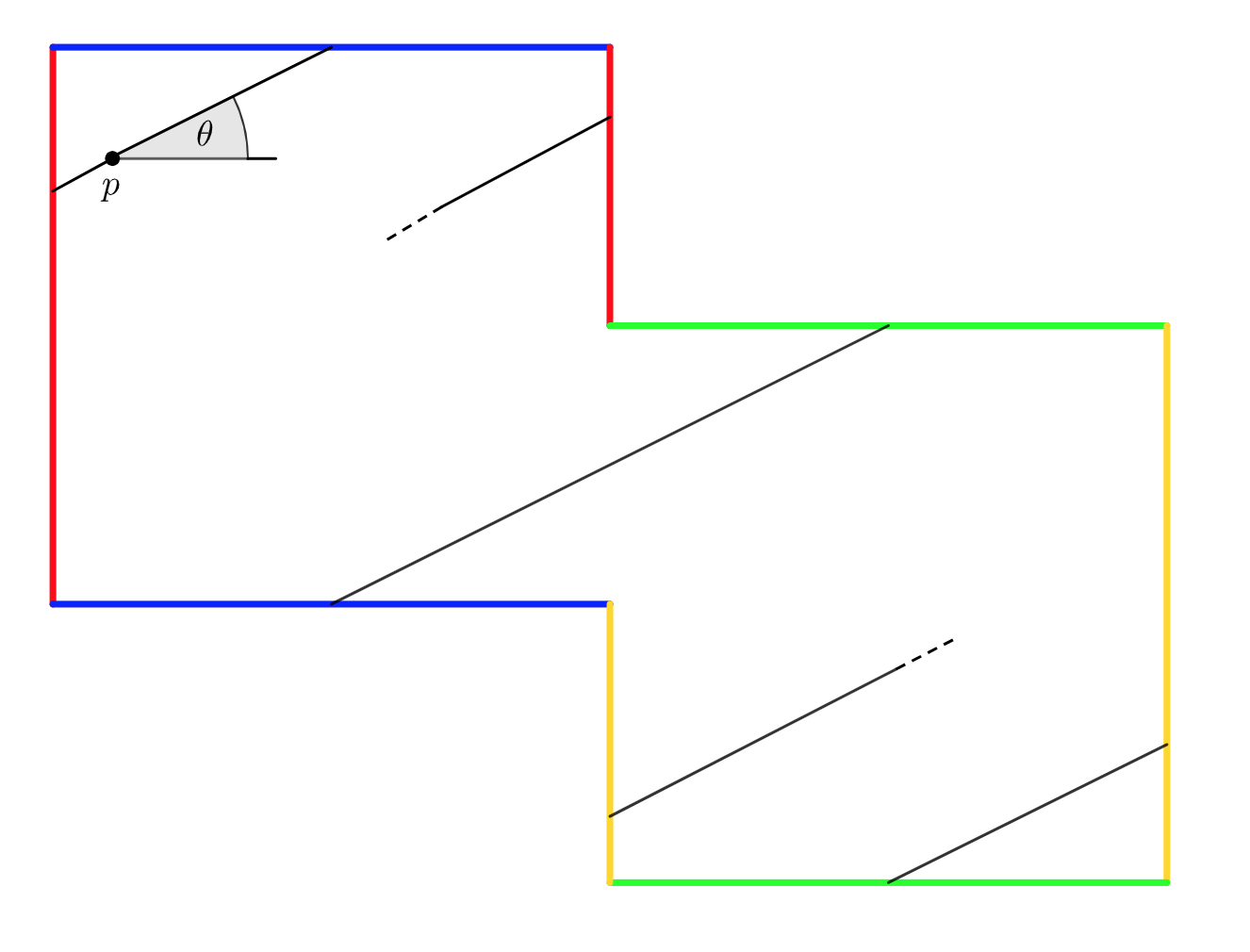}
\caption{The leaf $l_p$ for $\mathcal{F}_{\theta}$ on the two chamber surface.}
\end{figure} 

We note that the leaves of the directional foliation of a dilation surface consist of the same set of points in direction $\theta$ and $-\theta$. We can orient each leaf by denoting the direction $\theta$ as the \textit{forward direction} or the \textit{future} and the direction $-\theta$ as the \textit{backward direction} or the \textit{past}. This allows us to define a powerful tool we use to describe foliations called the \textit{first return map}. 

\begin{defs}\label{deffirstreturnmap} Let $\mathcal{F}_{\theta}$ be a directional foliation on a dilation surface $S$, let $\Sigma$ be transversal segment for the foliation. For any point $p \in \Sigma$, travel along the leaf through $p$ in the future direction, starting from the point $p$. If at some point we reach $\Sigma$ again, define $f(p)$ to be the first point of intersection between the leaf and $\Sigma$. This map extends to the \textit{first return map} $f: \Sigma \rightarrow \Sigma$, note that $f$ is not necessarily everywhere defined.
\end{defs}
\begin{figure}[h]
\centering
\includegraphics[width=0.28\textwidth]{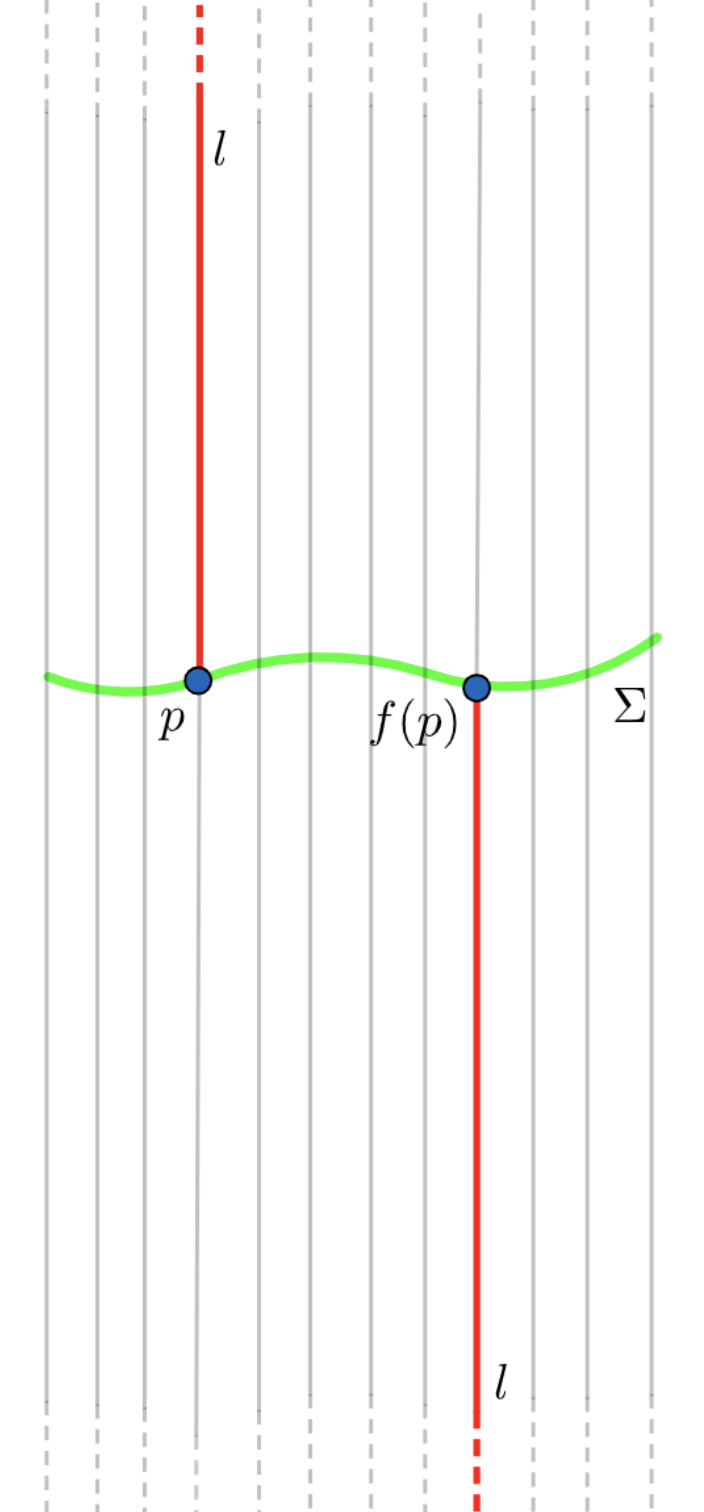}
\caption{An illustration of the first return map.}
\end{figure} 
\begin{rem} It follows from the next section that if we consider the first return map on a transversal segment on a dilation surface, when viewed in a chart that intersects this transversal segment, the first return map where defined is an affine interval exchange transformation. This is another key feature of dilation surfaces that has initiated their recent rise of popularity, as it means that the better we understand foliations on dilation surfaces, the better we also understand the dynamics of affine interval exchange maps. 
\end{rem}

\subsection{Linear holonomy} Due to the affine structure on a dilation surface, there might exist loops on the surface that have non-trivial \textit{linear holonomy}, meaning that the affine structure of the surface is on average either attracted towards the loop or repelled away from the loop. We now want to assign a number to each loop that measures this "amount" of attraction and repulsion. For this, let $\sigma: [0,1] \rightarrow S$ be a loop on a dilation surface $S$. Since $S$ is compact, we can cover $\sigma$ with finitely many charts $(U_i, \phi_i)_{i=0}^{k}$. We can choose these charts such that (\textasteriskcentered) there exists a partition of $[0,1]$ given by $0=t_0 < \dots < t_{k+1} = 1$ where $\sigma([t_i,t_{i+1}]) \subset U_i$ for all $0 \leq i \leq k$. 

\begin{defs}\label{def3.8} Let $\sigma$ be a loop on a dilation surface $S$, let $(U_i, \phi_i)_{i=0}^{k}$ be charts of the dilation atlas on $S$ that cover $\sigma$ such that they satisfy (\textasteriskcentered). Let the transition map between $(U_i, \phi_i)$ and $(U_{i+1}, \phi_{i+1})$  be of the form $z \rightarrow \lambda_i z + v_i$ for $0 \leq i \leq k-1$ and of the form $z \rightarrow \lambda_k z + v_k$ between $(U_k,\phi_k)$ and $(U_0, \phi_0)$, where $\lambda_i \in \mathbb{R}_{>0}$, $v_i \in \mathbb{C}$ for $0 \leq i \leq k$. We define $$\rho(\sigma) := \prod_{i=0}^{k} \lambda_i$$ to be the \textit{linear holonomy of $\sigma$}.
\end{defs}

\begin{figure}[h]
\centering
\includegraphics[width=0.5\textwidth]{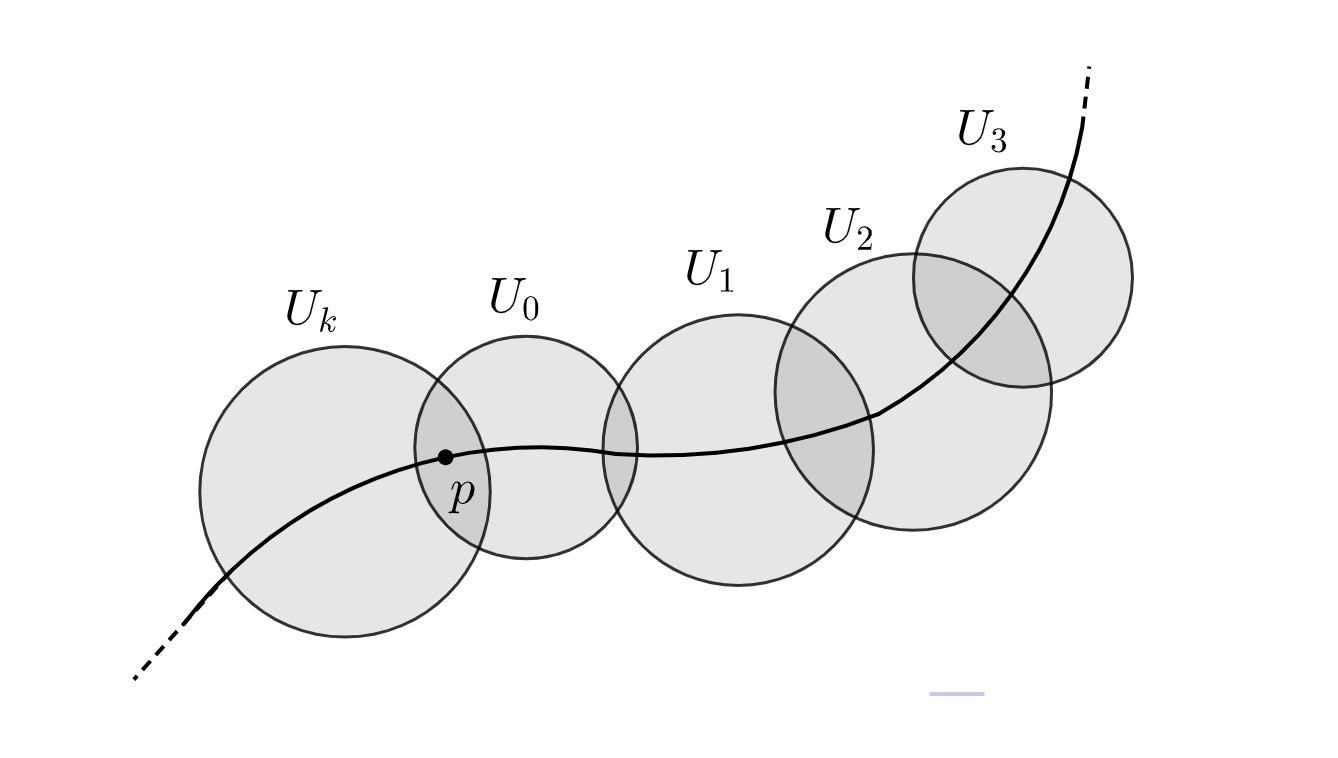}
\caption{A loop $\sigma$ covered with charts $(U_i, \phi_i)_{i=0}^{k}$, where $p = \sigma(0)$.}
\end{figure}

\begin{lem}\label{lem3.9} The linear holonomy $\rho(\sigma)$ of a loop $\sigma$ on a dilation surface $S$ does not depend on the point $\sigma(0)$. Furthermore, $\rho(\sigma)$ does not depend on the choice of charts $(U_i, \phi_i)_{i=0}^{k}$ as well as on the homotopy class of $\sigma$ based at $\sigma(0)$. 
\end{lem}

\begin{proof} The fact that $\rho(\sigma)$ does not depend on $\sigma(0)$ follows directly from the fact that the product $\prod_{i=0}^{k} \lambda_i$ does not change when we rearrange the factors $\lambda_i$. For the proof of the second part of the statement we refer the reader to \cite{Camacho} (see Theorem 1, p. 65). \end{proof} 

\begin{prop} Let $S$ be a dilation surface and $\Sigma \subset S$ the set of singularities of $S$. Let $\mathcal{L}(S)$ denote the set of closed loops on $S \backslash \Sigma$. Then $$\sigma: \mathcal{L}(S) \rightarrow \mathbb{R}_{>0}$$ defines a map that is constant on the free homotopy classes of $\mathcal{L}$(S). \end{prop}

\begin{proof}
    Let $\gamma$ be a loop in the same free homotopy class as $\sigma$. Let $\eta$ be a path from $\sigma(0)$ to $\gamma(0)$. Then $\tilde{\sigma} := \eta \circ \gamma \circ \eta ^{-1}$ is in the homotopy class of $\sigma$ based at $\sigma(0)$. From Lemma \ref{lem3.9} it follows that $\rho(\tilde{\sigma}) = \rho(\sigma)$. It further follows from Lemma \ref{lem3.9} that if $\tilde{\gamma}$ is a reparametrization of $\tilde{\sigma}$ such that $\tilde{\gamma}(0) = \gamma(0)$, then $\rho(\tilde{\gamma}) = \rho(\tilde{\sigma})$. Note that $\tilde{\gamma}$ is in the homotopy class of $\gamma$ based at $\gamma(0)$. Hence, by Lemma \ref{lem3.9}, $\rho(\gamma) = \rho(\tilde{\gamma})$. Overall we have $\rho(\gamma) = \rho(\tilde{\gamma}) = \rho(\tilde{\sigma}) = \rho(\sigma)$ and thus $\sigma: \mathcal{L}(S) \rightarrow \mathbb{R}_{>0}$ defines a map that is constant on the free homotopy classes of $\mathcal{L}(S)$.
\end{proof}
    
\begin{prop}\label{prop_lin_hol} Consider a dilation surface $S$ with directional foliation $\mathcal{F}_{\theta}$ such that there exists a closed leaf $l$ with linear holonomy $\rho(l)$. Let $p \in l$. Then there exists a segment $q_0$ that contains $p$ such that the first return map $f:q_0 \rightarrow q_0$, where defined, is of the form $z \rightarrow \rho(l) \cdot z $ when viewed in a chart that contains $q_0$. 
\end{prop}

\begin{proof} Let $\sigma: [0,1] \rightarrow S$ be a loop that is contained in a closed leaf of the directional foliation $\mathcal{F}_{\theta}$, where the orientation of $\sigma$ corresponds to the orientation of the leaf in the forward direction. Let $(U_i, \phi_i)_{i=0}^{k}$ be charts of the dilation atlas on $S$ that cover $\sigma$ such that (\textasteriskcentered) is satisfied (recall that (\textasteriskcentered) is explained just before Definition \ref{def3.8}). Consider the corresponding partition of $[0,1]$ given by $0=t_0 < \dots < t_{k+1}=1$. Let $q_0$ be a straight segment on $S$ through $\sigma(t_0)$ that is orthogonal to the direction $\theta$. Then $\phi_0(q_0)$ is a straight segment orthogonal to the direction $\theta$ in the complex plane. In the chart $\phi_0(U_0) \subset \mathbb{C}$, we can move the segment $\phi_0(q_0)$ through parallel transport in direction $\theta$ until the segment reaches the point $\phi_0(\sigma(t_1))$.
\begin{figure}[h]
\centering
\includegraphics[width=0.9\textwidth]{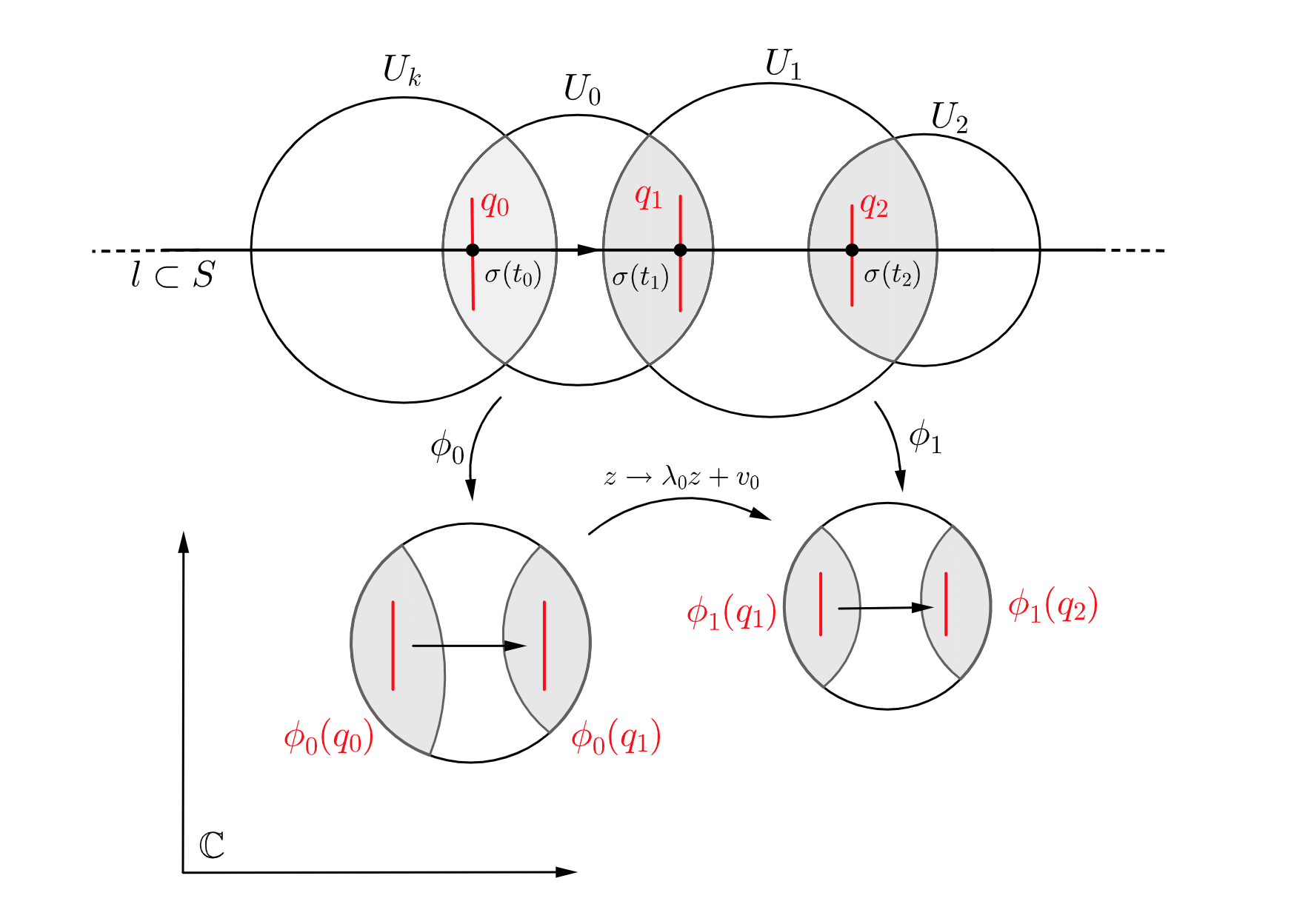}
\caption{}
\end{figure}

Let $q_1$ on the surface $S$ be such that $\phi_0(q_1)$ is equal to $\phi_0(q_0)$ after parallel transport. Note that as the length of $\phi_0(q_0)$ is equal to the length of $\phi_0(q_1)$, we can choose $q_0$ such that the length of $\phi_0(q_0)$, and thus the length of $\phi_0(q_1)$, is small enough for $q_1$ to be entirely contained in $U_0 \cap U_1$. We can further choose $q_0$ even smaller so that $\phi_0(q_0)$ is entirely contained in $\phi_0(U_0)$ at any time during parallel transport (where w.l.o.g we might have to replace $U_0$ with finitely many charts whose union contains $U_0$). 

Now when we transition from $U_0$ to $U_1$, the segment $\phi_0(q_1)$ is translated and dilated by a factor $\lambda_0 \in \mathbb{R}$ so that it becomes $\phi_1(q_1)$. We now move $\phi_1(q_1)$ in $\phi_1(U_1)$ along parallel transport in direction $\theta$ until we reach the point $\phi_1(\sigma(t_2))$ and we let $q_2$ on the surface be such that $\phi_1(q_2)$ is equal to $\phi_1(q_1)$ after parallel transport. Now the length of $\phi_1(q_1)$ is equal to $\lambda_0$ multiplied times the length of $\phi_0(q_0)$. This means that we can further choose $q_0$ such that the length of $\phi_0(q_0)$ multiplied times $\lambda_0$ is small enough so that $q_2 \in U_1 \cap U_2$. We can chose $q_0$ even smaller so that $\phi_1(q_1)$ is contained in $\phi_1(U_1)$ at any time during parallel transport. Note that we can repeat this procedure finitely many times, each time choosing $q_0$ small enough such that $\phi_{i}(q_{i})$ is contained in $\phi_{i}(U_{i})$ at any time during parallel transport and that $q_{i+1} \in U_{i} \cap U_{i+1}$ for $0\leq i \leq k-1$. As $\sigma$ is covered by finitely many charts $U_0, \dots, U_k$, we can thus eventually find $q_0$ such that $q_{k+1}$ is entirely contained in $U_{k} \cap U_0$ and the segment $\phi_0(q_0)$ is always contained in a chart during parallel transport of $\phi_0(q_0)$ along the image of $\sigma$ in the charts. In fact, $\phi_0(q_{k+1})$ and $\phi_0(q_{0})$ both are straight segments orthogonal to the direction $\theta$, where either $\phi_0(q_{k+1}) \subset \phi_0(q_{0})$ or $\phi_0(q_{0}) \subset \phi_0(q_{k+1}) $, and they differ from each other by a dilation of factor $\rho(\sigma) = \prod_{i=0}^{k} \lambda_i$.

\begin{figure}[h]
\centering
\includegraphics[width=0.5\textwidth]{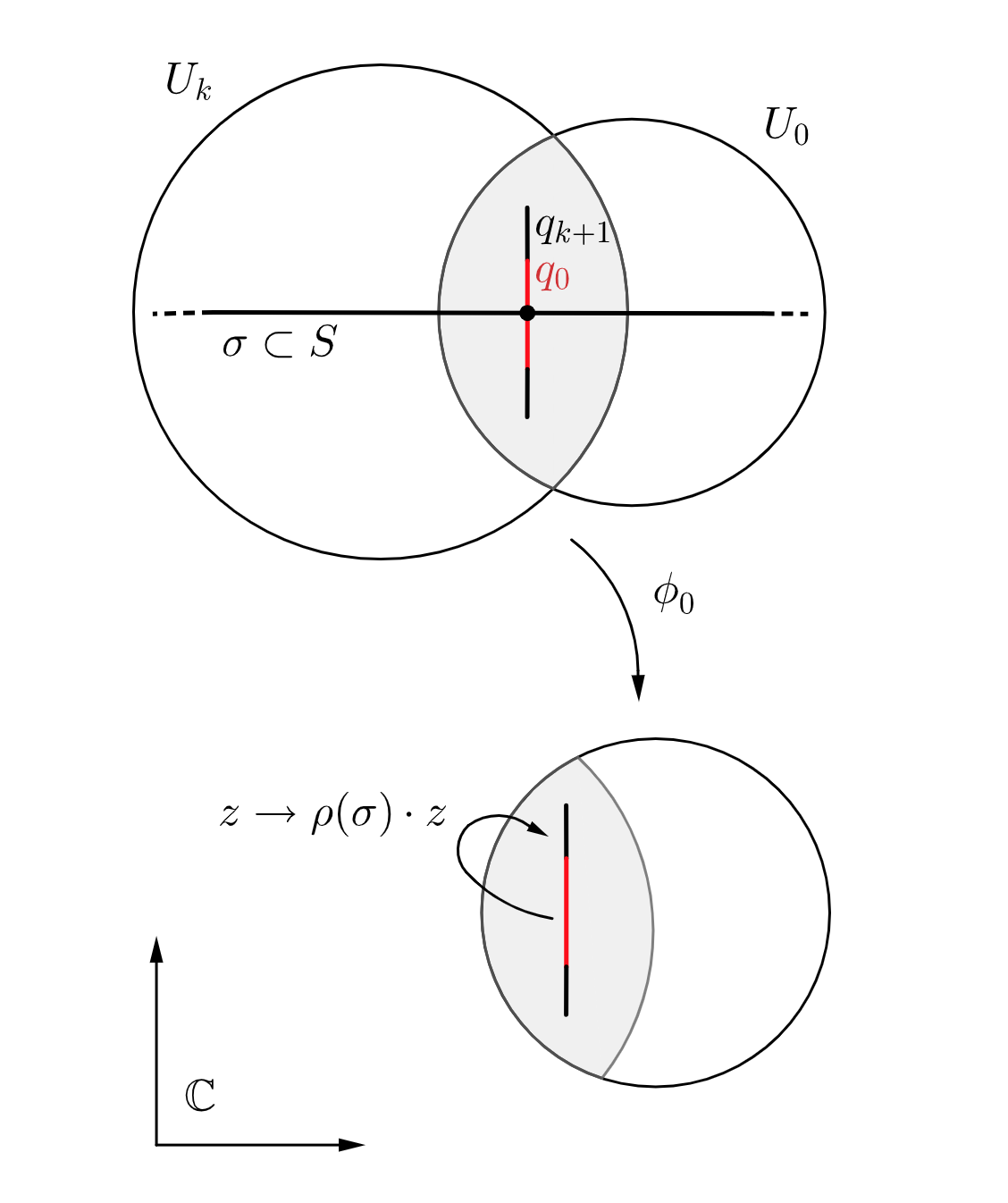}
\caption{}
\end{figure}

Note that the parallel transport of the segment $\phi_0(q_0)$ along the image of $\sigma$ in the charts corresponds exactly to moving $q_0$ along the foliation in direction $\theta$ on the surface. Consider hence the first return map $f: q_0 \rightarrow q_0$. If $q_{k+1} \subset q_0$, then the first return map $f$ is defined on all of $q_0$. If $q_0 \subset q_{k+1}$, then $f$ is only defined on a subinterval of $q_0$ that contains $\sigma(t_0)$. In both cases, $f$ is of the form $z \rightarrow \rho(\sigma) \cdot z$ when viewed in a chart that contains $q_0$. \end{proof}

Proposition \ref{prop_lin_hol} tells us that the linear holonomy of a closed leaf $l$ on a dilation surface, as we defined it using the transition maps between charts that cover $l$, really describes the local affine structure of the surface around the leaf. This means that if we begin to move a segment in forward direction along $l$ starting at a point $p \in l$, then by the time we reach $p$ again, the segment has been contracted or dilated by factor $\lambda$ with respect to any chart at $p$. This is equivalent to saying that the the leaves of the directional foliation $\mathcal{F}_{\theta}$ near $l$ are either "attracted" or "repelled" by $l$. We will discuss this behaviour more formally in the next chapter when we introduce \textit{flat} and \textit{affine cylinders}. 

The discussion of the linear holonomy concludes this chapter and we have hence mentioned all of the basic properties of dilation surfaces that will be important for us. In the next chapter, we will see an in-depth description of the possible types of behaviours for directional foliations on dilation surfaces.

\section{Different Types of Recurrence on Dilation Surfaces} 

In this chapter, we give explicit examples of directional foliations on dilation surfaces. We state the definition of a \textit{recurrent leaf} and we consider four different types of recurrent behaviour, starting with foliations where the leaves are trivially recurrent and concluding with foliations where the leaves are non-trivially recurrent. In our fourth example, we discuss in detail the \textit{Cantor-like} behaviour that arises on the Disco surface for some directions of the foliation. Our main theorem, proved in Chapter 6, will then assert that the types of behaviour we see in this chapter are already all possible types of behaviours for the directional foliation on a dilation surface. 

\subsection{Recurrence} Loosely speaking, a leaf is recurrent if it keeps coming back to any neighbourhood of any point that lies on it. To study the recurrent behaviour of leaves, we want to differentiate between leaves that are \textit{trivially recurrent} and leaves that are \textit{non-trivially recurrent}. For this, we first define the notion of a limit set for the forward and backward direction of a leaf.
\begin{defs} Let $S$ be a dilation surface with directional foliation $\mathcal{F}_{\theta}$. Let $p \in S$ and $l_p$ be the leaf through $p$. We denote by $l_p^+$ the \textit{forward half-leaf} consisting of all points reached starting from $p$ and travelling in the forward direction along $l_p$. We denote by $l_p^-$ the \textit{backward half-leaf} consisting of all points reached starting from $p$ and travelling in the backward direction along $l_p$.
\end{defs}
\begin{defs}
     The \textit{$\omega-$ limit set} of $p$ is the limit set of $l_p^+$ and we denote it by $L_{\omega}(p)$. The \textit{$\alpha-$ limit set} of $p$ is the limit set of $l_p^-$ and we denote it by $L_{\alpha}(p)$. 
\end{defs}
(For the definition of a limit set see Definition \ref{deflimitset}). The $\omega$-limit set of a point $p$ is hence simply the set of points that the leaf through $p$ accumulates to in the future, its $\alpha$-limit set is the set of points that the leaf through $p$ accumulates to in the past. In line with \cite{gardiner}, we define a leaf that contains its own limit set both in the future and in the past to be a \textit{recurrent leaf}. More formally,

\begin{defs}  A point $p \in M$ is said to be \textit{$\omega$-recurrent} if $p$ lies in $L_{\omega}(p)$, \textit{$\alpha$-recurrent} if $p$ lies in $L_{\alpha}(p)$ and \textit{recurrent} if $p$ is both $\omega$- and $\alpha$-recurrent.
\end{defs}

Note that if $p$ is recurrent, then any point of the leaf $l_p$ through $p$ will also be recurrent since limit sets are invariant. Hence, recurrence is a property of a whole leaf, not only a single point. A recurrent leaf is thus a leaf that intersects any neighbourhood of any point that lies on it, both in the future and in the past. There are two types of recurrent leaves:

\begin{defs} A recurrent leaf that is closed is called \textit{trivially recurrent}. All other recurrent leaves are called \textit{non-trivially recurrent}. The topological closure of a non-trivially recurrent leaf is called a \textit{non-trivially recurrent leaf closure}.
\end{defs}

In the remainder of this chapter, we give two examples of trivially recurrent behaviour and two examples of non-trivially recurrent behaviour for directional foliations on dilation surfaces. For trivially recurrent behaviour, we introduce \textit{flat} and \textit{affine cylinders}. For non-trivially recurrent behaviour, we introduce \textit{minimal} and \textit{Cantor-like subsurfaces}. We give an overview of all four cases below. \\

\begin{center}
\includegraphics[width=0.93\textwidth]{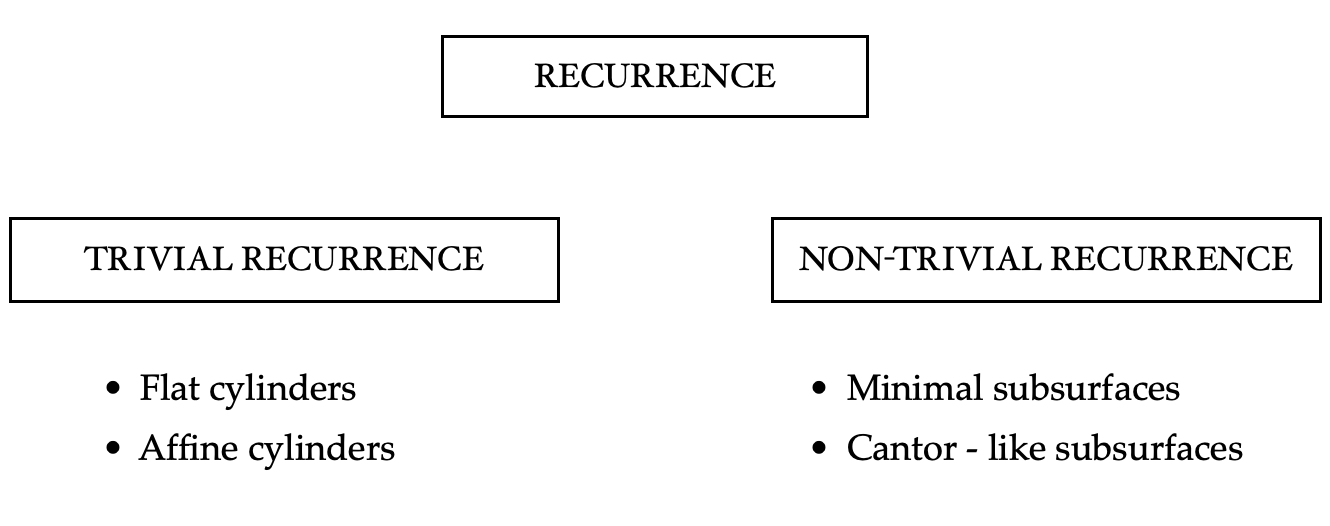}
\end{center}

\subsection{Trivial recurrence} Trivial recurrence on dilation surfaces can be split in two cases: the case where the trivially recurrent leaf is contained in a flat cylinder and the case where the trivially recurrent leaf is contained in an affine cylinder. We first consider the case of flat cylinders. 

 
\begin{defs} Let $\theta \in S^1$. A \textit{flat cylinder} $C_{\theta}$ is the structure we receive when glueing two parallel, distinct lines of the same length in the plane that are orthogonal to the direction $\theta$ using a translation only.\end{defs}
\begin{figure}[h]
\centering
\includegraphics[width=0.6\textwidth]{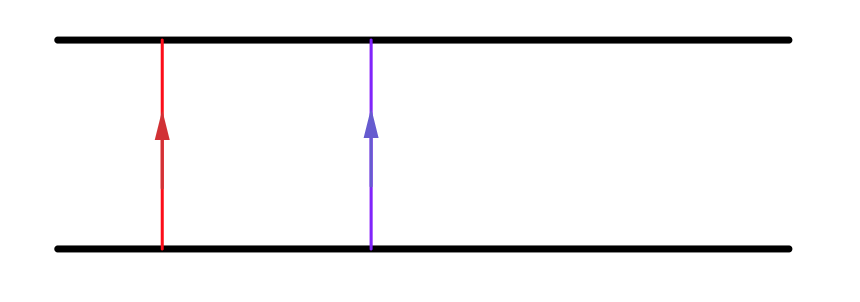}
\caption{A flat cylinder and two vertical closed leaves.}
\end{figure}
Because the sides of the cylinder are glued with translations only, any leaf of the directional foliation $\mathcal{F}_{\theta}$ on $C_{\theta}$ is closed for $\theta \in S^1$. Note that the linear holonomy of these closed leaves is equal to zero. We say that a dilation surface \textit{contains a flat cylinder} if there exists an embedding of this cylinder into the surface that preserves the affine structure. If the surface only consists of flat cylinders whose sides are all parallel, we call the corresponding foliation \textit{completely periodic}:

\begin{defs} We say that a foliation $\mathcal{F}_{\theta}$ on a dilation surface is \textit{completely periodic} if the surface can be completely decomposed into flat cylinders $C_{\theta}^1, \dots C_{\theta}^n$. \end{defs}

We now move from flat cylinders to affine cylinders. While flat cylinders can only be contained in dilation surfaces that are translation surfaces, affine cylinders can only be found on dilation surfaces that are not translation surfaces. Affine cylinders considerably enrich the dynamics of the directional foliation on dilation surfaces, as the closed leaves contained in them have non-trivial linear holonomy and hence act as attractors or repellers. 

\begin{defs} An \textit{affine cylinder} $C_{\alpha,\rho}$ is the dilation structure we receive when gluing an angular sector of the complex plane of angle $\alpha$ along the arcs of two concentric circles using the map $z \rightarrow \rho z + b \in \hspace{1mm}$Aff$_{\mathbb{R}_{+}^*}(\mathbb{C})$, where $\rho > 1$. \end{defs}
The construction is illustrated in Figure 18. Again we say that a dilation surface \textit{contains an affine cylinder} if there exists an embedding of an affine cylinder into the surface that preserves the affine structure.

\begin{figure}[h]
\centering
\includegraphics[width=0.7\textwidth]{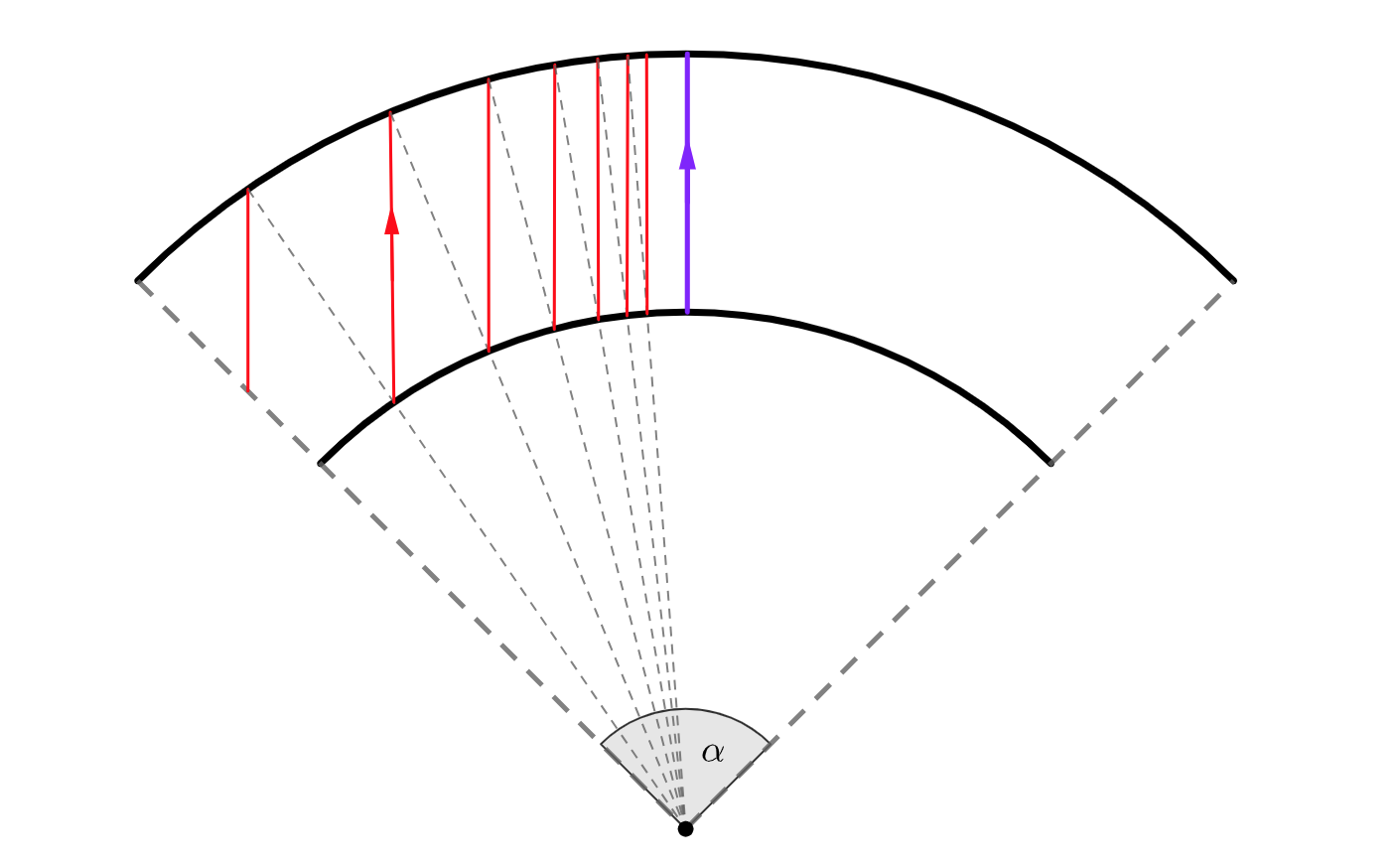}
\caption{Two leaves of the vertical foliation on an affine cylinder of angle $\alpha$.}
\end{figure}

Consider a leaf of a directional foliation $\mathcal{F}_{\theta}$ entering an affine cylinder $C_{\alpha,\rho}$ where $\theta$ lies in the angular sector covered by the cylinder. Then this leaf will be trapped inside the cylinder. Indeed, there exists a closed leaf inside the affine cylinder - simply draw a straight line in the given direction starting from the origin, then the line will project to a closed leaf on the affine cylinder (the closed leaf for the vertical foliation in Figure 18 is drawn in lilac). Note that as the edges of the cylinder are glued with a dilation, this leaf has non-trivial linear holonomy. Now fix a point $p$ on this closed leaf and consider a small transversal segment $l$ passing through $p$. The first return map $f$ on this segment is equal to the division by $\rho$ and hence a contraction whose fixed point is exactly $p$. Hence, the $\omega$-limit set of any leaf that intersects $l$ is equal to the closed leaf. As we can extend the segment $l$ to a cross-section of the whole affine cylinder, this is true for any leaf entering the cylinder. We call this closed leaf an \textit{attracting leaf}.
Note that if we had considered the foliation $\mathcal{F}_{-\theta}$, then the $\alpha$-limit set of any leaf in the cylinder would be equal to the same closed leaf, hence in this case the closed leaf would be called a \textit{repelling leaf}. There is a special name for directional foliations that only show this type of attracting or repelling behaviour:

\begin{defs} A directional foliation is called \textit{Morse-Smale} if there exist a finite number of closed leaves and the $\omega$-and $\alpha$-limit set of every regular leaf is a closed leaf. \end{defs}

In fact, this behaviour is believed to be almost always the case for a generic direction on a dilation surface that is not a translation surface. So far, no counterexample has been found to the following conjecture:

\begin{conj}[S. Ghazouani, \cite{ghazouani}]\label{selimconjecture}
For any dilation surface $S$ which is not a translation surface, for a full measure set of directions in $S^1$, the directional foliation on $S$ is Morse-Smale.
\end{conj}

An example of a dilation surface with a directional foliation that is Morse-Smale is the two-chamber surface, drawn in Figure 19. Consider the two affine cylinders embedded in the surface that are obtained by glueing the red and the yellow edges (note that if $\alpha < \pi$, we can also glue an angular sector of the complex plane of angle $\alpha$ along two straight lines instead of concentric circles to obtain an affine cylinder).
\begin{figure}[h]
\centering
\includegraphics[width=0.55\textwidth]{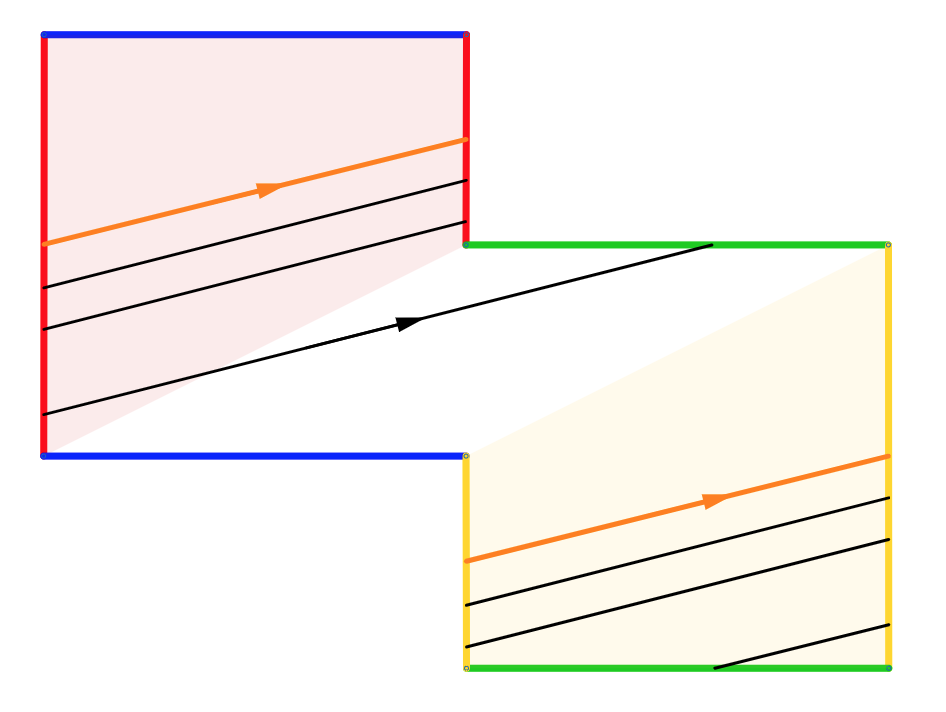}
\caption{Leaves of the foliation in direction $\theta = $arctan$(\frac{1}{4})$.}
\end{figure}
For any direction that lies in the angular sector of the red cylinder, there exist two closed leaves, one of them in the red cylinder, the other one in the yellow cylinder. Since the yellow cylinder is simply the red cylinder rotated by 180°, the linear holonomy of their closed orbits is exactly reciprocal to each other. Hence, any regular leaf of the foliation is attracted by one and repelled by the other leaf, meaning that the foliation is Morse-Smale.

\begin{rem} The directional foliation on this surface can also be completely periodic for $\theta = \pi/2$, or it can even exhibit \textit{Cantor-like} behaviour, explained more precisely after the next section. For a detailed and well-written study of the possible dynamical behaviour of the two-chamber surface, or more precisely of \textit{dilation tori with boundary} out of which we can construct the two-chamber surface, we refer the reader to \cite{dilationtori}. 
\end{rem}

Note that the dynamics of the directional foliation on affine cylinders is trivially recurrent, since the only recurrent leaves are closed leaves. Thus, together with the case of flat cylinders, we have now seen two different examples of trivially recurrent behaviour on dilation surfaces. We want to conclude this section by showing that in fact any closed leaf on a dilation surface is contained in either a flat or affine cylinder. 

\begin{prop}
    Let $l$ be a closed leaf on a dilation surface with directional foliation $\mathcal{F}_{\theta}$. If $\rho(l)=1$, then $l$ is contained in a flat cylinder, if $\rho(l) \neq 1$, then $l$ is contained in an affine cylinder of the surface. 
\end{prop}
\begin{proof}
    This follows almost entirely from Proposition \ref{prop_lin_hol} which asserts that if we have a closed leaf $l$, then there exists a segment $q_0$ orthogonal to $\mathcal{F}_{\theta}$ whose midpoint lies on $l$ such that the first return map $f:q_0 \rightarrow q_0$, possibly defined only on a subinterval of $q_0$ that contains the midpoint of $q_0$, is of the form $z \rightarrow \rho(l) \cdot z$. If $\rho(l) = 1$, then this is equivalent to saying that $l$ is contained in a flat cylinder, if $\rho(l) \neq 1$ this is equivalent to saying that $l$ is contained in an affine cylinder. 
\end{proof}

\subsection{Triangulations} At this point, we want to insert a brief comment on the existence of polygonal representations of dilation surfaces using the definition of affine cylinders. The reader is invited to verify that a dilation surface has a polygonal representation where the vertices project to the singularities of the surface if and only if it has a triangulation where the set of vertices is exactly the set of singularities of the surface and where the sides of the triangles are straight lines. Below we state a theorem from William A. Veech, as it is formulated in \cite{affinesurfacesandveechgroups}, that gives a necessary and sufficient condition for a dilation surface to have such a triangulation. 
\begin{thm}[Veech, see \cite{affinesurfacesandveechgroups}]
    A dilation surface has a triangulation if and only if it does not contain an affine cylinder of angle greater than $\pi$.
\end{thm}
\begin{proof} It is easy to see that if there is a cylinder of angle greater than $\pi$, then there is no such triangulation of the surface where the sides are straight lines, as any straight line that enters the cylinder will never leave the cylinder again (note that for any direction $\theta \in S^1$, there exists a closed leaf in direction $\theta$ inside the cylinder). The other direction is more involved and can be found in the appendix of \cite{affinesurfacesandveechgroups}. 
\end{proof}
An example of such a surface that has no triangulation is the Hopf Torus, a genus one dilation surface without singularities. It is obtained by taking the quotient of $\mathbb{C}\backslash\{0\}$ by an affine map of the form $z \rightarrow \lambda z$, where $\lambda \in \mathbb{R}_{>0}$. 

\begin{figure}[h]
\centering
\includegraphics[width=0.3\textwidth]{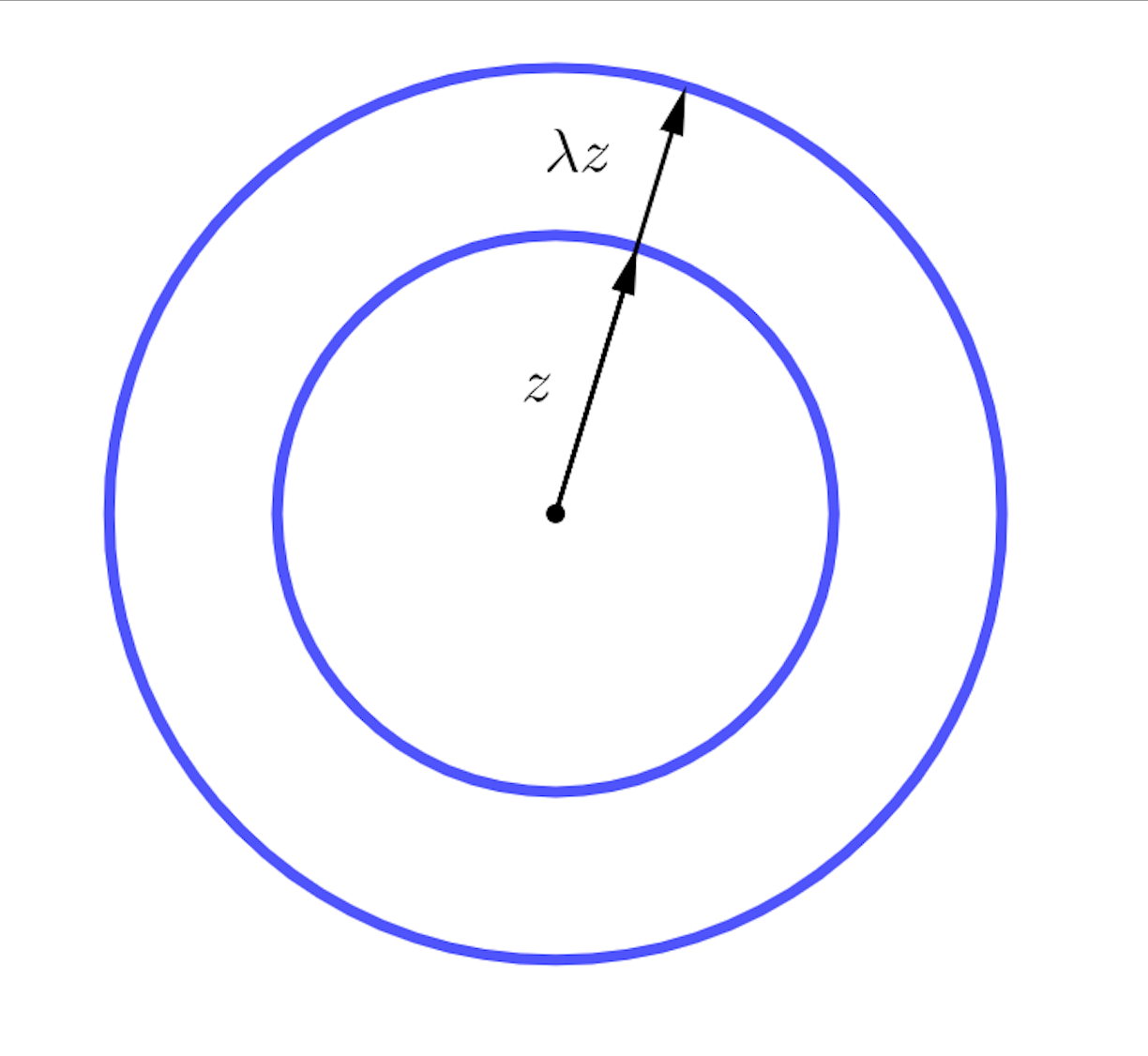}
\caption{The Hopf Torus.}
\end{figure}

\begin{rem} This surface is also an example of a dilation surface where every directional foliation is Morse-Smale. For any direction $\theta \in S^1$, the reader is invited to verify that there exist two closed, diametrically opposed leaves whose linear holonomy is non-trivial and reciprocal to each other, meaning that one is attracting and the other one repelling. 
\end{rem}

We now proceed with our discussion of recurrence on dilation surfaces. In the next section, we focus on non-trivial recurrence on dilation surfaces. 

\subsection{Non-trivial recurrence and the Disco surface} We want to give two different examples of directional foliations that exhibit non-trivial recurrence on dilation surfaces: \textit{minimal} foliations and \textit{Cantor-like} foliations. We first discuss the easier case of minimal foliations that can also be found on dilation surfaces that are translation surfaces. Consider directional foliation $\mathcal{F}_{\theta}$ on the Torus, where the angle $\theta \in S^1$ is irrational. It is well known that in this case, any regular leaf will never be periodic but instead be dense on the torus. By definition, a leaf that is dense on the surface is a non-trivially recurrent leaf, note that here the non-trivially recurrent leaf closure is simply the whole torus.
\begin{defs}
    A directional foliation on a dilation surface is called \textit{minimal} if every regular leaf is dense on the surface. We further call a dilation surface together with a minimal foliation on the surface a \textit{minimal dilation surface}. 
\end{defs}

As in the case of flat and affine cylinders, there is also a case of non-trivial recurrence that only arises on dilation surfaces that are not translation surfaces. In this case, the foliation accumulates on a set whose cross-section is a Cantor set, which is why this case is called Cantor-like. In the remainder of this section, we explicitely construct an example for this case using the \textit{Disco surface.}. The Disco surface is a genus two dilation surface with two singularities of angle $4\pi$, obtained by glueing the polygon in Figure 21 along edges of the same color. In the figure, vertices of the same color project to the same singularity. The figure further shows the four affine cylinders that give the Disco surface its name. 
 
\begin{figure}[h]
\centering
\includegraphics[width=0.75\textwidth]{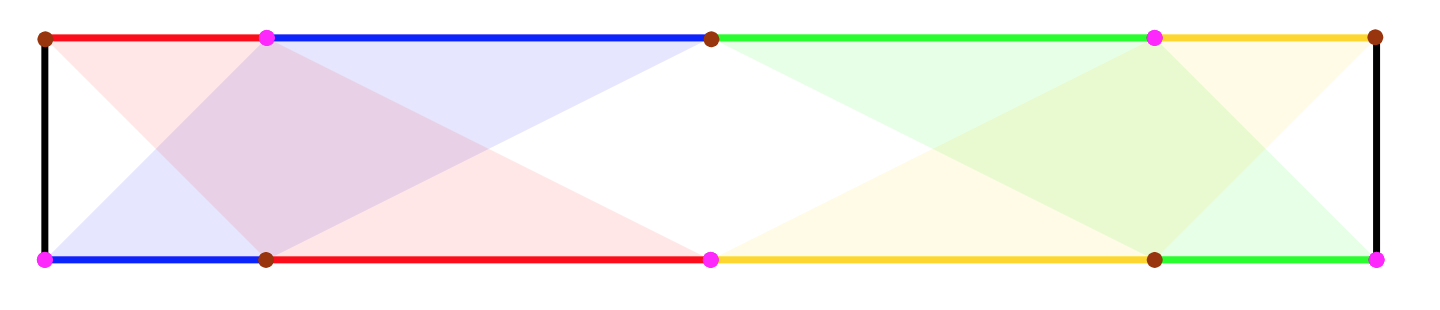}
\caption{Four affine cylinders embedded in the Disco surface.}
\end{figure}

The Disco surface is a particularly interesting example of a dilation surface as its directional foliations can be in any of the three cases we previously mentioned, however there are also some directions for which the foliation accumulates on a Cantor set. We want to geometrically visualize these directions and to do so, we modify the polygonal representation of the Disco surface using two so-called "cut and paste" operations. In a first step, these operations involve cutting the polygon into two pieces along a straight line. The two new edges that are created this way are then identified with each other, hence they are assigned the same color. In a second step, the two pieces of the polygon are then glued back together, this time along a different pair of edges that shares the same color. The operations that we use are visualized in Figure 22 below.

\begin{figure}[h]
\centering
\includegraphics[width=1.1\textwidth]{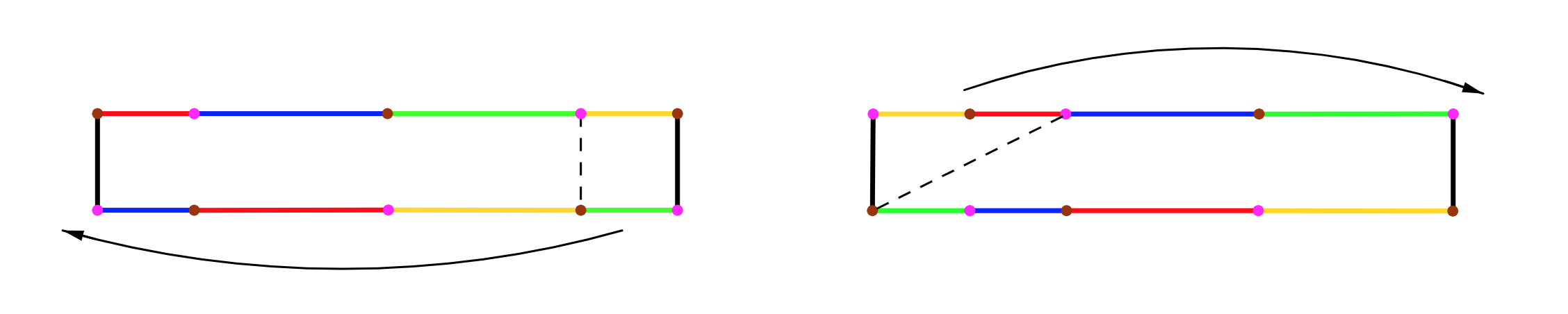}
\caption{The two "cut and paste" operations.}
\end{figure}

Note that for both operations in Figure 22, we choose the line along which we cut in such a way that the operation is continuous. This means that we do change the polygonal representation, but the surface obtained from glueing the polygon remains the Disco surface. 

\begin{figure}[h]
\centering
\includegraphics[width=0.85\textwidth]{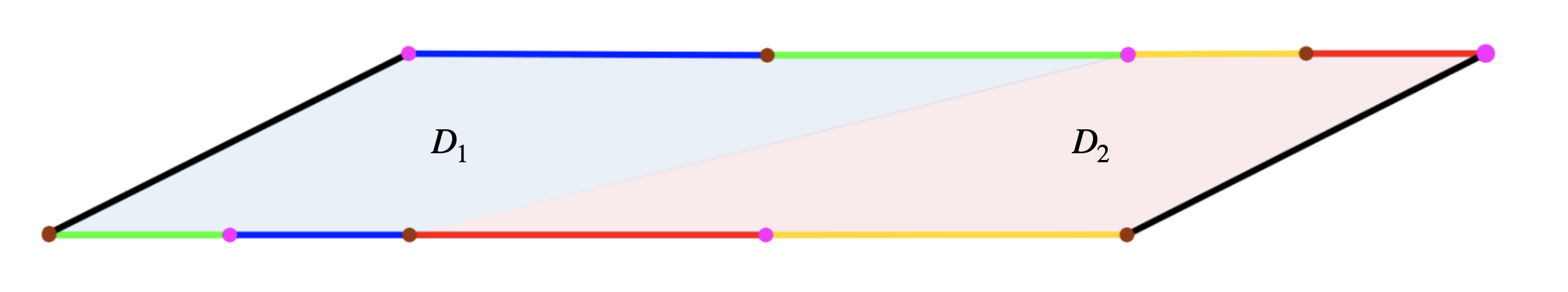}
\caption{The new polygonal representation of the Disco surface.}
\end{figure}

The new polygonal representation of the Disco surface makes two subsurfaces $D_1$ and $D_2$ appear, colored in Figure 23 in blue and in red. Let $R \subset S^1$ denote the set of directions that lie between the two skewed lines that bound $D_1$ to the left and right. Choose $\theta \in R$ and note that any leaf of $\mathcal{F}_{\theta}$ that enters $D_1$ will stay trapped in it thereafter and never leave $D_1$ again. Similarly, any leaf of $\mathcal{F}_{-\theta}$ that enters $D_2$ will stay trapped in it, as the red subsurface is just the blue subsurface rotated by 180°. Denote by $L$ the closed interval formed by the green and blue segment on the bottom of the polygonal representation of the Disco surface as illustrated in Figure 24. Draw two parallel lines in direction $\theta$ starting at the right and the left endpoint of $L$, call one of these lines $l$ as illustrated in Figure 24. Do the same in the red subsurface but rotate every action by 180°. We color the complement in $D_1$, respectively $D_2$, of the region bounded by the parallel lines in direction $\theta$ in dark blue, respectively dark red.

\begin{figure}[H]
\centering
\includegraphics[width=0.9\textwidth]{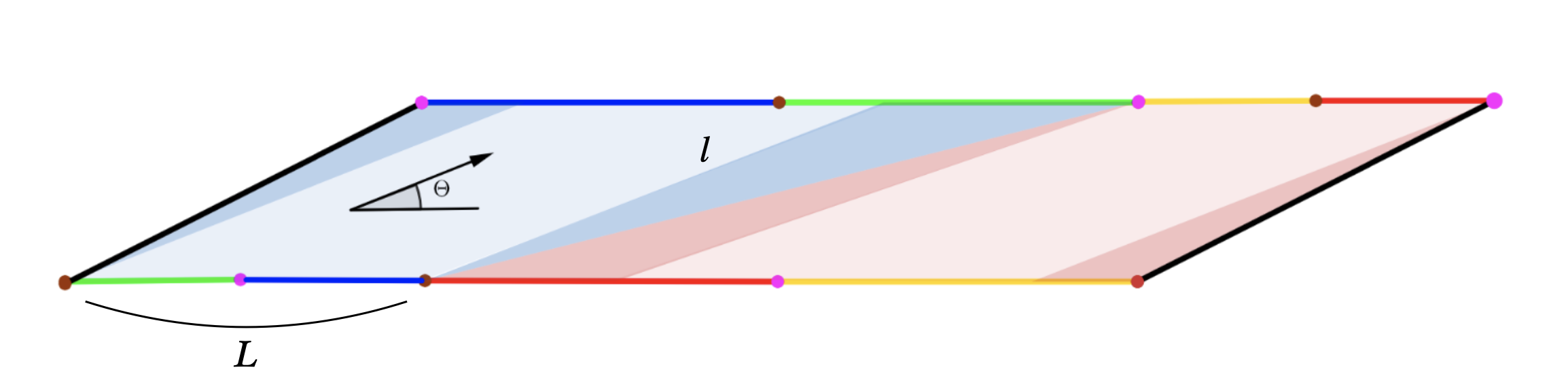}
\caption{}
\end{figure}

In a next step, we begin to "flow" all of the dark blue area in direction $\theta$. This means we travel in the forward direction along the leaves of $\mathcal{F}_{\theta}$ contained in the dark blue area and we color all the points that we reach on the way in dark blue. After one iteration, the initial dark blue area will join together to form a strip in the middle of the blue subsurface, as shown in Figure 25. We repeat exactly the same process for the dark red area, where we travel in the backward direction along the leaves of $\mathcal{F}_{\theta}$ contained in the dark red area. If we repeat this process infinitely many times, the dark blue, respectively the dark red strip will continue winding around the blue, respectively red subsurface. Denote by $I \subset L$ the interior of the intersection between $L$ and the dark blue strip. 

\begin{figure}[H]
\centering
\includegraphics[width=0.92\textwidth]{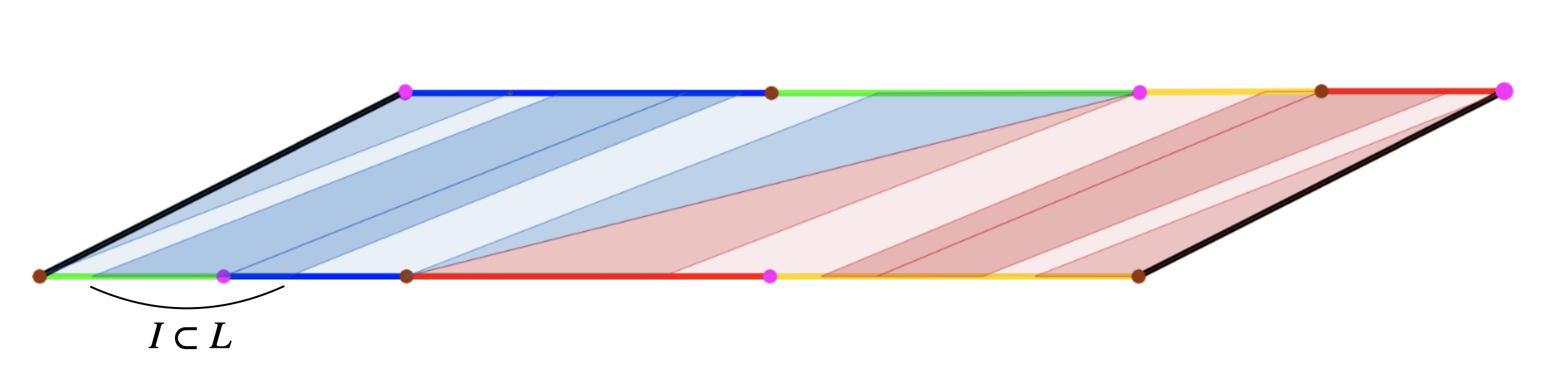}
\caption{The first iteration.}
\end{figure} 

\begin{prop}\label{prop_disco_surface}
Let $\Sigma$ denote the set of singularities of the Disco surface $D$, let $\theta \in S^1$, let $f$ be the first return map on $L$ with respect to $\mathcal{F}_{\theta}$. Assume that $f^n(I) \cap \Sigma = \emptyset$ for all $n \in \mathbb{N}_{>0}$. Then the set 
$$\Omega := L-\bigcup_{n=0}^{\infty} f^n(I)$$
is a Cantor Set and for all $p \in L$, the set of accumulation points of $(f^n(p))_{n \in \mathbb{N}_{>0}}$ is equal to $\Omega$.
\end{prop} 

\begin{proof} $\Omega$ is a closed set as $I$, and hence also $f^n(I)$, is open for all $n \in \mathbb{N}_{>0}$ and a countable union of open sets is open. Moreover, note that since the edges of the Disco Surface are glued with dilation factor $2$, we have $\lambda(I) = \frac{1}{2} \cdot \lambda(L)$, where $\lambda$ denotes the Lebesgue-measure in the complex plane. Furthermore, $\lambda(f^n(I)) = \frac{\lambda(I)}{2^n}$ as every iteration of $I$ divides its length by two. Moreover, $f^n(I) \cap f^m(I) = \emptyset$ for $n,m \in \mathbb{N}_{>0}$ where $n < m$ ; if there was some $x \in f^n(I) \cap f^m(I)$, then $f^{-n}(x) \in I \cap f^{m-n}(I)$ which is impossible as $I$ is disjoint from $f(L)$. Thus, 

$$ \lambda(\Omega) = \lambda(L)-\lambda(L)\cdot \sum_{n=0}^{\infty} \frac{1}{2^n} = \lambda(L) - \lambda(L) = 0$$

Therefore, $\Omega$ is totally disconnected, as any connected component of $\Omega$ other than a singleton would have positive measure. Moreover, $\Omega$ has no isolated points: Indeed, if there was an isolated point, then this would mean that for some $m,n \in \mathbb{N}_{>0}$ we have that $f^m(I)$ and $f^n(I)$ are two open sets "right next to each other", i.e their closures intersect in exactly one point $p'$. Then also the closures of $f^{m-n}(I)$ and $I$ intersect in  $f^{-n}(p')$ and hence the closure of $f^{m-n-2}(I)$ contains the singularity $f^{-(n+2)}(p')$. But this is only possible if $I$ at some point "hits" a singularity, otherwise there will always be some open ball around the singularity that does not intersect $f^{m-n-2}(I)$. Hence, $\Omega$ is a Cantor set. Note further that any leaf that intersects $I$ will accumulate to $\Omega$: pick a point $x \in \Omega$, then in any neighbourhood of $x$ there is a sequence of points in $\Omega$ that accumulates to $x$. Take two of these points, then the segment between them contains some image of $I$ and hence any leaf that intersects $I$ also passes through the neighbourhood of this point. So any leaf in $I$ is attracted by $\Omega$.
\end{proof}

If $\Omega \subset L$ has these properties, then the product $C_{\theta}^+ := \Omega \times l$ (where $l$ is the right skewed line in direction $\theta$ in Figure 24) has the following properties: \vspace{2mm}

    \begin{itemize}
    \item $C_{\theta}^+$ is a set of straight lines in direction $\theta$ whose intersection with $L$ is a Cantor Set. 
    \item $L_{\omega}(p) = C_{\theta}^+$ for all $p \in D_1$.
    \end{itemize}
    \vspace{2mm}
    
Note that the set $C_{\theta}^+$ is a non-trivially recurrent leaf closure for $\mathcal{F}_{\theta}$ by definition since it is closed and for any $p \in D_1$, hence also for $p \in C_{\theta}^+$, we have that $L_{\omega}(p)$ is equal to $C_{\theta}^+$ and thus $p \in L_{\omega}(p)$. Furthermore, the foliation $\mathcal{F}_{\theta}$ on $D_1$ behaves exactly in the same way as the foliation $\mathcal{F}_{-\theta}$ on $D_2$, hence there exists a set $C_{\theta}^-$ in the red subsurface $D_2$ with the same properties as $C_{\theta}^+$, except that it is a repelling non-trivially recurrent leaf closure, meaning that $L_{\alpha}(p) = C_{\theta}^-$ for all $p \in D_2$. Moreover, the "holes" of the Cantor-like set $C_{\theta}^-$ are obtained by iterating $I$ backwards in direction $-\theta$, in the same way that the "holes" of $C_{\theta}^+$ are obtained by iterating $I$ in direction $\theta$. Hence, the Disco surface can be fully decomposed into two non-trivially recurrent leaf closures $C_{\theta}^+$ and $C_{\theta}^-$ whose cross-section is a Cantor set and their complement $D - C_{\theta}^+ - C_{\theta}^-$ which consists of one connected component that winds around the surface, made up of all the leaves that intersect $I$. Any regular leaf in $D - C_{\theta}^+ - C_{\theta}^-$ is attracted by $C_{\theta}^+$ and repelled by $C_{\theta}^-$. 

\begin{rem} The fact that there exist directions $\theta \in S^1$ for which $\mathcal{F}_{\theta}$ on the Disco surface satisfies the assumptions of Proposition \ref{prop_disco_surface} follows from the detailed study of the Disco surface presented in \cite{cascades}. In this paper, the authors show that there exists a Cantor set of directions $\theta \in S^1$ for which the corresponding foliation on the Disco surface has a non-trivially recurrent leaf closure whose cross-section is a Cantor set. In order to find these directions, they use a procedure called \textit{Rauzy-Veech-induction}. This procedure associates to any directional foliation on the Disco surface a word in the alphabet $\{ L,R \}$. The authors then show that the words that are infinite as well as not eventually constant correspond to the directions with a non-trivially recurrent leaf. In the appendix, we give a detailed explanation why these directions correspond exactly to the directions which satisfy the assumptions of Proposition \ref{prop_disco_surface}. \end{rem}

We conclude with the definition of Cantor-like foliations. In the case of the Disco surface there are two distinct non-trivially recurrent leaf closures. We call a foliation Cantor-like if it only has one such closure and no closed leaves, an example would be the foliation $\mathcal{F}_{\theta}$ on the Disco surface restricted to the sub-surface $D_1$, where $\theta$ satisfies the assumptions in Proposition \ref{prop_disco_surface}.

\begin{defs} A directional foliation on a dilation surface is called \textit{Cantor-like} if it does not contain a closed leaf and if there exists a unique non-trivially recurrent leaf closure $\Omega$ such that the intersection of $\Omega$ and any transversal segment is either empty or a Cantor set.
\end{defs}

Hence we have discussed four different types of behaviours for the directional foliation on a dilation surface: completely periodic, Morse-Smale, minimal and Cantor-like. The main theorem of this thesis will show that these four types of directional foliations are the only ones that arise on dilation surfaces. 
A key ingredient in the proof of this statement is Gardiner's decomposition theorem, which will be the main focus of the next chapter. 


\section{Gardiner's Decomposition Theorem}\label{chapter3}

In this chapter, we state and explain Gardiner's decomposition theorem which will be of key importance for the proof of our main theorem in Chapter 6. The statement and proof of the decomposition theorem was published by C.J Gardiner in the \textit{Journal of Differential Equations} in 1985 (see \cite{gardiner}). It asserts that given a foliation or flow with finitely many singularities on a surface $M$, we can decompose $M$ into finitely many subsurfaces that contain at most one non-trivially recurrent leaf closure. To be able to state the theorem in full, we want to introduce the notion of an \textit{irreducible foliation}, as defined for flows in \cite{gardiner}: 
\begin{defs} We call a foliation on a surface $M$ \textit{irreducible} if the surface contains a unique non-trivially recurrent leaf closure that meets every homotopically nontrivial curve on $M$ in at least one non-trivially $\alpha-$ or $\omega-$recurrent point. 
\end{defs}
Hence, if the foliation is irreducible, it is not possible to separate $M$ further into two subsurfaces using a homotopically non-trivial curve such that one of the subsurfaces contains the entire non-trivially recurrent leaf closure. We can also define irreducibility for an open submanifold $N$ of $M$:
\begin{defs}Let $M$ be a surface and $\mathcal{F}$ a foliation on $M$. Let $N$ be an open submanifold of $M$, let $N^*$ be the manifold obtained by compactifying each end of $N$ with a point. There is a foliation $\mathcal{F}|_N$ on $M$ whose leaves are the leaves of $M$ intersected with $N$. Let $\mathcal{F}|_{N^*}$ be the foliation on $N^*$ obtained from $\mathcal{F}|_{N}$ by making each point of $N^*\backslash N$ a singular point. The foliation $\mathcal{F}|_N$ is called \textit{irreducible} if the frontier of $N$ contains no non-trivially $\alpha$- or $\omega$-recurrent point of $\mathcal{F}$ and if the foliation $\mathcal{F}|_{N^*}$ is irreducible. 
\end{defs}
\subsection{Statement of the theorem} As Gardiner's decomposition theorem was originally stated for flows on surfaces, we will reformulate it using the following notion:

\begin{defs} We say that a foliation $\mathcal{F}$ on a surface $M$ \textit{admits a continuous flow with finitely many singularities} if there exists a continuous flow $\phi$ on $M$ with finitely many singular points where the leaves of $\mathcal{F}$ are exactly the orbits of $\phi$.

\end{defs}
We now have all the background material necessary to state Gardiner's decomposition theorem (see also \cite{gardiner}, page 152). 

\begin{thm} [Gardiner's Decomposition Theorem] Let $\mathcal{F}$ be a foliation on a closed surface $M$ that admits a continuous flow with finitely many singularities. There is a finite set $\mathcal{C}$ of homotopically nontrivial closed curves on $M$ such that 
\begin{enumerate}
    \item no curve of $\mathcal{C}$ contains a non-trivially $\alpha$ or $\omega$- recurrent point of $\mathcal{F}$. 
    \item if $M_i, 1 \leq i \leq n$, are the components of $M \backslash \bigcup_{C \in \mathcal{C}}C$ then, for each $i$, either $\mathcal{F}|_{M_i}$ is irreducible or $M_i$ contains no non-trivially recurrent point of $\mathcal{F}$. 
\end{enumerate}
\end{thm}
In simpler terms, Gardiner's decomposition theorem asserts that given a foliation that admits a continuous flow with finitely many singularities on a closed surface $M$, we can "cut" $M$ into subsurfaces using finitely many closed curves such that any subsurface contains at most one non-trivially recurrent leaf closure. The irreducibility criterion tells us that we cannot further decompose the components that contain a non-trivially recurrent leaf closure. 

We claim that we can continuously deform the curves in $\mathcal{C}$ to obtain a decomposition that satisfies (1) and (2).

\begin{prop} Let $\mathcal{F}$ be a foliation on a closed surface $M$, let $\mathcal{C}$ be a set of homotopically non-trivial closed curves on $M$ that satisfy (1) and (2) of Gardiner's decomposition theorem. Let $M_i, 1 \leq i \leq n$ be the components of $M \backslash \bigcup_{C \in \mathcal{C}}C$ and let $\mathcal{C}'$ be obtained from $\mathcal{C}$ by continuously deforming the curves in $\mathcal{C}$. Then $\mathcal{C}'$ still satisfies (1) and (2) of Gardiner's decomposition theorem. 
\end{prop}
 \begin{proof} Let $M_i'$ be obtained from $M_i$ by continuously deforming one of its boundary curves $C \in \mathcal{C}$ without hitting a singularity or passing through a non-trivially $\alpha-$ or $\omega$ recurrent point. Note that then $M_i'$ still contains the same unique non-trivially recurrent leaf closure as $M_i$. Let $C'$ be the deformed curve. Assume that $\mathcal{F}|_{M_i'}$ is not irreducible anymore. Then there exists a homotopically nontrivial curve $c$ on $(M_i')^{*}$, the surface obtained from $M_i'$ by compactifying each of its ends with a point, such that $c$ does not intersect the unique non-trivially recurrent leaf closure of $M_i$ in a non-trivially $\alpha-$ or $\omega-$ recurrent point. If $M_i' \subset M_i$ then this would contradict the irreducibility of $\mathcal{F}|_{M_i}$. Now consider a connected component $\tilde{c}$ of the intersection between $c$ and $M_i'-\overline{M_i}$. Because $c$ is homotopically nontrivial on $(M_i')^{*}$, it cannot be entirely contained in $M_i' - \overline{M_i}$ and hence $\tilde{c}$ has to intersect $C$ at its endpoints $p,q$. Since $C'$ is obtained by a continuous deformation of $C$, we can continuously deform $\tilde{c}$ so that it is contained in the segment of $C$ between $p$ and $q$. 
\begin{figure}[h]
\centering
\includegraphics[width=0.8\textwidth]{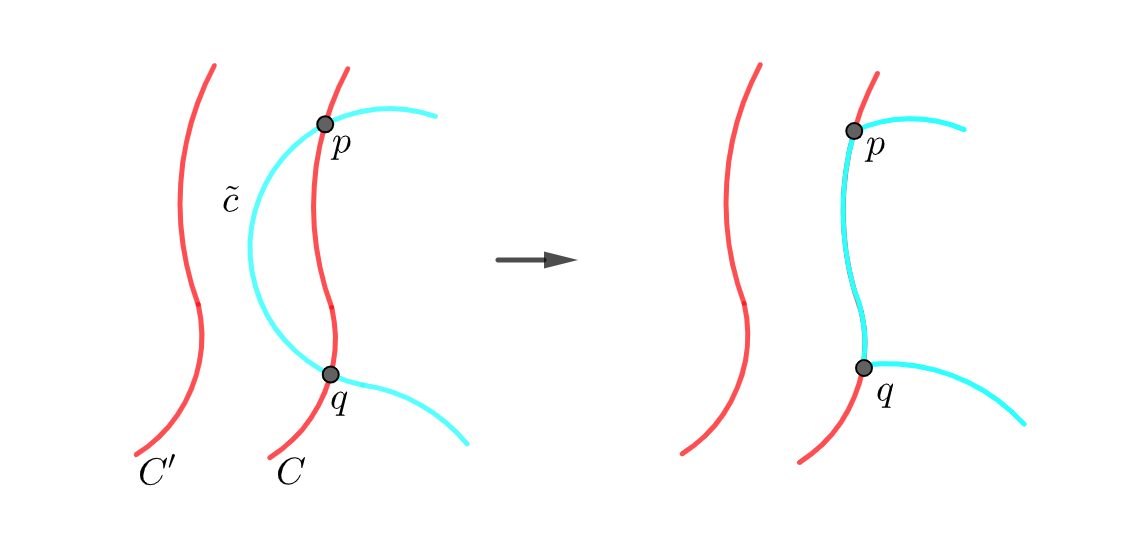}
\caption{}
\end{figure}
In this way, we obtain a new curve homotopically equivalent to $c$ which is entirely contained in $M_i \cup C$ and homotopically nontrivial when viewed as a curve in $M_i^{*}$ (where we contract all elements in $\tilde{c} \cap C$ to a point). Indeed, if it was homotopically trivial then also $c$ would be homotopically trivial in $(M_i')^{*}$. This implies that also $\mathcal{F}|_{M_i'}$ is irreducible. Hence we have shown that when we continuously deform the curves in $\mathcal{C}$, the new set of curves still gives us a Gardiner decomposition of $M$. \end{proof}
In order to use Gardiner's decomposition theorem for the proof of our main theorem, we want to show that the directional foliation on a dilation surface satisfies the assumptions of the decomposition theorem. 
\begin{prop}\label{foliationadmitsflow}
    Let $S$ be a dilation surface with directional foliation $\mathcal{F}_{\theta}$. Then $\mathcal{F}_{\theta}$ admits a continuous flow with finitely many singularities. 
\end{prop}
\begin{proof} In the following, we use the definition of a flow via the \textit{intergral curves} of a vector field (see \cite{Camacho}, p. 28). We assign to each point $p \in S \backslash \Sigma$ a vector in the tangent space to $S$ at $p$, such that this vector has length one and points in the direction of the forward leaf through $p$. This defines a smooth, differentiable vector field with finitely many singular points. The \textit{integral curves} of this vector field then define a continuous flow on $S$ whose orbits are exactly the leaves of $\mathcal{F}$. 
    \end{proof}
\subsection{Gardiner decomposition of the Disco surface} 
We want to give an example of the Gardiner decomposition on the Disco surface. In the last chapter, we have seen examples of directions for which the directional foliation on the Disco surface contains two non-trivially recurrent leaf closures. 

\begin{prop} Let $\mathcal{F}_{\theta}$ be a directional foliation on the Disco surface $D$ such that $\theta \in S^1$ satisfies the assumptions of Proposition \ref{prop_disco_surface}. Then the Gardiner decomposition with respect to $\mathcal{F}_{\theta}$ is the closed curve $c$ that separates the blue and red subsurfaces $D_1$ and $D_2$. 
\begin{figure}[H]
\centering
\includegraphics[width=0.8\textwidth]{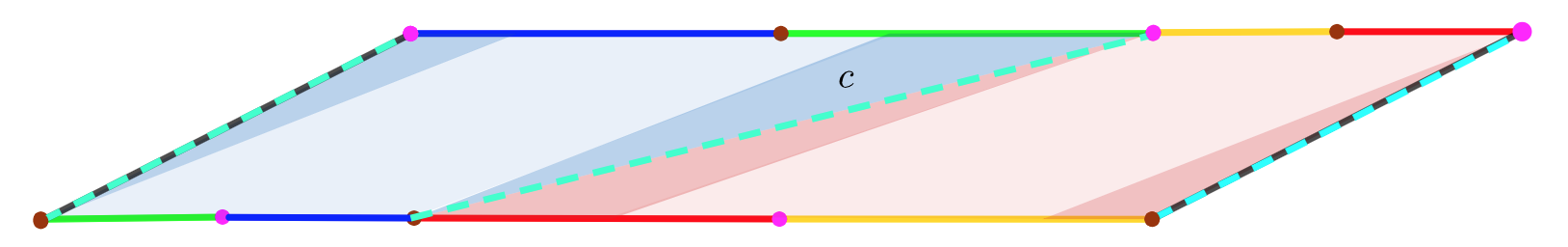}
\caption{The curve $c$ cuts the Disco Surface into two genus one sub-surfaces.}
\end{figure}
\end{prop}
\begin{proof} The two non-trivially recurrent leaf closures $C_{\theta}^+$ and $C_{\theta}^-$ are fully contained in the light blue, respectively light red region. The curve $c$ only intersects this region at the singularities of the Disco surface $D$, hence $c$ does not contain a non-trivially  $\alpha$- or $\omega$-recurrent point of $\mathcal{F}_{\theta}$. This proves the first part.  For the second part, we want to show that $\mathcal{F}_{\theta}|_{D_1}$, and hence $\mathcal{F}_{\theta}|_{D_2}$, are irreducible. This is clear as any closed curve on $D_1$ will have to intersect some image of $I$ (see Proposition \ref{prop_disco_surface}) and after this point it is contained in the iterates of this image as it cannot cross $\Omega$, however these iterates are all disjoint so the curve cannot be closed.
\end{proof}

Hence, we are now equipped with Gardiner's decomposition theorem, our most important tool to prove our main theorem. The next chapter is dedicated fully to our main theorem and its corollary for affine interval exchange maps.

\section{Structure Theorem for Foliations on Dilation Surfaces}
\label{chapter4}

\subsection{Statement of theorem} In this chapter, we prove our main theorem and show its application to affine interval exchange maps. Given a directional foliation on a dilation surface, our main theorem allows us to decompose the surface into different subsurfaces that only exhibit one type of dynamical behaviour. Furthermore, we can characterize the foliations on these subsurfaces by semi-conjugating the first return map on a transversal segment to an IET. We restate the theorem below:

\begin{namedthm*}{Theorem \ref{maintheorem}} Given a directional foliation $\mathcal{F}_{\theta}$ on any dilation surface $S$, there exists a decomposition of $S$ into subsurfaces that either have no recurrent leaf or are in one of the following cases:
\begin{enumerate}
\item Flat cylinders where the foliation is completely periodic,
\item Affine cylinders where the foliation is Morse-Smale,
\item Minimal subsurfaces where the foliation is minimal,
\item Subsurfaces where the foliation is Cantor-like.
\end{enumerate}
In case (3) and (4), the first return map on any finite union of segments transversal to $\mathcal{F}_{\theta}$ that intersects a non-trivially recurrent leaf is semi-conjugated to a minimal IET. 
\end{namedthm*}


The next three sections are dedicated to the proof of this theorem. The main idea behind the proof is to use Gardiner's decomposition theorem to decompose a given dilation surface into subsurfaces on which the foliation is either trivially or non-trivially recurrent. Before we proceed to the proof, we first want to discuss these two types of subsurfaces. 

\subsection{Subsurfaces with trivial recurrence} Let $S_i$ be a subsurface of a dilation surface $S$ that contains only trivially recurrent leaves, then either there exists at least one closed leaf or no recurrent leaf at all. We have seen in Chapter 4 that that any closed leaf is contained either in a flat or in an affine cylinder, depending on the linear holonomy of the closed leaf. There are only finitely many such cylinders in $S_i$ by the compactness of $S$. Hence, we can decompose $S_i$ further into finitely many flat cylinders (1) or affine cylinders (2) or components that have no recurrent leaf at all.  

\subsection{Subsurfaces with nontrivial recurrence} We now consider the case where $S_i$ is a subsurface that contains a unique non-trivially recurrent leaf closure.

\begin{prop}\label{mainproposition} Let $S_i$ be a subsurface of a dilation surface with directional foliation $\mathcal{F}_{\theta}$ that contains a unique non-trivially recurrent leaf closure $\Omega = \overline{\gamma}$ where $\gamma$ is a non-trivially recurrent leaf. Let $L$ be a finite union of segments transversal to $\mathcal{F}_{\theta}$ that intersects $\gamma$. Then the first return map $f:L \rightarrow L$ is semi-conjugated to a minimal IET. Furthermore, $L\hspace{0.05cm}\cap\hspace{0.05cm}\Omega$ is either a Cantor set or a finite union of closed intervals. \end{prop}

\begin{rem} The general idea behind the proof of this proposition originates from the proof of a structure theorem for continuous flows established by Carlos Gutierrez in \cite{gutierrez} (see 3.3 to 3.9). In order to apply his ideas onto the case of dilation surfaces, Lemma \ref{lem6.5}, \ref{prop6.7}, \ref{lem6.9} and \ref{lem6.10} have been established by the author of this thesis independently. \end{rem}

\noindent \underline{Outline of the proof.} We first construct a quotient space $\mathcal{H}(L)$ from $L$ by collapsing intervals in the complement of $\Omega \cap L$ to a point. We then explicitly construct a map $f_{L}: \mathcal{H}(L) \rightarrow \mathcal{H}(L)$ that is semi-conjugated to $f:L \rightarrow L$ via the quotient map. We then show that the map $f_L$ satisfies the assumptions of the following key lemma, stated as Lemma 3.8 in \cite{gutierrezkeylemma}. We have attached a proof of this lemma in the appendix. 

\begin{lem}[Key Lemma]\label{keylemma} Let $I$ be an interval and $T: I \rightarrow I$ be a continuous injective map defined everywhere except possibly at finitely many points. If $T$ has a dense positive semi-orbit, then $T$ is conjugated to an IET.
\end{lem}

Using this lemma, we can show that $f_L$ is conjugated to a minimal IET and hence $f$ is semi-conjugated to a minimal IET. A simple argument using the non-trivial recurrence of $\gamma$ will further give us that $L\hspace{0.05cm}\cap\hspace{0.05cm}\Omega$ is either a Cantor set or a finite union of closed intervals. 

To set up for the proof of Proposition \ref{mainproposition}, we want to introduce the following definition. 

\begin{defs}\label{maindef} For $a,b \in L$ we say that $a,b$ \textit{belong to the same interval of continuity} if $a,b$ belong to the same connected component of $L$ and $f|_{[a,b]}$ is continuous (with respect to the topology induced by the complex plane via the atlas). 
\end{defs}

We define the set $\mathcal{H}(L)$ as the set of all intervals $[a,b] \in L$ such that $[a,b]$ is the closure of a connected component of $L - \Omega$ or $a=b$ and $a$ does not belong to the closure of any connected component of $L - \Omega$. We claim that because $\gamma$ is non-trivially recurrent, this set forms a partition of $L_i$. Indeed, two intervals $[a,b],[c,d] \in \mathcal{H}(L)$ cannot intersect in a point, as this would mean that $\gamma \cap L$ has an isolated point. We denote by $h:L \rightarrow \mathcal{H}(L)$ the quotient map. Note that as $\mathcal{H}(L)$ is a nontrivial partition of $L$ it inherits a quotient topology from $L$. We claim that while $\gamma$ accumulates at the endpoints of each interval closure in $L - \Omega$, it will never actually intersect such an interval. 

\begin{lem}\label{lem6.5} For $[a,b] \in \mathcal{H}(L)$ where $a\neq b$ we have that $\gamma \cap [a,b] = \emptyset$. \end{lem}

\begin{proof} Fist of all note that $\gamma$ can only intersect $[a,b]$ at the endpoints by definition of $\mathcal{H}(L)$. W.l.o.g assume that $\{ a \} \in \gamma \cap [a,b]$. As $\gamma$ is non-trivially recurrent, there exist $(a_k)_{k\in \mathbb{N}} \in L \cap \gamma$ which accumulate to $a$ from the left as $k \rightarrow \infty$. (if $a$ is a left endpoint of a connected component of $L$, then simply join a small transversal segment to this left endpoint of $L$ to obtain an extension $L'$ and consider $(a_k)_{k\in \mathbb{N}} \in L' \cap \gamma$). 
Choose a chart $(U_1,\phi_1)$ at $a$ and define the distance between $a,b \in U_1$ to be the Lebesgue-distance between $\phi_1(a)$ and $\phi_1(b)$. For small $d>0$, let $a_d \in \cup_{k=1}^{\infty} a_k$ with distance at most $d$ from $a$. 
\begin{figure}[H]
\centering
\includegraphics[width=0.6\textwidth]{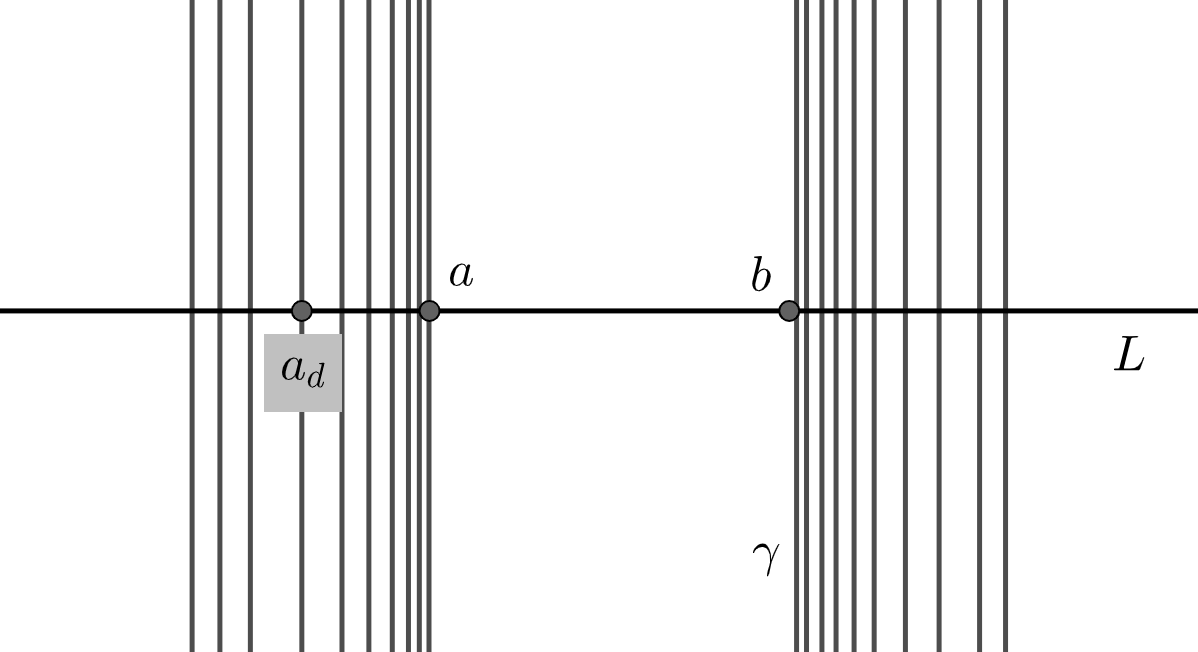}
\caption{}
\end{figure}
Now $a_d$ and $a$ both belong to $\gamma$, hence we can assume that there exists $n \in \mathbb{N}$ such that $f^n(a_d)=a$ (the case when $n \in \mathbb{Z}$ follows analogously). By the regularity of $\gamma$, we can choose $d$ small enough such that there is a cover of the arc of $\gamma$ between $a_d$ and $a$ with charts $(U_1, \phi_1), \dots, (U_l, \phi_l)$ where $a_d,a \in U_1$ and $f^n(a_d),f^n(a) \in U_l$ and the intersection of the forward half-leaves through $[a_d,a]$ and the backward half-leaves through $[f^n(a_d),f^n(a)]$ is entirely covered by the charts. Note that for small $d$, $f^n(a)$ is to the right of $f^n(a_d) = a$ because orientation is preserved in the plane under $\mathcal{F}_{\theta}$ and because the leaves through $a_d$ and $a$ never "split up" due to a singularity. 

From the proof of Proposition \ref{prop_lin_hol} in Chapter 3 we can deduce that the Lebesgue-distance  between the images of $f^n(a_d)=a$ and $f^n(a)$ with respect to $\phi_l$ is equal to $\lambda \cdot d$, where $\lambda$ is the product of the affine factors of the transition maps between $(U_1, \phi_1), \dots, (U_l, \phi_l)$. Hence, by letting $d \rightarrow 0$, we have that $a$ is also accumulated on the right by points in $\gamma \cap L$ which is a contradiction as $a\neq b$, i.e $a$ should be isolated from the right. 
\end{proof}


We now define a map $f_{L}: \mathcal{H}(L) \rightarrow \mathcal{H}(L)$ that is semi-conjugated to $f:L \rightarrow L$. The idea is to define the image of a point in $\mathcal{H}(L)$ that was obtained from collapsing an interval of $L$ as the point in $\mathcal{H}(L)$ that is obtained from collapsing the image of the first return map on this interval to a point. More formally:

\begin{defs} Let $f:L \rightarrow L$ be the first return map induced by $\mathcal{F}_{\theta}$. If $\{ a \} \in \mathcal{H}(L)$ is such that $\gamma$ contains ${a}$, then $f(a)$ is well-defined by the non-trivial recurrence of $\gamma$ and we set
$$ f_{L}(\{ a \}) = \{ f(a) \}$$
If $[a,b], [c,d] \in \mathcal{H}(L)$ such that $[a,b] \cap \gamma = \emptyset$, then we define
$$ f_{L}([a,b]) = [c,d]$$ 
provided there exist sequences $(p_n)_{n \in \mathbb{N}}, (q_n)_{n \in \mathbb{N}}$  contained in $L \hspace{0.05cm}\cap \hspace{0.05cm} \gamma$ such that
\begin{enumerate}[label = \roman*), itemsep=0pt]
    \item $\lim_{n\rightarrow \infty} p_n = a, \hspace{0.05cm}\lim_{n\rightarrow \infty}q_n = b$,
    \item $\lim_{n \rightarrow \infty} f(p_n) = c, \lim_{n\rightarrow \infty} f(q_n) = d$,
     \item $p_n, q_n$ belong to the same interval of continuity $\forall n \in \mathbb{N}$.
\end{enumerate}
\end{defs}
Note that the fact that $f_{L}$ is continuous and injective on its domain of definition follows directly from the fact that $f$ is continuous and injective on its domain of definition. 
\begin{prop}\label{prop6.7} There are only finitely many points in $\mathcal{H}(L)$ where $f_{L}$ is not defined. 
\end{prop} 
\begin{proof} In fact, whenever $\gamma$ intersects $[a,b] \in \mathcal{H}(L)$ it is clear that $f_{L}([a,b] )$ is well-defined. If $[a,b] \cap \gamma = \emptyset$, then either there exist sequences $(p_n)_{n \in \mathbb{N}}, (q_n)_{n \in \mathbb{N}}$ on $L \hspace{0.05cm}\cap \hspace{0.05cm} \gamma$ which accumulate to $[a,b]$ from the right and the left or $[a,b]$ is an interval at the end of a connected component of $L$. In the first case, $f_{L}([a,b])$ fails to be well defined if and only if there is at least one singularity of the surface contained in the strip made up of the intersection of all the forward half-leaves that start at at point in $[a,b]$ and all the backward half-leaves that start at at a point in $f([a,b])$. There are only finitely many such singularities and hence only finitely many intervals $[a,b]$ for which this is the case. In the second case, $[a,b]$ is accumulated by elements in $L \hspace{0.05cm}\cap \hspace{0.05cm} \gamma$ only from one side and thus $f_{L}$ is not well defined. As $L$ is a finite union of intervals, this case also only arises for finitely many elements in $\mathcal{H}(L)$. \end{proof}

\begin{lem}\label{lem6.8} $\mathcal{H}(L)$ is homeomorphic to $L$. \end{lem}
\begin{proof}[Proof (see also \cite{gutierrez}).] This is clear if $L \hspace{0.5mm} \cap \hspace{0.5mm} \Omega$ contains a subinterval of $L$, where $\Omega$ is the non-trivially recurrent leaf closure. If $L \cap \Omega$ is a Cantor set, then consider the Cantor function $\mathcal{L}: L_i \rightarrow L_i$ which is a monotone continuous map of degree one. The map is constant on a closed subinterval of $L$ if and only if this interval is the closure of a connected component of $L - \Omega$. Then the quotient space $L / \mathcal{L}$ is homeomorphic to $L$ and $L / \mathcal{L}$ is precisely $\mathcal{H}(L)$. 
\end{proof}

\begin{lem}\label{lem6.9} $f_{L}$ has a dense positive semi-orbit.
\end{lem}
\begin{proof}
Let $\sigma \in \mathcal{H}(L)$ be such that $\sigma$ is contained in $\gamma$. Then $(f_{L})^n(\sigma)$ is defined in $\sigma$ for all $n \in \mathbb{N}$. Let us show that with respect to the quotient topology on $\mathcal{H}(L)$ we also have that $\{ (f_{L})^n(\sigma)|n \in \mathbb{N} \}$ is a semi-orbit dense in $\mathcal{H}(L)$. Let us pick $[a,b] \in \mathcal{H}(L)$ with $\gamma \cap [a,b] = \emptyset$. Then the pre-image under the quotient map $h$ of any open ball with respect to the quotient topology around $[a,b]$ is an open interval of $L$ containing $[a,b]$. W.l.o.g let $\gamma$ accumulate at $a$, then as in the proof of Lemma 6.4, let $a_d$ be a point in $L \cap \gamma$ whose Lebesgue-distance from $a$ is smaller than $d$ with respect to a chart $(U_1,\phi_1)$ that contains $a$ and $a_d$. Furthermore, the forward half-leaf of $\gamma$ starting at $\sigma$ accumulates at $a_d$ by the non-trivial recurrence of $\gamma$, so for any $\tilde{d} > 0$ we can find $k \in \mathbb{N}$ such that $f^k(\sigma) \in U_1$ and $f^k(\sigma)$ is $\tilde{d}$-close to $a_d$ and hence at most $d+\tilde{d}$-close to $a$ with respect to the Lebesgue-distance in the chart $(U_1, \phi_1)$. Thus, any open interval containing $[a,b]$ will have non-empty intersection with the positive half-leaf of $\gamma$ starting at $\sigma$ and hence the corresponding open set in the quotient topology will intersect the trajectory $(f_{L})^n(\sigma)$ for some $n \in \mathbb{N}$. Therefore, $f_{L}$ has a dense positive semi-orbit. \end{proof}

\begin{lem}\label{lem6.10} $L \cap \Omega$ is either a Cantor set or a finite union of closed intervals. \end{lem}
\begin{proof} Assume  $L \cap \Omega$ contains an open sub-segment $l$. This means that $\gamma \cap l$ is dense in $l$. If $a \in \gamma \cap l$ and $b \in \gamma \cap L$, then there exists $n \in \mathbb{N}$ such that $f^n(a) = b$. We can cover the arc of $\gamma$ between $a$ and $f^n(a) = b$ by charts $(U_1,\phi_1), \dots, (U_l, \phi_l)$ and choose a neighbourhood $B_\epsilon(a)$ of $a$ in $l$ such that $B_\epsilon(a) \in U_1$, $f^n(B_\epsilon(a)) \in U_l$ and the forward half-leaves that start at the points in $B_\epsilon(a)$ intersected with the backward half-leaves that start at the points in $f^n(B_\epsilon(a))$ are entirely covered by the charts (again this is possible by Proposition \ref{prop_lin_hol} in Chapter 3). Then by the continuity of $f^n$ on $B_\epsilon(a)$, as $B_\epsilon(a)$ belongs to the closure of $\gamma \cap L$, also $f^n(B_\epsilon(a))$ belongs to the closure of $\gamma \cap L$. Hence, for any $b \in \gamma \cap L$ we can find a neighbourhood that belongs to the closure of $\gamma \cap L$, meaning that $L \cap \Omega$ is a finite union of closed intervals. If on the contrary $L \cap \Omega$ does not contain an open sub-segment, then $L \cap \Omega$ is totally disconnected, closed, nonempty and has no isolated points, meaning that it is a Cantor set. \end{proof}

\begin{proof}[Proof of Proposition \ref{mainproposition}] Lemma \ref{prop6.7} - \ref{lem6.9} show us that $f_{L}$ satisfies the assumptions of the key lemma and thus, by the key lemma, $f_{L}$ is conjugate to an interval exchange transformation $E$ which has a dense positive semi-orbit. It then follows from \cite{keane} that any orbit of $E$ is either finite or dense and that there are only finitely many finite orbits. This is precisely the definition of a minimal interval exchange map. Because the first return map $f:L \rightarrow L$ is semi-conjugated to $f_{L}$ by construction, and because $f_{L}$ is conjugated to $E$, we obtain that $f$ is semi-conjugate to the minimal IET $E$. Lemma \ref{lem6.10} gives us further that $L \cap \Omega$ is either a Cantor set or a finite union of closed intervals.
\end{proof}

\subsection{Proof of theorem} We now proceed to the proof of Theorem 1.1, combining Gardiner's decomposition theorem together with the results we obtained in the previous two sections. 

\begin{proof} [Proof.] Given a dilation surface $S$ and a directional foliation $\mathcal{F}_{\theta}$, by Proposition \ref{foliationadmitsflow} we can apply Gardiner's decomposition theorem and obtain a decomposition of $S$ into finitely many components that contain at most one non-trivially recurrent leaf closure. More precisely, there is a finite set $\mathcal{C}$ of homotopically nontrivial closed curves on $S$ such that 
\begin{enumerate}
    \item no curve of $\mathcal{C}$ contains a non-trivially $\alpha$ or $\omega$- recurrent point of $\mathcal{F}_{\theta}$. 
    \item if $S_i, 1 \leq i \leq n$, are the components of $S \backslash \bigcup_{C \in \mathcal{C}}C$ then, for each $i$, either $\mathcal{F}_{\theta}|_{S_i}$ is irreducible or $S_i$ contains no non-trivially recurrent point of $\mathcal{F}_{\theta}$. 
\end{enumerate}
We have shown in Section 6.2 that if $S_i$ is a component of $S \backslash \bigcup_{C \in \mathcal{C}}C$ that has no non-trivially recurrent leaf closure, we can further decompose $S_i$ into finitely many flat cylinders, affine cylinders or components that have no recurrent leaf at all. The flat and affine cylinders correspond to case (1) and case (2) of our main theorem. If $S_i$ is a component of $S \backslash \bigcup_{C \in \mathcal{C}}C$ that contains a unique non-trivially recurrent leaf closure, note first that by the irreducibility of $\mathcal{F}_{\theta}|_{S_i}$ there are no closed leaves of $\mathcal{F}_{\theta}$ on $S_i$. Furthermore, by Proposition \ref{mainproposition} we know that if $L$ is any finite union of segments transversal to $\mathcal{F}_{\theta}$ that intersect a non-trivially recurrent leaf, then $L \cap \Omega$ is either a finite union of closed intervals or a Cantor set. In the first case, the unique non-trivially recurrent leaf closure forms a subsurface where the foliation is minimal, meaning that every leaf is dense on the subsurface, which corresponds to case (3) of the theorem. In the second case, the foliation on $S_i$ is Cantor-like, which yields case (4) of the theorem. Proposition \ref{mainproposition} also allows us to conclude that the first return map $f:L \rightarrow L$ is semi-conjugated to a minimal IET. 
\end{proof}




\subsection{Structure theorem for AIETs} We now want to show an important application of the main theorem to affine interval exchange transformations. As already mentioned in the introduction, to any AIET $T: X \rightarrow X$ we can associate a dilation surface $S_T$ which is its suspension. This surface is obtained by taking two copies of $X$ and arranging one on the top, the other one on the bottom. We identify the point $x$ on the copy above with the point $T(x)$ on the copy below. We further join the left and right endpoints with two vertical lines that we identify with each other, as illustrated in Figure 29. The object obtained in this way is the polygonal representation of a dilatin surface such that the first return map on any cross-section, for example the lower line $L$ in Figure 29, with respect to the vertical foliation $\mathcal{F}_{\pi/2}$ is exactly $T$. 

\begin{figure}[h]
\centering
\includegraphics[width=0.76\textwidth]{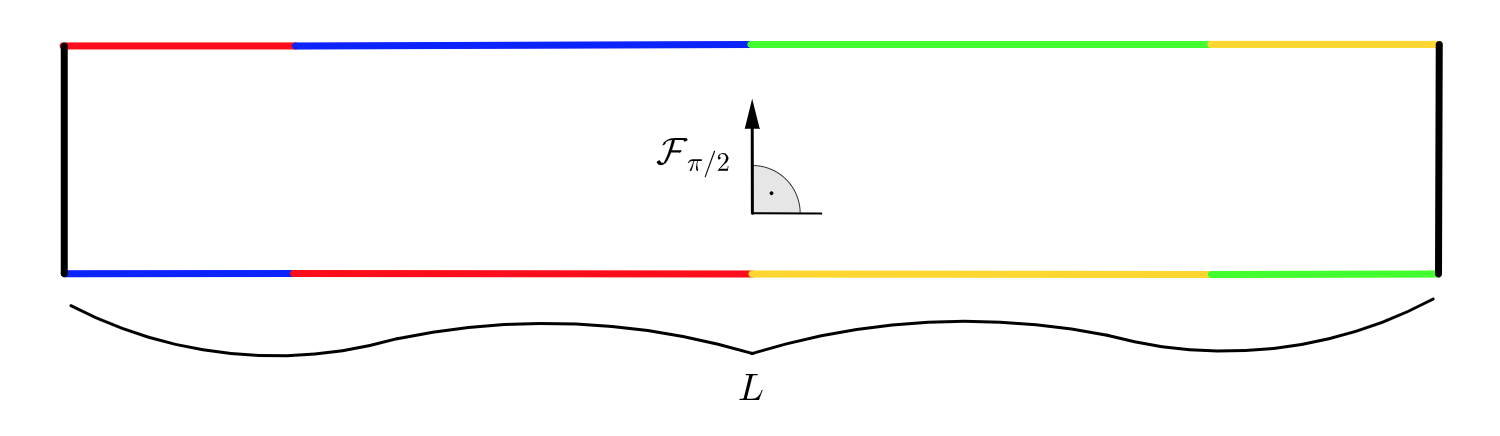}
\caption{}
\end{figure}

For completeness, we provide the following definitions for any map from a topological space $X$ to itself (see also Definitions 2.19-2.21 in Chapter 2). Let $X$ be a topological space, let $f:X \rightarrow X$ be a map from $X$ to itself.

\begin{defs} A \textit{recurrent orbit} of $f$ is an orbit of $f$ that is contained in its own set of accumulation points. A \textit{non-trivially recurrent orbit} is a recurrent orbit that is not periodic. 
\end{defs}

\begin{defs} $f$ is called \textit{completely periodic} if every infinite orbit of $f$ is periodic. \end{defs}

\begin{defs} $f$ is called \textit{Morse-Smale} if there exist a finite number of periodic orbits of $f$ and for any other point $p \in X$ the forward and backward orbit of $p$ accumulate to a periodic orbit.   
\end{defs}

\begin{defs} $f$ is called \textit{Cantor-like} if $f$ has no periodic orbit and if there exists a unique closure of a non-trivially recurrent orbit that is equal to a Cantor set. 
\end{defs}

\begin{defs} Let $X' \subset X$ be a subset of $X$. For $p \in X'$, define $n_p := \inf \{n\in\mathbb{N}_{>0} \hspace{1mm}|\hspace{1mm} f^n(p) \in X' \} $. We define the \textit{first return map} $\tilde{f}: X' \rightarrow X'$ with respect to $T$ as the map $p \rightarrow f^{n_p}(p)$. Note that this map is only defined for $p \in X$ for which $n_p$ is finite. \end{defs}

We are now in the position to state the corollary of our main theorem for affine interval exchange transformations. 

\begin{namedthm*}{Corollary \ref{maincorollary}} Given an affine interval exchange transformation $T:X \rightarrow X$, there exists a decomposition of $X$ into finitely many subsets $L_1, \dots L_n$ such that $L_i$ is a finite union of intervals for $i \in \{1,\dots,n \}$ that either does not intersect a recurrent orbit of $T$ or the first return map $f: L_i \rightarrow L_i$ is in one of the following cases: 

\begin{enumerate} 
\item completely periodic,
\item Morse-Smale,
\item minimal,
\item Cantor like.
\end{enumerate}
In case (3) and (4), $f$ is semi-conjugated to a minimal IET.
\end{namedthm*}

\begin{proof} Given an affine interval exchange transformation $T$ and its suspension $S_T$, we can apply our main theorem for the vertical foliation $\mathcal{F}_{\pi/2}$ to obtain a decomposition of $S_T$ into finitely many subsurfaces. These subsurfaces either have no recurrent leaf, or are flat or affine cylinders, or subsurfaces on which the foliation is minimal, or subsurfaces on which the foliation is Cantor-like. Let $\bigcup_{i=1}^n S_i$ denote the collection of these subsurfaces. From the main theorem we further know that in the case where the foliation is minimal or Cantor-like, the first return map on any finite union of transversal segments such that at least one of them intersects the non-trivially recurrent leaf is semi-conjugated to a minimal IET. 

Let $L$ denote the cross-section of $S_T$ given by the lower copy of $X$ as shown in Figure 29. For each subsurface $S_i$, denote by $L_i$ the intersection between $L$ and $S_i$. Note that $L_i$ is again a finite union of segments as shown in Figure 30. 
\begin{figure}[h]
\centering
\includegraphics[width=0.7\textwidth]{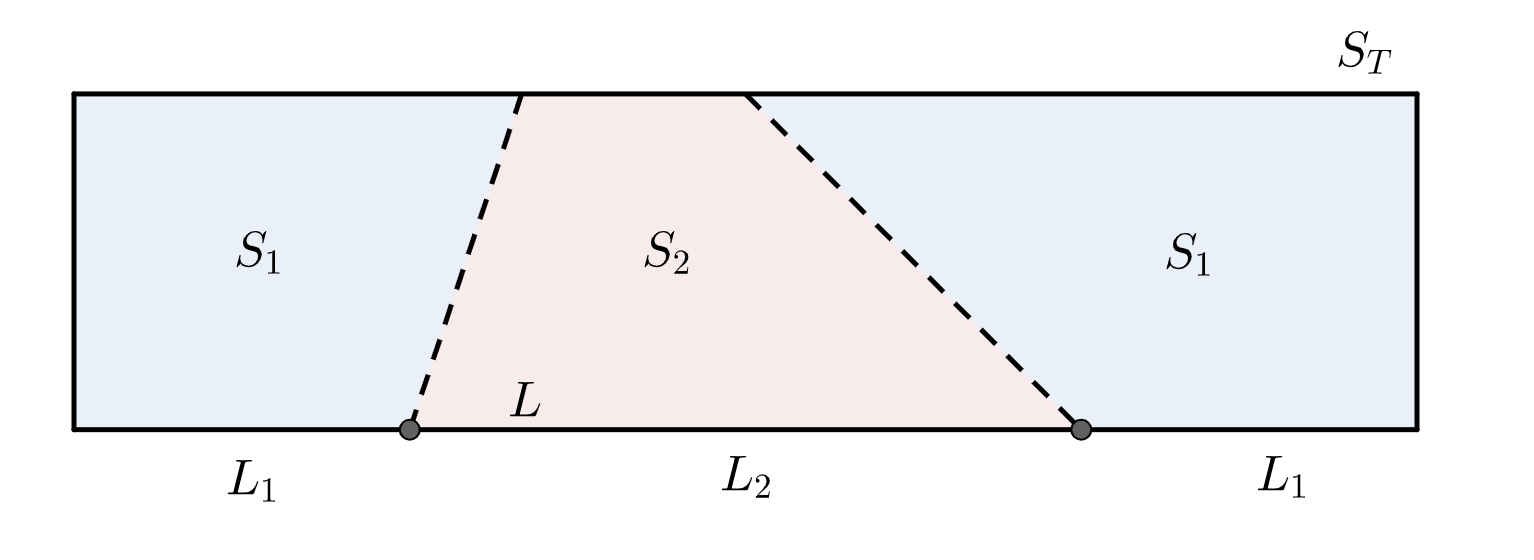}
\caption{An example of a suspension $S_T$ whose decomposition according to the main theorem consists of two subsurfaces $S_1$ and $S_2$. $L_i$ is the intersection of $L$ with $S_i$ for $i=1,2$.}
\end{figure}
Note also that any leaf of $\mathcal{F}_{\pi/2}$ is a collection of vertical lines when viewed in the polygonal representation of $S_T$. Hence, the behaviour of the leaves transfers directly to the behaviour of the orbits of $T$. If the subsurface that contains $L_i$ is a flat cylinder, then every orbit of the first return map $f: L_i \rightarrow L_i$ will be periodic, if the subsurface is an affine cylinder then every orbit will be attracted or repelled by a finite number of periodic orbits. If the vertical foliation on the subsurface is minimal, then also every infinite orbit of $f:L_i \rightarrow L_i$ will be dense on $L_i$ and if it is Cantor-like, then $f$ with respect to $T$ will be Cantor-like. This shows the first part of the theorem. As in the last two cases $L_i$ consists of a finite union of transversal segments and at least one of them intersects the non-trivially recurrent orbit closure, we can also deduce from the main theorem that $f:L_i \rightarrow L_i$ is semi-conjugated to a minimal IET. 
\end{proof}

\subsection{Attracting and repelling Cantor sets} Before we conclude this chapter, we want to comment on the fact that there are different types of Cantor sets that arise as the closure of leaves of foliations in case 4 of our main theorem. By different types of Cantor sets we mean that the sets can exhibit different kinds of attracting or repelling behaviour with respect to nearby leaves. In the following, we want to illustrate this using two examples. Our first example of such a Cantor set is obtained from "blowing up" points of a minimal IET. 

\begin{defs} Let $X$ and $\tilde{X}$ be intervals. Given a minimal IET $T:X \rightarrow X$ and an AIET $f:\tilde{X} \rightarrow \tilde{X}$ that is semi-conjugated to $T$, we say that $f$ has been obtained from $T$ through \textit{blowing up an orbit} if the following are satisfied:
\begin{enumerate} 
\item $\tilde{X}$ has been obtained from $X$ by replacing $T^n(p)$ by an interval $[a_n,b_n]$ for some $p \in X$ that belongs to an infinite orbit, for all $n \in \mathbb{Z}$. 
\item $f[a_n,b_n] = [a_{n+1},b_{n+1}]$ for all $n \in \mathbb{Z}$.
\end{enumerate}
\end{defs}
Note that by definition the forward and backward images with respect to $f$ of $[a_n,b_n]$ never intersect each other. We call such an interval a \textit{wandering interval}. While it is difficult to explicitly construct such "blow-ups", they are in fact quite common. Those familiar with the subject will recall the famous theorem from Marmi, Moussa and Yoccoz that proves that "almost every" IET $T$ admits a semi-conjugated AIET that has been obtained from $T$ through blowing up an orbit (see 3.2 in \cite{MarmiMoussaYoccoz}). 

Consider the wandering interval of an AIET obtained from such a blow up and remove its endpoints. Then the domain of definition of the AIET obtained consists of the future and past images of the open wandering interval and their complement. Since the iterates of this wandering interval never contain a singularity by definition of a blow-up, the complement forms a Cantor set (using the same arguments as in the proof of Proposition \ref{prop_disco_surface}). Hence, the suspension of an AIET obtained from a blow-up consists of a Cantor set as well as a connected component that winds around the whole surface and that does not contain a singularity, made up of the iterates of the wandering interval. Note that because the length of the iterates of the wandering interval has to go to zero both in the future and in the past, the $\omega-$ and $\alpha-$ limit set of any leaf is equal to the unique Cantor set, hence any leaf is both attracted and repelled by the same Cantor set. 

The second example that we want to provide is the Cantor set that arises as the closure of a leaf of a foliation that exhibits Cantor-like behaviour on the Disco surface. We have seen that for such foliations, there are two invariant subsurfaces that each contain an invariant Cantor set. The complement of these sets also consists of the forward and backward iterates of one interval that winds around the whole surface. The future iterates of this interval are eventually all contained in the blue subsurface, the past iterates are eventually all contained in the red subsurface. However, when the interval transitions from one surface to the other, it passes through a singularity, splits in two for one iteration and then forms an interval again. In particular, the AIET whose suspension is the Disco surface is not obtained from a blowing up an orbit of a minimal IET. 

\begin{figure}[H]
\centering
\includegraphics[width=0.88\textwidth]{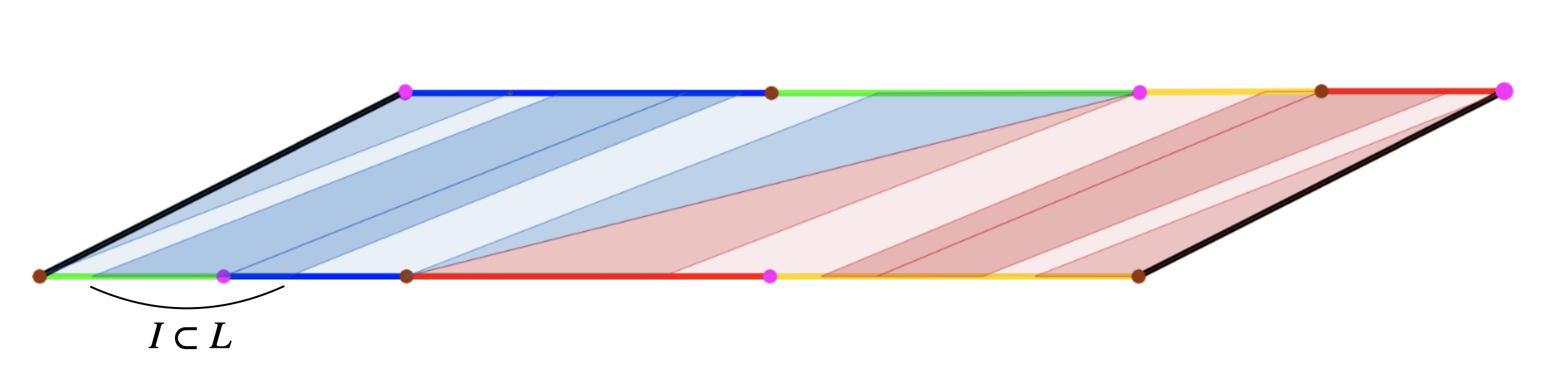}
\caption{The first iteration.}
\end{figure}

Note that because the "holes" of the Cantor set in the blue subsurface are obtained as the forward iterates of the interval $L$ and the length of these iterates goes to zero, the $\omega$-limit set of any leaf is equal to the Cantor set in the blue subsurface. Similarily, the $\alpha$-limit set of any leaf is equal to the Cantor set in the red subsurface. Hence, one Cantor set is attracting all of the leaves whereas the other Cantor set is repelling all of the leaves. In particular, we claim that both Cantor sets satisfy the definition of an \textit{attracting set}, respectively \textit{repelling set}, for the first return map $f$ on $L$. 

\begin{defs} Let $X$ be a topological space, let $g: X \rightarrow X$ be invertible. We say that $A \subset X$ is an \textit{attracting set for $g$} if it is compact, invariant and there exists a set $N$ such that $A \subsetneq N$ and $A = \bigcap_{n \in \mathbb{N}} g^n(N)$. A \textit{repelling set} for $g$ is an attractor for $g^{-1}$. 
\end{defs}

This definition of an \textit{attracting set} is satisfied if we choose $X=L,N=L, g=f$ and $A$ to be the Cantor set in the blue subsurface intersected with $L$. Similarly, the definition of a \textit{repelling set} is satisfied for the Cantor set in the red subsurface. This is a stark contrast to the Cantor set obtained in the first case of blowing up an orbit of a minimal IET. In the first case, the Cantor set does not satisfy the definition of an attracting or repelling set as any nearby leaf is always attracted and repelled at the same time. An interesting question to ask would be whether we can find more explicit examples of Cantor sets that are non-trivially recurrent orbit closures and to understand when these Cantor sets are attracting sets, repelling sets or neither.






.

\section{Concluding Remarks}

\subsection{Short Summary} 
In this thesis, we proved a decomposition theorem for the structure of directional foliations on dilation surfaces. After defining dilation surfaces and their directional foliations, we gave concrete examples on the different foliation structures that can arise on dilation surfaces. In particular, we introduced the example of the Disco surface and discussed the directions for which the corresponding foliation exhibits non-trivial recurrence. We then introduced Gardiner's decomposition theorem and showed, using the theorem, that we can decompose any dilation surface into subsurfaces on which its directional foliation is either completely periodic, Morse-Smale, minimal or Cantor-like. We further showed that if the foliation is minimal or Cantor-like, then the first return map on any transversal segment is semi-conjugated to a minimal IET. We then applied the decomposition theorem to the dilation surface obtained from the suspension of an affine interval exchange transformation to obtain the equivalent result for AIETs. 

\subsection{Further questions} 
There are a number of questions that could be interesting investigate further. Firstly, while we have shown that the directional foliation on a dilation surface is locally always in one of four cases, we would like to understand more about the set of directions on $S^1$ for which these different cases arise. Recall that in the case of the Disco surface the authors of \cite{cascades} have shown that the directional foliation is Morse-Smale for almost all $\theta \in S^1$ and for a measure zero set of directions it is either completely periodic, minimal or Cantor-like (where it is Cantor-like for a Cantor set of directions). One would like to conclude similar results for any dilation surface, in particular one would like to prove Selim Ghazouani's conjecture that for almost all $\theta \in S^1$, the directional foliation on a dilation surface that is not a translation surface is Morse-Smale (see Conjecture \ref{selimconjecture}). Furthermore, while the set of directions with Cantor-like behaviour has measure zero in all the examples we know, it would be interesting to examine its Hausdorff dimension. 
\vspace{2mm}

\textbf{Q1)} Can we prove Selim Ghazouani's conjecture that for almost all $\theta \in S^1$, the directional foliation on a dilation surface that is not a translation surface is Morse-Smale? What can we say about the Hausdorff dimension for sets of directions that satisfy Cantor-like behaviour?\vspace{2mm}

We note however that the techniques needed to answer such questions will likely be very different and much more involved than the techniques used in this thesis. A more realistic objective would be to construct more examples in the spirit of the Disco surface, meaning dilation surfaces that exhibit Cantor-like behaviour, but of higher genus. So far, we have only seen examples where the corresponding Cantor set is fully contained in a genus one subsurface. 
 \vspace{2mm}

\textbf{Q2)} Can we construct explicit examples of dilation surfaces that exhibit Cantor-like behaviour where the subsurfaces that contain the Cantor set have genus strictly greater than one? \vspace{2mm}

One way to solve this could be to explicitely find an AIET whose suspension is a surface of higher genus and that has been obtained from blowing up an orbit of a minimal IET, perhaps using ideas contained in \cite{MarmiMoussaYoccoz}, such that the corresponding Cantor-sets are not contained in genus one subsurfaces. 

The last question we want to propose is related to the last section of the previous chapter in which we were discussing the different types of Cantor sets that can arise for directional foliations that exhibit Cantor-like behaviour. Ideally, we would like to come up with results that allow us to further separate case 4 into subcases according to the nature of the corresponding Cantor set. 
 \vspace{2mm}

\textbf{Q3)} For the directional foliation on a dilation surface that satisfies case 4 of our main theorem, when is the corresponding Cantor set an attracting set, repelling set or neither? \vspace{2mm}

We would like to end with this final question. We hope that during the course of the thesis, we have provided the reader with an insightful exposition to dilation surfaces and that the rise in their popularity might soon yield more answers to the questions that remain open.

\section{Appendix}

\subsection{Rauzy-Veech-Induction for the Disco surface} In  \cite{cascades}, the authors provide a full study of the directional foliation on the Disco surface. In particular, the authors showed that there is a Cantor set of directions $\theta$ in $S^1$ for which the corresponding word of the Rauzy-Veech induction for the first return map $f$ on $L$ is infinite and not constant and then deduced that in these directions, the foliation accumulates on a Cantor set. We show in this section that these directions are exactly the ones that satisfy the assumptions of Proposition \ref{prop_disco_surface}. To draw this connection, we first give a brief overview of the Rauzy-Veech for the Disco Surface with respect to the first return map $f:L \rightarrow L$, for more detail please refer to \cite{cascades}. 

\begin{figure}[H]
\centering
\includegraphics[width=0.9\textwidth]{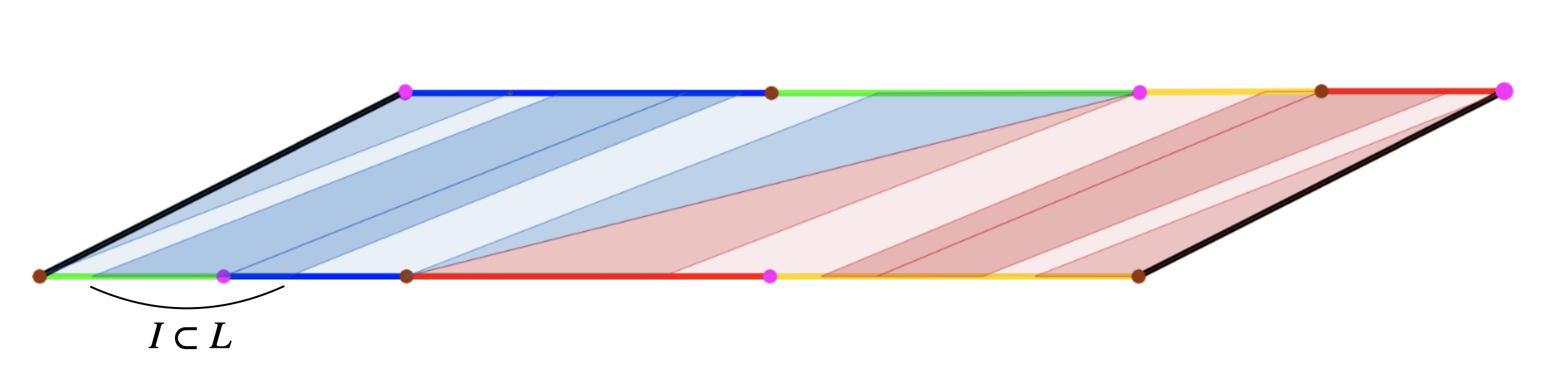}
\caption{}
\end{figure}

For $\theta \in S^1$, for the directional foliation $\mathcal{F}_{\theta}$ on the Disco surface, consider the corresponding first return map $f:L \rightarrow L$. The idea behind the Rauzy-Veech algorithm is to smartly chose smaller and smaller subintervals of $L$ on which $f$ is well-defined. The algorithm stops as soon as we reach a subinterval that gets mapped strictly into itself. There are three possible cases for $f$, we say that $f$ is in case 1 if it is of the following form: 

\begin{figure}[H]
\centering
\includegraphics[width=0.5\textwidth]{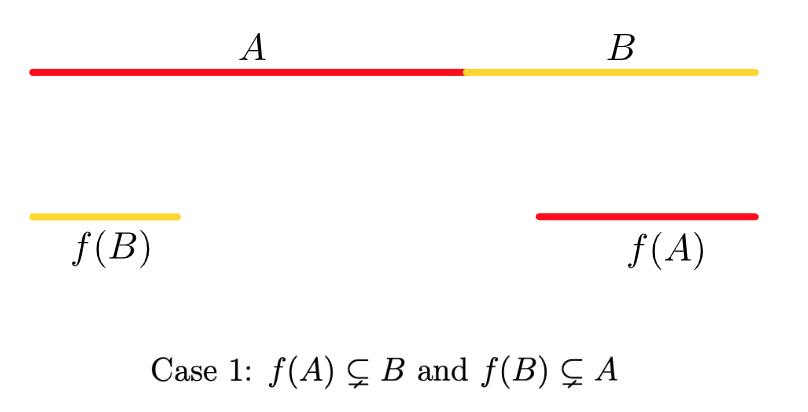}
\labelformat{empty}
\caption{Case 1: $f(A) \subsetneq B$ and $f(B) \subsetneq A$}
\end{figure}

For $f$ in case 1, consider the first return map $f$ on $L-f(A)$. It holds $f(L-f(A)) \subsetneq L-f(A)$ and hence $f$ is a contraction and has an attractive fixed point on $L-f(A)$. This then implies that $\mathcal{F}_{\theta}$ has a closed leaf. We say that $f$ on $L$ is in case 2 if satisfies one of the following two subcases: 

\begin{figure}[h]
\centering
\includegraphics[width=1.1\textwidth]{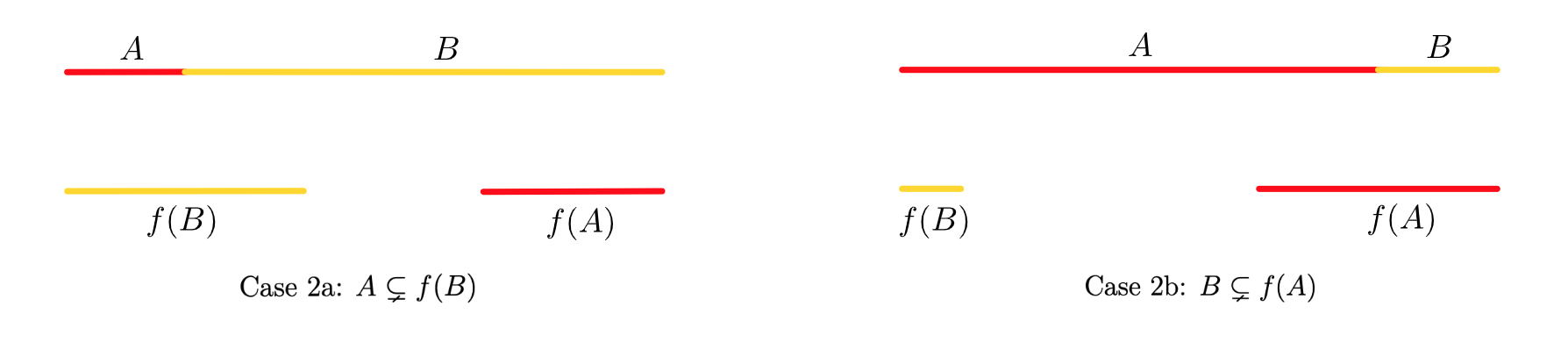}
\labelformat{empty}
\caption{}
\end{figure}

Rauzy-Veech induction follows three steps: 

\begin{enumerate}
    \item (Left Rauzy Veech Induction) If $f$ on $L$ is in case 2a, then consider the first return map on $L-A$. Repeat the loop.
    \item (Right Rauzy Veech Induction) If $f$ on $L$ is in case 2b, then consider the first return map on $L-B$. Repeat the loop.
    \item If $f$ on $L$ is in case 1, the algorithm terminates. 
\end{enumerate}
Note that the induction keeps the left endpoint of $f(A)$ and the right endpoint of $f(B)$ unchanged. Furthermore, the first return map on the subintervals $L-A$, respectively $L-B$, is again piecewise continuous  on two intervals. Here, the point of discontinuity is exactly $f^{-1}(p)$ where $p$ is the point of discontinuity of $f$ on $L$, i.e the map that we started with. Moreover, if we perform right Rauzy-Veech induction, then the right interval of $L-A$ is contracted by a factor $\frac{1}{2^2}$, as it is first sent to a subinterval of $f(A)$ and then to $f(B)$, whereas the left intervals is still contracted by $\frac{1}{2}$ (and analogous for the left Rauzy-Veech induction). So after a finite number of induction steps we consider $f$ on a subinterval of $L$ with one point of discontinuity, where the left and the right interval $A'$ and $B'$ are contracted by powers of $\frac{1}{2}$ and the left (resp. right) endpoint of $f(A')$ (resp. $f(B')$) is the left endpoint of $f(A)$ (resp. $f(B)$). 

To any first return map $f$ on $L$ we can associate a word in the alphabet $\{L,R \}$ by applying Rauzy-Veech induction and recording $L$ (resp. $R$) whenever we use left (resp. right) Rauzy Veech Induction. There are three cases that might occur: 
\begin{itemize}
    \item The word is finite. In this case the foliation $\mathcal{F}_{\theta}$ has an attracting closed leaf and is Morse-Smale. The authors of \cite{cascades} showed that this is the case for full measure set of directions $\theta \in S^1$ and that the complement of this set is a Cantor set.\vspace{2mm}
    
    \item The word is infinite but eventually constant, meaning it has an infinite tail that consists only of the letter "L" (or "R").  In this case, the top left (or right) interval of the maps obtained by Rauzy-Veech induction is multiplied each time by a positive power of $2$, however the total length of $L$ is bounded, hence this is the case only when one of the top intervals has length zero. In this case, the foliation accumulates on a saddle connection. There are only countably many directions in $S^1$ for which this happens. \vspace{2mm} 
    
    

    \item  The word is infinite but not eventually constant. The authors of \cite{cascades} showed that this is true for a Cantor set of directions $\theta \in S^1$ and that in this case, the foliation accumulates on a Cantor set. 
\end{itemize}

\begin{prop} The set of directions for which the word is infinite but not constant are exactly the set of directions which satisfy the assumptions of Proposition \ref{prop_disco_surface}.
\end{prop}

\begin{proof} Note that for the intersection $I$ between the dark blue strip from Figure 32 and $L$ (where the strip is taken to be open) it holds that $I=L-f(A)-f(B)$, i.e the "hole" of the first return map $f$ on $L$. We claim that Rauzy-Veech induction is infinite if and only if the strip, when we iterate it further, never "hits" a singularity. Indeed, as mentioned before, the induction keeps the right end of $f(B)$ and of $f(A)$ unchanged. Let $p \in L$ be the point of discontinuity of $f$ on $L$. To say that for up to all $n>0$, at the $n$th induction step the resulting map is always in case 2a) or 2b) means exactly that $f^{-k}(p) \cap I_{\infty} = \emptyset$ for all $k > 1$, which is the case if and only if  $p \cap f^k(I_{\infty}) = \emptyset$ for all $k > 1$. \end{proof}

Hence there exists a Cantor set of directions $\theta \in S^1$ that satisfies the assumptions of Proposition \ref{prop_disco_surface}. 

\subsection{Proof of key lemma} In this section, we want to include the proof of Lemma \ref{keylemma} due to Carlos Gutierrez (see \cite{gutierrez}) that is crucial for the proof of our main theorem. The lemma easily follows from the following proposition whose statement and proof was originally established by C. Gutierrez in \cite{gutierrezkeylemma}: 

\begin{prop}\label{prop8.2}
    Let $C$ be a circle and $T:C \rightarrow C$ be a continuous injective map defined everywhere except possibly at finitely many points $z_1,\dots,z_n \subset C$. If $T$ has dense positive semi-orbit, then $T$ is topologically conjugate to a standard interval exchange transformation.
\end{prop}
\begin{defs} Let $(p_i)_{i\in \mathbb{N}}$ be the dense positive semi-orbit. For $0<\beta<1$ we define an atomic measure $\mu_{\beta}: C \rightarrow \mathbb{R}$ in the following way: 
$$ \mu_{\beta}(p_i) := \beta(1- \beta)^{i-1} $$ and if $A \subset C$ and $P = A \cap \{ p_1,\dots,p_n,\dots \}$ then $$\mu_{\beta}(A) = \sum_{p\in P} \mu_{\beta}(p).$$ \end{defs}

Let $\theta$ be a fixed orientation in $C$, let $a,b \in C$ such that $a \neq b$. We define the interval $(a,b) = \{z \in C - \{a \} \hspace{0.1cm}|\hspace{0.1cm} z < b \}$ where $<$ is the linear order induced by the orientation $\theta$ in $\Gamma - \{ a \}$. The notation $a<c<b$ means that $c \in (a,b)$. We define the same linear order for $\mathbb{R} / \mathbb{Z}$
\begin{lem} Let $a,b \in C, a \neq b$. Then inf $\{\mu_{\beta}((a,b)) \hspace{0.1cm}| \hspace{0.1cm} 0 < \beta \leq \frac{1}{2}\} > 0$
\end{lem}
\begin{proof}
    Let $x,y \in (a,b)$ such that $a < x < y < b$ and 
    \begin{align}
    \{x,y\} \cap \{p_1,p_2, \dots , p_n, \dots \} = \emptyset.
    \end{align}
Certainly, for all $\beta \in (0,1/2]$, \hspace{0.1cm}$\mu_{\beta}((x,y)) < \mu_{\beta}((a,b))$. Hence, we only have to prove that inf$\{ \mu_{\beta}((x,y)) \hspace{0.1cm}|\hspace{0.1cm} 0 < \beta \leq \frac{1}{2}\} > 0$. Let $\sigma_1,\sigma_2,\dots,\sigma_{n-1} \in (x,y)$ where $x = \sigma_0 < \sigma_1 < \dots \sigma_n = y$ such that the following is satisfied: \vspace{0.2cm}
\begin{equation}
\begin{split}
 &\text{Given } p \in (x,y),\text{ then we have that} \hspace{0.1cm} p \in \{\sigma_1,\sigma_2,\dots,\sigma_{n-1} \} \text{ if and only if} \\& 
 \text{there exists } k \in \mathbb{N} \text{ such that }T^k(p) \in \{ \text{singular points of $T$}\} \cup \{x\} \cup \{y\} \\& \cup \{p_1 \}  \text{ and for } 
 \text{all k' $<$ k }, T^{k'}(p) \notin (x,y)-\{p_1 \}.
 \end{split}
 \end{equation}
 Note that there are only finitely many points that satisfy (2) as $T$ only has finitely many singular points. Next consider the first return map $S: (x,y) \rightarrow (x,y)$ induced by $T$.
 Then because $T$ is continuous everywhere except at its singular points, $S$ is well-defined and continuous at each $(\sigma_i,\sigma_{i+1})$ for all $i \in \{0,1,\dots,n-1\}$. Therefore the following holds:
\begin{equation}
\begin{split}
& \forall i \in \{ 0,1,\dots,n-1\} \text{ there exists } l_i \in \mathbb{N} \text{ such that } \forall s \in \{1,2, \dots , l_i \} \\& \text{it holds that } T^s((\sigma_i,\sigma_{i+1})) \cap (x,y) = \emptyset, \text{ but } T^{l_i+1}((\sigma_i,\sigma_{i+1})) \subset (x,y) \\& \text{(this implies that }S|_{(\sigma_i,\sigma_{i+1})} = T^{l_i+1}|_{(\sigma_i,\sigma_{i+1})} ).
\end{split}
\end{equation}

Moreover, since $(p_i)_{i\in \mathbb{N}}$ is dense in $C$ we have that
\begin{equation} 
\begin{split}
\exists \hspace{0.1cm} l_n \in \mathbb{N} \text{ such that } \{p_1,p_2, \dots , p_{l_n} \} \cap (x,y) = \emptyset \text{ but } p_{l_n+1} \in (x,y).
\end{split}
\end{equation}

If $p_k \in (\sigma_i,\sigma_{i+1})$ for some $k \in \mathbb{N}$ and $i \in \{0,1, \dots , n-1 \}$, then, by (3), $p_{k+l_i+1} \in (x,y)$ which implies by (1) and (3) that $p_{k+l_i+1} \in (\sigma_j, \sigma_{j+1})$ for some $j \in \{ 0,1, \dots, n-1 \}$. Thus, using (3) and (4) we obtain that $\forall \tilde{s} \in \mathbb{N}$, 
\begin{equation}
    p_{l_n+1+\tilde{s} } \in \bigcup_{i=0}^{n-1} \bigcup_{s=0}^{l_i} T^s((\sigma_i,\sigma_{i+1})).
\end{equation}
But since $\{ p_{l_n+1+\tilde{s}} \}_{ \tilde{s} \in \mathbb{N}}$ is dense in $C$, we conclude that 
\begin{equation}
  \overline{\bigcup_{i=0}^{n-1} \bigcup_{s=0}^{l_i} T^s((\sigma_i,\sigma_{i+1}))} = C.
\end{equation}
Next, we claim that $\forall i \in \{0,1,\dots,n-1 \}$ and $\forall s \in \{0,1, \dots , l_i \}$,
\begin{equation} \mu_{\beta}(T^s(\sigma_i,\sigma_{1+1})) = (1-\beta)^s \mu_{\beta}((\sigma_i,\sigma_{i+1})). 
\end{equation}
In fact, this follows from 
\begin{align*}
\mu_{\beta}(T((\sigma_i,\sigma_{i+1}))) &=\sum_{p_j \in (\sigma_i,\sigma_{i+1})} \mu_{\beta}(p_{j+1}) = (1-\beta) \sum_{p_j \in \sigma_i,\sigma_{i+1})} \mu_{\beta}(p_j) \\&= (1-\beta)\mu_{\beta}(\sigma_i,\sigma_{i+1}).\end{align*}
We use (4), (5), (6) and (7) to conclude that: 
\begin{equation} \mu_{\beta}(C) = \sum_{i=0}^{n-1} \mu_{\beta}((\sigma_i,\sigma_{i+1}))(1+(1-\beta)+ \dots + (1-\beta)^{l_i}) + \sum_{i=1}^{l_n}\mu_{\beta}(p_i).
\end{equation}
If we assume that there is a sequence $\beta_1$, $\beta_2, \dots, \beta_j$ such that $\lim_{j \to \infty} \mu_{\beta_j}((x,y)) = 0$, we have that $\lim_{j \to \infty} \mu_{\beta_j}((\sigma_i,\sigma_{i+1})) = 0$. Therefore, by (8), $\lim_{j \to \infty} \mu_{\beta_j}(C) < 1$ which is a contradiction as $\mu_{\beta}(C) = 1$ for all $\beta \in (0,1/2]$. 
\end{proof}
Fix $\lambda_0 \in C-\{p_1,p_2,\dots,p_n,\dots\}$. Because the set $\{ \mu_{\beta}((\lambda_0,p_i)) \hspace{0.1cm}|\hspace{0.1cm} \beta \in (0,1/2] , \hspace{0.1cm} i \in \mathbb{N} \}$ is bounded, we can find a sequence $\{ \beta_j\}_{j \in \mathbb{N}}$ where $\beta_j \in (0,1/2]$ for all $j \in \mathbb{N}$, such that $\lim_{j \to \infty} \beta_j =0$ and, $\forall i \in \mathbb{N}$, $\lim_{j \to \infty} \mu_{\beta_j} ((\lambda_0,p_i)) =: h(p_i)$ exists. Given $x \in C-\lambda_0$, we define $h(x) = \sup \{ h(p_i) \hspace{0.1cm}| \hspace{0.1cm} \lambda_0 < p_i <x \}$. We also define $h(\lambda_0)=0$. 
\begin{lem}\label{lem8.5} The map $h: C \rightarrow \mathbb{R}/\mathbb{Z}$ is a homeomorphism and $\tilde{T} = h \circ T \circ h^{-1}$ is a standard interval exchange transformation.
\end{lem}
\begin{proof}
    By definition, $\lambda_0 < p_n < p_m$ implies that $0 < h(p_n) < h(p_m)$. Hence $h$ is monotonic and continuous from below. Suppose that $h$ is not continuous. Thus, there exists $x \in C$ and sequences $\{p_{n_j} \}, \{p_{m_j} \}, j \in \mathbb{N}$, such that
    \begin{equation}
        p_{n_1}<p_{n_2}<\dots<p_{n_j}<\dots < p_{m_j}< \dots  p_{m_2}<p_{m_1},
    \end{equation}
      \begin{equation}
       \lim_{j \to \infty} p_{n_j} = x = \lim_{j \to \infty} p_{m_j}, \text{and}
    \end{equation}
    \begin{equation}
        \lim_{j \to \infty} (h(p_{m_j}))- h(p_{n_j} )) > \delta. 
    \end{equation}
We only consider the case where $x$ is such that there does not exist a $ k \in \mathbb{Z}$ with $T^k(x) \in \{z_1,\dots,z_n\}$. Let $N$ be a positive integer satisfying $N \geq 2/\delta$. It follows from (9) and (10) that there exists $n \in \{n_1,n_2, \dots \}$ and $m \in \{m_1,m_2,\dots \}$ such that $\{ (p_n,p_m), \hspace{0.1cm} T((p_n,p_m)), \dots, \hspace{0.1cm} T^N((p_n,p_m)) \}$ are pairwise disjoint intervals contained in the domain of $T$. By (7), (11) and the fact that $\lim_{j \to \infty} \beta_j = 0$ we observe that $\forall s \in \{0,1, \dots , N\}$,
\begin{align*} \lim_{i \to \infty} \mu_{\beta_i}(T^S((p_n,p_m)))=\lim_{i \to \infty} (1-\beta_i)^s \mu_{\beta_i} ((p_n,p_m))> \delta.
\end{align*}
Therefore, 
\begin{align*}
\lim_{i \to \infty} \mu_{\beta_i} ( \bigcup_{s=0}^N T^s ((p_n,p_m))) \geq N\delta > 2
\end{align*}
This is a contradiction, because $\lim_{i \to \infty} \mu_{\beta_i}(C) = 1$. Consequently, $h$ is a homeomorphism. 

Let $(a,c)$ be an interval contained in the domain of definition of $T$. We claim that 
\begin{equation} |h(a)-h(c)| = |hT(a)-hT(c)|.
\end{equation}
Since $\{p_k \}$ is dense in $C$ and $h$ and $T$ are continuous, we only have to prove (12) when $a=p_i$ and $c=p_j$, for some $i,j \in \mathbb{N}, i \neq j$. Notice that $\lambda_0 < p_i < p_j$. We only consider the case $T(p_j) < T(p_i) < T(\lambda_0)$. Then

\begin{equation*}
    \begin{split}
        |h(p_i)-h(p_j)| &= |\lim_{k \to \infty} \mu_{\beta_k}((\lambda_0,p_i))- \lim_{k \to \infty} \mu_{\beta_k}((\lambda_0,p_j))| \\
        &= \lim_{k \to \infty} \mu_{\beta_k}((p_i,p_j)) \\
        & \stackrel{(7)}{=} \lim_{k \to \infty} \frac{1}{1-\beta_k} \cdot \lim_{k \to \infty} \mu_{\beta_k}((T(p_j),T(p_i))) \\
        &= |hT(p_i)-hT(p_j)|.         
    \end{split}
\end{equation*}
If $(\tilde{a},\tilde{c})$ is an interval of the domain of definition of $\tilde{T} = h \circ T \circ h^{-1}$, then, by (12), $|\tilde{T}(\tilde{a})-\tilde{T}(\tilde{c})| = |\tilde{a}-\tilde{c}|.$ This proves Lemma \ref{lem8.5} as well as Proposition \ref{prop8.2}. 
\end{proof}

\printbibliography

@article{gardiner,
title   = "The structure of flows exhibiting nontrivial recurrence on two-dimensional manifolds",
author  = "C. J. Gardiner",
journal = "Journal of Differential Equations",
volume  = "57",
pages   = "138-158",
year    = "1985"
}

@article{keane,
title   = "Interval exchange transformations",
author  = "M. Keane",
journal = "Mathematische Zeitschrift",
volume  = "141",
number  = "",
pages   = "25-31",
year    = "1975"
}

@article{gutierrez,
title   = "Smoothing continuous flows on two-manifolds and recurrences",
author  = "C. Gutierrez",
journal = "Ergodic Theory Dynam. Systems",
volume  = "6",
number  = "3",
pages   = "17-44",
year    = "1986"
}

@article{gutierrezkeylemma,
title   = "Smoothability of Cherry flows on two-manifolds",
author  = "C. Gutierrez",
journal = "Springer Lecture Notes in Mathematics, Geometric Dynamics, Proc. Rio de Janeiro",
volume  = "1007",
pages   = "308-331",
year    = "1981"
}

@article{cascades,
title   = "Cascades in the dynamics of affine interval exchange transformations",
author = {{Boulanger}, A. and {Fougeron}, C. and {Ghazouani}, S.},
journal = "Ergodic Theory Dynam. Systems",
volume  = "40",
number  = "8",
pages   = "2073-2097",
year    = "2020"
}

@article{MarmiMoussaYoccoz,
title   = "Affine interval exchange maps with a wandering interval",
author = {{Marmi}, S. and {Moussa}, P. and {Yoccoz}, J.C.},
journal = "Proc. Lond. Math. Soc.",
volume  = "100",
number  = "3",
pages   = "639-669",
year    = "2010"
}

@article{ghazouani,
title   = "Une invitation aux surfaces de dilatation",
author  = "S. Ghazouani",
journal = "Séminaire de théorie spectrale et géométrie",
volume  = "35",
pages   = "69-107",
year    = "2019"
}

@article{aprioribounds,
       author = {{Ghazouani}, S. and {Ulcigrai}, C.},
        title = "{A priori bounds for GIETs, affine shadows and rigidity of foliations in genus 2}",
      journal = {arXiv e-prints},
     keywords = {Mathematics - Dynamical Systems},
         year = 2021,
        month = jun,
          %eid = {arXiv:2106.03529},
        %pages = {arXiv:2106.03529},
          %doi = {10.48550/arXiv.2106.03529},
archivePrefix = {arXiv},
       eprint = {2106.03529},
 primaryClass = {math.DS},
       adsurl = {https://ui.adsabs.harvard.edu/abs/2021arXiv210603529G},
      adsnote = {Provided by the SAO/NASA Astrophysics Data System}
}

@book{Camacho,
title     = "Geometric Theory of Foliations",
author    = {{Camacho}, C. and {Neto}, A.L.},
edition   = "1",
publisher = "Birkhäuser Boston, MA",
year      = "1984",
}

@ARTICLE{dilationtori,
       author = {{Haberle}, M. and {Wang}, J.},
        title = "{A Full Study of the Dynamics on One-Holed Dilation Tori}",
      journal = {arXiv e-prints},
     keywords = {Mathematics - Dynamical Systems, 37E35},
         year = 2020,
        month = dec,
          %eid = {arXiv:2012.04159},
        %pages = {arXiv:2012.04159},
          %doi = {10.48550/arXiv.2012.04159},
archivePrefix = {arXiv},
       eprint = {2012.04159},
 primaryClass = {math.DS},
       adsurl = {https://ui.adsabs.harvard.edu/abs/2020arXiv201204159H},
      adsnote = {Provided by the SAO/NASA Astrophysics Data System}
}

@ARTICLE{affinesurfacesandveechgroups,
       author = {{Duryev}, Eduard and {Fougeron}, Charles and {Ghazouani}, Selim},
        title = "{Affine surfaces and their Veech groups}",
      journal = {arXiv e-prints},
     keywords = {Mathematics - Geometric Topology, Mathematics - Dynamical Systems},
         year = 2016,
        month = sep,
          %eid = {arXiv:1609.02130},
        %pages = {arXiv:1609.02130},
          %doi = {10.48550/arXiv.1609.02130},
archivePrefix = {arXiv},
       eprint = {1609.02130},
 primaryClass = {math.GT},
       adsurl = {https://ui.adsabs.harvard.edu/abs/2016arXiv160902130D},
      adsnote = {Provided by the SAO/NASA Astrophysics Data System}
}

\end{document}